\documentclass{article}

\usepackage{verbatim}

\usepackage{amsfonts}

\usepackage{amsthm}

\usepackage{amssymb}

\usepackage{amsmath}

\usepackage{mathabx}

\usepackage{enumerate}

\usepackage[all]{xy}

\usepackage{graphicx}

\usepackage{centernot}

\usepackage{mathtools} 

\usepackage{stmaryrd}

\usepackage[pagebackref,colorlinks,linkcolor=blue,citecolor=blue,urlcolor=blue,hypertexnames=true]{hyperref}

\usepackage[pagebackref,hypertexnames=true]{hyperref}

\newcommand{\N}{\mathbb{N}}

\newcommand{\Z}{\mathbb{Z}}

\newcommand{\R}{\mathbb{R}}

\newcommand{\C}{\mathbb{C}}

\newcommand{\K}{\mathcal{K}}

\newcommand{\Hawaii}{Hawai\kern.05em`\kern.05em\relax i}

\newcommand{\emu}{\text{ucp}}

\newcommand{\lin}{\text{lin}}

\newcommand{\LL}{\mathcal{L}}

\theoremstyle{plain}
\newtheorem{theorem}{Theorem}[section]
\newtheorem{lemma}[theorem]{Lemma}
\newtheorem{corollary}[theorem]{Corollary}
\newtheorem{proposition}[theorem]{Proposition}

\newtheorem{definition-theorem}[theorem]{Definition / Theorem}

\newtheorem*{conjecture*}{Conjecture}
\newtheorem*{theorem*}{Theorem}

\theoremstyle{definition}
\newtheorem{definition}[theorem]{Definition}
\newtheorem{example}[theorem]{Example}

\theoremstyle{remark}
\newtheorem{remark}[theorem]{Remark}
\newtheorem{remarks}[theorem]{Remarks}

\newtheorem*{example*}{Example}  
\newtheorem*{remark*}{Remark}

\begin{document}
\title{An Atiyah-Segal homomorphism for approximate representations of infinite groups, and almost multiplicative ucp maps}
\author{Rufus Willett}

\maketitle

\begin{abstract}
The Atiyah-Segal homomorphim relates the representation ring of a finite group to the $K$-theory ring of its classifying space.  We develop a version of this for infinite groups, but based on a ring of approximate representations rather than representations.  We then use known results on the Baum-Connes conjecture to show that our homomorphism is an isomorphism for many groups: this includes all torsion-free amenable groups, and all torsion-free fundamental groups of compact manifolds of dimension at most three.  The main motivation is the desire to understand approximate representations using topological tools: as a sample application we give some examples of torsion phenomena in approximate representation theory.

The methods of proof are $C^*$-algebraic: the key idea is to relate our Atiyah-Segal type homomorphism to Kasparov's coassembly map.  The main step realizes (Hausdorffized) $KK$-theory as approximately multiplicative ucp maps modulo stable approximate unitary equivalence for appropriate $C^*$-algebras.  This is in the spirit of the stable uniqueness theorems of Dadarlat-Eilers and Lin, and may be of purely $C^*$-algebraic interest. 
\end{abstract}

\tableofcontents

\section{Introduction}

Quasi-representations of a group are maps from finite subsets of the group to finite-dimensional unitary groups that approximately preserve the group multiplication table (restricted to the finite subset).  The main goal of this paper is to build the quasi-representations of a group into a sort of $K$-theoretic quasi-representation ring, and then relate this to the $K$-theory of the classifying space of the group.  This is in the spirit of classical work of Atiyah \cite{Atiyah:1961uq} and Atiyah-Segal \cite{Atiyah:1969aa} on finite groups: they relate actual representations to the $K$-theory of the classifying space.  It is also very much inspired by recent work of Dadarlat and co-authors on the relationship between quasi-representations, quasidiagonality, and $K$-theory: see for example \cite{Carrion:2015ld,Dadarlat:2011kx,Dadarlat:2011uq,Dadarlat:384aa,Dadarlat:2026aa}.  This is discussed in more detail in Subsection \ref{qr vb ss} below.

The methods of proof involve $KK$-theory of $C^*$-algebras, and particularly Kasparov's coassembly map \cite{Kasparov:1975ht,Kasparov:1993hc,Kasparov:1988dw}.  Along the way, we develop some purely $C^*$-algebraic results in the spirit of the stable uniqueness theorems of Dadarlat-Eilers \cite{Dadarlat:2002aa} and Lin \cite{Lin:2002aa}: the main point is to define a model of $KK$-theory for unital and quasidiagonal $C^*$-algebras where cycles are (formal differences of) approximately multiplicative ucp maps.  This could be of independent interest, and is discussed in detail in Subsection \ref{c* ss} below (which is readable independent of the group-theoretic material).

\subsection{Quasi-representations and vector bundles}\label{qr vb ss}

Let $\Gamma$ be a finite group, and let $R(\Gamma)$ be the representation ring of $\Gamma$, i.e.\ the ring of formal linear combinations of isomorphism classes of finite-dimensional unitary representations of $\Gamma$, with operations induced by direct sum and tensor product.  Let $B\Gamma$ denote a CW complex model for the classifying space of $\Gamma$.  If $\Gamma$ is non-trivial (and finite), this is an infinite-dimensional complex: we define its $K$-theory $LK^0(B\Gamma)$ as the inverse limit of the $K$-theory groups $K^0(Y)$, as $Y$ varies over all compact subspaces of $B\Gamma$.  An influential result of Atiyah \cite{Atiyah:1961uq}, subsequently generalized and conceptualized by Atiyah-Segal \cite{Atiyah:1969aa}, establishes a ring isomorphism 
\begin{equation}\label{as iso}
\alpha:\widehat{R(\Gamma)}\to LK^0(B\Gamma)
\end{equation}
between a certain completion of $R(\Gamma)$ and $LK^0(B\Gamma)$.  On the level of cycles, the map $\alpha$ from line \eqref{as iso} is determined as follows: it takes the class of a unitary representation $\pi:\Gamma\to U(V)$ (with $V$ a finite-dimensional Hilbert space) to the class $[E_\pi]$ of the associated (flat) vector bundle over $B\Gamma$.  Precisely, $E_\pi$ is the `balanced product' \begin{equation}\label{ass flat}
E_\pi:=E\Gamma\times_\Gamma V,
\end{equation}
where $E\Gamma$ is the universal cover of $B\Gamma$ and the notation means we take the quotient of $E\Gamma\times V$ by the diagonal action given on $E\Gamma$ by deck transformations, and on $V$ by $\pi$.

The map $\alpha$ in line \eqref{as iso} also makes sense for infinite groups, but it fails badly to be an isomorphism.  However, there are several interesting variants for infinite groups that are at least closer to being isomorphisms: see for example \cite{Ramras:2023aa} and \cite[Section 8]{Kasparov:1975ht}.  One particularly notable such adaptation is due to Kasparov (see \cite[Section 8]{Kasparov:1975ht}, \cite[Section 9]{Kasparov:1993hc}, and \cite[Section 6]{Kasparov:1988dw}), who introduced a \emph{coassembly map}
\begin{equation}\label{kca}
\alpha:KK^\Gamma(\C,\C)\to RK^0(B\Gamma).
\end{equation}
Here cycles for the left hand side are pairs of  possibly infinite-dimensional unitary representations which have `compact' difference in some sense, and the right hand side is the representable $K$-theory of $B\Gamma$.  This is known to be an isomorphism in many cases\footnote{But not for finite groups! - for $\Gamma$ finite, the right hand side identifies with $LK^0(B\Gamma)$ in our notation, but the left hand side $KK^\Gamma(\C,\C)$ identifies with the uncompleted representation ring $R(\Gamma)$.} thanks to work of Kasparov and others on this and the related Baum-Connes conjecture \cite{Baum:1994pr}: perhaps most notably, the work of Higson-Kasparov \cite{Higson:2001eb} implies that the map in line \eqref{kca} is an isomorphism whenever $\Gamma$ is torsion-free and has the Haagerup property, and in particular for all torsion-free amenable $\Gamma$.

Our goal in this paper is to produce another generalization of the Atiyah-Segal map for infinite groups, but now based on \emph{quasi-representations}.  Roughly, a quasi-representation is a map from a finite subset of $\Gamma$ to a unitary group that approximately respects the multiplication table for $\Gamma$.  Precisely, for a finite subset $S$ of $\Gamma$ and $\epsilon>0$, an \emph{$(S,\epsilon)$-representation} is a map $\phi:S\to U(V)$ from $S$ to the unitary group $U(V)$ of a finite-dimensional Hilbert space $V$ such that 
$$
\|\phi(s)\phi(t)-\phi(st)\|<\epsilon
$$
whenever $s$, $t$ and $st$ are all in $S$ (the norm here, and throughout, is the operator norm).  We use ``quasi-representation'' as a generic term for an $(S,\epsilon)$-representation when we do not want to record the parameters. In the spirit of $K$-theory, we consider $(S,\epsilon)$-representations up to the following notion of equivalence: $(S,\epsilon)$-representations $\phi:S\to U(V)$ and $\psi:S\to U(W)$ are \emph{stably $(S,\epsilon)$-equivalent} if there exists an $(S,\epsilon)$-representation $\theta:S\to U(X)$ and a unitary isomorphism $u: V\oplus X\to W\oplus X$ such that 
$$
\|u(\phi(s)\oplus \theta(s))u^*-(\psi(s)\oplus \theta(s))\|<\epsilon
$$
for all $s\in S$.  This is an equivalence relation; the set of equivalence classes forms a commutative semigroup under the operation of direct sum, and we define $K_\epsilon(S)$ to be its Grothendieck group.

On the other hand, following ideas in \cite{Connes:1990vo,Carrion:2015ld,Hunger:2016wv,Kubota:2022aa} one can show that for any compact subset $Y$ of the classifying space $B\Gamma$ there exists $S$ and $\epsilon$ such that there is a well-defined homomorphism
$$
K_\epsilon(S)\to K^0(Y).
$$
It is defined using an approximate version of the balanced product construction in line \eqref{ass flat} above.  This construction depends on some choices, but taking the inverse limits of both sides (on the left as $S$ gets bigger and $\epsilon$ smaller, and on the right as $Y$ gets bigger), we get a well-defined ring homomorphism 
\begin{equation}\label{intro qr as}
\alpha:QR(\Gamma)\to LK^0(B\Gamma),
\end{equation}
where the left hand side is defined to be inverse limit of the groups $K_\epsilon(S)$, and should be thought of as a $K$-theoretic `approximate representation ring' with operations induced by direct sum and tensor product of quasi-representations.  The map in line \eqref{intro qr as} does not depend on any of the choices involved in its definition.   It `extends' the Atiyah-Segal map of line \eqref{as iso} in the sense that there is a canonical homomorphism $R(\Gamma)\to QR(\Gamma)$, and that if $[\pi]$ is in the image of this map, then $\alpha([\pi])$ is the class of the flat bundle $[E_\pi]$ in $K^0(B\Gamma)$.

Our main motivation for introducing the homomorphism in line \eqref{intro qr as} is to use it as a topological tool to study quasi-representations: see Section \ref{tor sec} for a sample application. If one knows $\alpha$ is surjective, there should also be applications to the Novikov conjecture: indeed, our approach is in some sense a $K$-theoretic formalization of ideas of Connes-Gromov-Moscovici from \cite{Connes:1990vo}.  We will not give any applications of the latter sort, but it seems interesting to develop this more fully (this would also necessitate introducing an `odd-dimensional' version of $QR(\Gamma)$; we do not do that here).

The following is perhaps our main theorem, although our main motivation is really to study phenomena like that in Theorem \ref{intro tor the} below. 

\begin{theorem}\label{intro main}
Let $\Gamma$ be a countable torsion-free group with the Haagerup property such that $C^*\Gamma$ is quasidiagonal and has Kirchberg's local lifting property.  Then the ring homomorphism 
$$
\alpha:QR(\Gamma)\to LK^0(B\Gamma)
$$
in line \eqref{intro qr as} is a homeomorphic ring isomorphism.
\end{theorem}

We will not go into details on what the hypotheses mean here, but let us at least say that the theorem applies in particular to all countable torsion-free amenable groups.  Indeed, amenable groups (if characterized appropriately) are easily seen to have the Haagerup property.  Moreover, if $\Gamma$ is amenable then $C^*\Gamma$ is nuclear by the main results of \cite{Hulanicki:1964aa} and \cite{Lance:1973aa}, whence it is quasidiagonal by the main result of \cite{Tikuisis:2015kx}, and has the (local) lifting property by the main result of \cite{Choi:1976aa}.  The theorem also applies to other interesting classes of groups such as torsion-free fundamental groups of compact $2$- and $3$-manifolds, and many (possibly all) torsion-free one-relator groups: the necessary group-theoretic inputs for this are established in \cite{Fournier-Facio:2026aa}.  

\begin{remark}
There are some recent results of other authors that are closely related to Theorem \ref{intro main}; all were arrived at independently using different proofs, and all emphasize somewhat different classes of examples\footnote{For me, I have been particularly interested in $K$-theoretic torsion phenomena in this context (see \cite[Section 6]{Willett:2024aa} and \cite{Willett:2026aa}) and was motivated largely by proving analogues of Theorem \ref{intro tor the} below.}.  First, in \cite[Theorems 1.2 and 1.3]{Dadarlat:2026aa} Dadarlat and Glebe establish results that are close to being equivalent to Theorem \ref{intro main} for torsion-free amenable groups, although stated much more concretely and without packaging everything into a ring $QR(\Gamma)$ as we do.  The methods used by Dadarlat-Glebe are also based on comparison with Kasparov's coassembly map, although they go through Dadarlat's work on stable uniqueness and the topology on $KK$-theory \cite[Section 3]{Dadarlat:2005aa}, whereas ours go through `controlled $KK$-theory' as developed by the author and Yu \cite{Willett:2020aa}.  Our methods have the advantage that they apply to some non-amenable groups, but the work of Dadarlat-Glebe gives a much shorter approach in the amenable case.

Second, Weinberger, Wu, and Yu \cite[Theorem 1.1]{Weinberger:2026aa} establish essentially the same result as our Theorem \ref{intro main}, but in more generality: these authors also allow the presence of non-trivial torsion elements in $\Gamma$, which necessitates replacing $K^0(B\Gamma)$ with the equivariant $K$-theory of the classifying space for proper actions in the spirit of the Baum-Connes conjecture \cite{Baum:1994pr}.  On the other hand, the methods of Weinberger, Wu, and Yu rely more on equivariant $KK$-theory and explicitly use the Higson-Kasparov machine \cite{Higson:2001eb} (we just black-box it), and so are relatively specific to groups: the current paper contains an approach that gives a new description of $KK$- (or $KL$-)theory for a fairly general class of $C^*$-algebras (see Subsection \ref{c* ss} below), which we hope will be useful in other contexts.  
\end{remark}

Let us describe an application of Theorem \ref{intro main} (more precisely, of a slightly elaborated version of it that applies when the group $C^*$-algebra is residually finite-dimensional: see Section \ref{rfd sec} below).  This gives examples of new torsion phenomena in the theory of quasi-representations and specifically the sort of `weak stability' considered by Dadarlat in \cite{Dadarlat:384aa} and the author in \cite{Willett:2024aa}.  Being able to construct and understand non-trivial torsion examples like this was one of our main motivations for proving Theorem \ref{intro main}.

\begin{theorem}\label{intro tor the}
Let $\Gamma$ be the fundamental group of the product $T\times K$, where $T$ is the $2$-torus, and $K$ is the Klein bottle.  Then there exists a finite subset $S_0$ of $\Gamma$ and $\epsilon_0>0$ such that for any finite subset $S$ of $\Gamma$ containing $S_0$, and any $\epsilon\in (0,\epsilon_0$), there exists an $(S,\epsilon)$-representation $\phi:S\to U(V)$ of $\Gamma$ such that:
\begin{enumerate}[(i)]
\item \label{far part} for any finite-dimensional unitary representations $\pi:\Gamma\to U(W)$ and $\sigma:\Gamma\to U(V\oplus W)$ of $\Gamma$,
$$
\|\phi(s)\oplus \pi(s)-\sigma(s)\|\geq \epsilon_0
$$
for all $s\in S_0$;
\item \label{close part} there exist finite-dimensional representations $\pi:\Gamma\to U(W)$ and $\sigma:\Gamma\to U(V\oplus V\oplus W)$ of $\Gamma$ such that 
$$
\|\phi(s)\oplus \phi(s)\oplus \pi(s)-\sigma(s)\|< \epsilon
$$
for all $s\in S$.
\end{enumerate}
\end{theorem}

One can write down an explicit formula for a quasi-representation $\phi$ with the properties in the theorem: see Remark \ref{exp phi} below.  For this and related computations, it seems quite useful that Theorem \ref{intro main} gives a \emph{ring} isomorphism, not `just' an isomorphism of groups.

The key topological ingredient in the proof of Theorem \ref{intro tor the} is the construction of torsion classes in $LK^0(B\Gamma)$ (which equals $K^0(B\Gamma)$ in this case) that are not in the subgroup generated by flat bundles; this is elementary, but the existence of such classes seems not to be well-known and may be of interest in its own right.  Similar examples (also based on fundamental groups of closed flat manifolds) show the existence of $p$-torsion phenomena for `quasi-representations modulo representations' for any prime $p$.

Let us conclude this subsection by remarking that the methods we use to show Theorem \ref{intro main} rely on identifying the map of line \eqref{intro qr as} with (a slight variant) of the map of line \eqref{kca}.  As such, they cannot possibly give an isomorphism for any group for which the map of line \eqref{kca} is known to fail to be an isomorphism, and in particular for any infinite property (T) group.  However, we do not see an a priori obstruction to the map of line \eqref{intro qr as} being an isomorphism for some property (T) groups, and it would be interesting to develop proof methods that work in that case.  One such approach might be to show that $QR(\Gamma)$ is \emph{always} isomorphic to some sort of `almost flat' $K$-theory group of $B\Gamma$ along the lines of \cite{Carrion:2015ld,Hunger:2016wv,Kubota:2022aa}, and then try to directly relate this almost flat $K$-theory group to $K^0(B\Gamma)$.  We do not pursue this here, but it seems an interesting avenue for future work.

\subsection{$C^*$-algebra techniques}\label{c* ss}

The techniques we use to show Theorem \ref{intro main} rely on a long detour through $KK$-theory of $C^*$-algebras.  The key point is that the assumptions that $C^*\Gamma$ is quasidiagonal and has the local lifting property allow us to identity the map in line \eqref{intro qr as} with the `Hausdorffification' of the map in line \eqref{kca}; we can then appeal to the deep result \cite{Higson:2001eb} that the latter is an isomorphism when $\Gamma$ is torsion-free, and has the Haagerup property.

The key technical point is as follows.  Let $A$ and $B$ be separable unital $C^*$-algebras, and let $KL(A,B)$ denote R\o{}rdam's $KL$-group \cite[Section 5]{Rordam:1995aa}.  Following Dadarlat \cite[Section 5]{Dadarlat:2005aa} and Lin \cite{Lin:2005aa}, we consider $KL(A,B)$ as the maximal Hausdorff quotient of Kasparov's $KK$-group $KK(A,B)$; unlike R\o{}rdam's original approach, this definition makes sense for general separable $C^*$-algebras.  Assume that $A$ is quasidiagonal, and that $B$ is nuclear.  Then we show that $KL(A,B)$ can be realized as an inverse limit of groups of approximately multiplicative completely positive maps from $A$ into matrix algebras over $B$, modulo stable approximate unitary equivalence.  This should be viewed as an analogue of the stable uniqueness theorems of Dadarlat-Eilers \cite{Dadarlat:2002aa} and Lin \cite{Lin:2002aa}, although we will not pursue that analogy here.  If in addition $A$ satisfies the local lifting property and we assume some additional structure (roughly, a suitably well-behaved generating set), we may remove the condition that the maps need to be completely positive: applying all of this to the case $A=C^*\Gamma$ and $B=\C$ gives the desired isomorphism $QR(\Gamma)\cong KL(C^*\Gamma,
\C)$ that underlies our proof of Theorem \ref{intro main}.

To give a more precise idea of all this without getting too weighed down in technicalities, let us discuss some precise statements in the case that $B=\C$, so we are just dealing with $K$-homology rather than $KK$-theory; let us also ignore the role of the local lifting property for now.  Let $X$ be a finite subset of the unit ball of $A$.  Then a unital completely positive map $\phi:A\to M_n(\C)$ is an \emph{$(X,\epsilon)$-$*$-homomorphism} if for all $a,b\in X$,
$$
\|\phi(ab)-\phi(a)\phi(b)\|<\epsilon.
$$
Two $(X,\epsilon)$-$*$-homomorphisms $\phi,\psi:A\to M_n(\C)$ are \emph{stably $(X,\epsilon)$-equivalent} if there exists another $(X,\epsilon)$-$*$-homomorphism $\theta:A\to M_m(\C)$ and a unitary $u\in M_{n+m}(\C)$ such that
$$
\|u(\phi(a)\oplus \theta(a))u^*-\psi(a)\oplus \theta(a)\|<\epsilon
$$
for all $a\in X$.  The set of equivalence classes is an abelian semigroup when equipped with the direct sum operation, and we define $K_\epsilon(X)$ to be its Grothendieck group.  As $X$ and $\epsilon$ vary, the collection of groups $K_\epsilon(X)$ naturally define an inverse system, and we may define $KL(A)$ to be its inverse limit.  We then have the following theorem.

\begin{theorem}\label{k hom the}
Assume that $A$ is a separable, unital, and quasidiagonal $C^*$-algebra.  Then the group $KL(A)$ defined above is canonically homeomorphically isomorphic to the maximal Hausdorff quotient of the $K$-homology group $K^0(A)$.
\end{theorem}

In connection with this theorem, the pairing between $K^0(A)$ and the $K$-theory group $K_0(A)$ can be realized concretely as follows.  Let a class $x$ in $K^0(A)$ be given, and let a projection $p\in M_n(A)$ represent a $K$-theory class $[p]\in K_0(A)$.  Then for all suitably large finite subsets $X$ of the unit ball of $A$, and all suitably small $\epsilon>0$ the following holds.   Write the canonical image of $x$ in $K_\epsilon(X)$ as a formal difference $[\phi]-[\psi]$ of $(X,\epsilon)$-$*$-homomorphisms.  Then the pairing of $x$ with $[p]$ is given by  
$$
\text{rank}(\chi((\phi\otimes \text{id}_{M_n(\C)})(p))) - \text{rank}(\chi((\psi\otimes \text{id}_{M_n(\C)})(p)))\in \Z,
$$ 
where $\chi$ is the characteristic function of $(1/2,\infty)$.  Informally, one computes the ranks of any projections that are suitably close to $\phi(p)$ and $\psi(p)$, and takes their difference.

Analogous results (to both Theorem \ref{k hom the} and the subsequent comments on pairings) hold for $KK(A,B)$ in place of $K^0(A)$, at least if one assumes that $B$ is separable, nuclear, and unital: see for example Theorems \ref{main ucp} and \ref{kl ck hom}, and Remark \ref{lin pair} below for details.  We note however that we do not make any nuclearity or exactness assumptions on $A$: this is crucial for our applications.

\subsection{Outline of paper}

In Section \ref{qr bun sec}, we define the rings and homomorphism in line \eqref{intro qr as}.  This is all done fairly concretely, requiring `only' some  basic operator $K$-theory at the level of say \cite{Rordam:2000mz} or \cite{Wegge-Olsen:1993kx}.  In particular, we do not use any $KK$-theory or advanced $C^*$-algebra theory here.

We then completely switch track, and spend most of the rest of the paper on abstract $C^*$-algebra machinery.  Section \ref{qd sec} recalls the controlled $KK$-theory groups $KK_\epsilon(X,B)$ of the author and Yu, and presents a `finite-dimensionalization' of these in the presence of quasidiagonality: roughly, the new groups consist of formal differences of compact projections on a Hilbert $B$-module that commute with elements of the finite subset $X$ of $A$ up to error $\epsilon$.  Sections \ref{mvn sec} and \ref{ucp sec} then reformulate this picture of $KK$-theory based on almost commuting compact projections in terms of almost multiplicative ucp maps from $A$ into the operators on finitely generated projective modules over $B$, up to approximate unitary equivalence; this is a fairly standard translation exercise using Stinespring's theorem and absorbing representations in the sense of Voiculescu's theorem, although complicated by the fact that we have to keep track of estimates throughout.  Section \ref{llp sec} then uses the lifting property to drop the ucp assumption on our almost multiplicative maps. 

All this $C^*$-algebraic work culminates in Section \ref{main sec} where we give our main theorems describing $KL(A,B)$ in terms of inverse limits of groups of approximately multiplicative (ucp) maps.

The next two sections refine the theory.  Section \ref{rfd sec} shows that under a residual finite-dimensionality assumption we can give more restrictive descriptions of our groups of approximately multiplicative maps.  Section \ref{pair sec} then describes the pairing of $KL(A,B)$ with $K$-theory in concrete terms: a cycle in $KL(A,B)$ represented by an approximately multiplicative ucp map acts on $K$-theory simply by applying the map to a projection, and using the functional calculus to perturb the image to an honest projection.

Finally, in the last two sections we go back to group theory.  Section \ref{qr finish} puts everything together and in particular deduces that the map of line \eqref{intro qr as} is an isomorphism under appropriate assumptions.  Finally, Section \ref{tor sec} establishes Theorem \ref{intro tor the}.

We also summarize some necessary algebraic facts about inverse limits in Appendix \ref{inv lim app}.

The paper is written in such a way that a reader interested only in the group-theoretic material can read only Sections \ref{qr bun sec}, \ref{qr finish}, and \ref{tor sec}, treating the $C^*$-algebraic material as a black box.  On the other hand, a reader interested only in the $C^*$-algebraic material can read only Sections \ref{qd sec} through \ref{pair sec}, and ignore the group-theoretic applications.

\subsection{Acknowledgements}

I am grateful to Jos\'{e} Carri\'{o}n, Robin Deeley, Christopher Schafhauser, Andreas Thom, Jianchao Wu, and Guoliang Yu for several useful discussions around the issues in this paper.  I am also grateful to Guihua Gong for an invitation to give a lecture series at Hebei Normal University during the summer of 2025, and the audience from that lecture series: this motivated me (eventually) to develop some of the ideas in this paper more fully.  Support from the US NSF (DMS 2247968) and Simons Foundation (MP-TSM-00002363) is also gratefully acknowledged.

\subsection{AI statement}

The ideas in this paper predate the really big advances in the uses of AI in mathematics from Summer 2026, and almost all of the mathematics was done without any attempt to use AI assistance. 

There is one exception where I at least tried quite hard to find an AI-assisted proof: this was to show that there is a closed aspherical manifold $M$ and torsion classes in $K^0(M)$ that are not in the subgroup generated by flat bundles.  I was not successful, and the example appearing in the paper (see Proposition \ref{flat sg}) was found without AI input.  However, this was around three months ago and the models are of course increasingly powerful: I suspect it would do a better job now.  

The writing was not in any way AI-assisted (probably foolishly).

\section{Quasi-representations and vector bundles}\label{qr bun sec}

Let $\Gamma$ be a countable group.  A quasi-representation is a map from a finite subset of $\Gamma$ to a finite-dimensional unitary group that approximately satisfies the relations needed to be a representation.  If $Y$ is a reasonably nice space equipped with a map $\pi_1(Y)\to \Gamma$, then such a quasi-representation can be used to build a vector bundle over $Y$, and thus an element of $K^0(Y)$.  This is a variant of the Atiyah-Segal maps \cite{Atiyah:1961uq,Atiyah:1969aa} taking honest representations to vector bundles.  This is well-known (see for example \cite{Carrion:2015ld,Hunger:2016wv,Kubota:2022aa}), although we will develop this independently in this section in a way that is convenient for our purposes.  

In this section we formalize all of this.  In particular, we build the collection of all quasi-representations into a commutative ring, and build an Atiyah-Segal type homomorphism from this ring to the $K$-theory ring of a suitably nice space with fundamental group $\Gamma$ (in applications, the space will be the classifying space of $\Gamma$).

Throughout this section, if $V$ is a finite-dimensional complex Hilbert space, $U(V)$ denotes its unitary group.  The group $U(V)$ will always be equipped with the operator norm.

\begin{definition}\label{qr}
Let $\Gamma$ be a group, let $S\subseteq \Gamma$ be a finite subset that contains the identity, and let $\epsilon>0$.  An \emph{$(S,\epsilon)$-(unitary) representation} of $\Gamma$ is a map $\phi:S\to U(V)$ from $S$ to the unitary group of a finite-dimensional complex Hilbert space $V$ such that $\phi(1)=1$, and such that
$$
\|\phi(st)-\phi(s)\phi(t)\|<\epsilon
$$
whenever $s,t$ and $st$ are all in $S$.  We allow the case when $V$ is a zero-dimensional Hilbert space, in which case the conditions should be read as vacuous, i.e.\ the unique map $\phi:S\to U(V)$ is automatically an $(S,\epsilon)$-representation.

If we want to be ambiguous about the data $(S,\epsilon)$, we will just say that $\phi$ is a \emph{quasi-representation}.  We will often write $V_\phi$ for the Hilbert space underlying a quasi-representation $\phi$.

If $\phi,\psi$ are $(S,\epsilon)$-representations, we write $\phi\oplus \psi$ for the map defined by
$$
\phi\oplus \psi:S\to U(V_\phi\oplus V_\psi),\quad g\mapsto \phi(g)\oplus \psi(g).
$$
\end{definition}

\begin{remark}
With notation as in Definition \ref{qr}, if $\phi,\psi$ are $(S,\epsilon)$-representations, then so is $\phi\oplus \psi$.  This follows as the operator norm satisfies $\|a\oplus b\|=\max\{\|a\|,\|b\|\}$ for the direct sum of operators on a direct sum of Hilbert spaces.  This fact fails for all other `reasomable' matrix norms, and is one reason we work only with the operator norm.
\end{remark}

\begin{remark}
In what follows, we will often want to talk about the collection of `all' $(S,\epsilon)$-representations modulo a certain equivalence relation.  As the collection of all $(S,\epsilon)$-representations is not a set, this takes a little care: one way to get around this issue is (for example) to look only at those $(S,\epsilon)$-representations whose underlying Hilbert space is a subspace of some fixed Hilbert space such as $\ell^2(\N)$; as we are only interested in quasi-representations up to unitary conjugation, fixing such a choice makes no real difference.  We will not say any more on this point, or on similar points as they arise later in the paper leaving the necessary choices to the preference of the reader.
\end{remark}

We record a basic lemma for later use: it says that an quasi-representation can always be perturbed to an inverse-preserving quasi-representation.

\begin{lemma}\label{sa approx}
Let $\Gamma$ be a group, let $S\subseteq \Gamma$ be a finite subset that contains the identity, and let $\epsilon>0$.  Let $\phi:S\to U(V)$ be an $(S,\epsilon)$-representation as in Definition \ref{qr}.  Then there exists a map $\psi:S\to U(V)$ such that $\psi(1)=1$, such that $\psi(s^{-1})=\psi(s)^*$ for all $s\in S$ such that $s^{-1}$ is also in $S$, and such that $\|\psi(s)-\phi(s)\|<\epsilon$ for all $s\in S$.
\end{lemma}

\begin{proof}
Partition $S$ into the following sets 
$$
S=\{1\}\sqcup \{s\in S\mid s^{-1}\in S,s^2\neq 1\}\sqcup \{s\in S\mid s\neq1,s^2=1\}\sqcup\{s\in S\mid s^{-1}\not\in S\}.
$$
For $s\in \{1\}\sqcup \{s\in S\mid s^{-1}\not\in S\}$, we define $\psi(s):=\phi(s)$.  

For the set $\{s\in S\mid s^{-1}\in S,s^2\neq 1\}$ we can choose a partition $\{s\in S\mid s^{-1}\in S,s^2\neq 1\}=A\sqcup B$ such that the map $s\mapsto s^{-1}$ is a bijection from $A$ to $B$.  We define $\psi(s):=\phi(s)$ for $s\in A$, and $\psi(s):=\phi(s^{-1})^*$ for $s\in B$.  For any $s\in B$, we then have that 
$$
\|\phi(s)-\psi(s)\|=\|\phi(s)-\phi(s^{-1})^*\|=\|\phi(s^{-1})\phi(s)-1\|=\|\phi(s^{-1})\phi(s)-\phi(s^{-1}s)\|,
$$
which is less that $\epsilon$ by definition of an $(S,\epsilon)$-representation (the second equality follows as $\phi(s^{-1})$ is unitary and the third as $\phi(s^{-1}s)=\phi(1)=1$).  

Finally, for $s\in \{s\in S\mid s\neq1,s^2=1\}$, the definition of an $(S,\epsilon)$-representation implies that $\|\phi(s)^2-1\|<\epsilon$.  Hence for any $z$ in the spectrum of $\phi(s)$, $|z^2-1|<\epsilon$, i.e.\ $|1+z||1-z|<\epsilon$.  As $z$ is in the unit circle, one of $|1-z|$ or $|1+z|$ is at least one, whence the other of these terms is less than $\epsilon$.  Hence if $f(z)=\text{sign}(\text{Re}(z))$ (interpreted as $0$ if $\text{Re}(z)=0$) then $f(\phi(s))^2=1$ and $\|\phi(s)-f(\phi(s))\|<\epsilon$ by the functional calculus.  We define $\psi(s):=f(\phi(s))$.  This completes the construction of $\psi$, and the proof.
\end{proof}

\begin{definition}\label{qr equiv}
Let $\Gamma$ be a group.  Two $(S,\epsilon)$-representations $\phi,\psi$ are \emph{stably $(S,\epsilon)$-equivalent} if there exists another $(S,\epsilon)$-representation $\theta$ and a unitary isomorphism $v:V_\phi\oplus V_\theta\to V_\psi\oplus V_\theta$ such that 
$$
\|v(\phi(s)\oplus \theta(s))v^*-\psi(s)\oplus \theta(s)\|<\epsilon
$$
for all $s\in S$.  
\end{definition}

\begin{lemma}
With notation as in Definition \ref{qr equiv}, stable $(S,\epsilon)$-equivalence is an equivalence relation on the collection of all $(S,\epsilon)$-representations.
\end{lemma}

\begin{proof}
We just show transitivity: reflexivity and symmetry are clear.  For this, assume that $\phi_1$ and $\psi$, and $\psi$ and $\phi_2$, are pairs of stably $(S,\epsilon)$-equivalent $(S,\epsilon)$-representations.  Let $\theta_1,\theta_2$ and $v_1$, $v_2$ realize these equivalences in the sense that 
$$
\|v_1(\phi_1(s)\oplus \theta_1(s))v_1^*-\psi(s)\oplus \theta_1(s)\|<\epsilon,\quad \|v_2(\psi(s)\oplus \theta_2(s))v_2^*-\phi_2(s)\oplus \theta_2(s)\|<\epsilon
$$
for all $s\in S$.  Define now $\theta:=\theta_1\oplus \psi\oplus \theta_2$.  Let 
$$
w:V_{\psi}\oplus V_{\theta_1}\oplus V_{\phi_2}\oplus V_{\theta_2}\to V_{\phi_2}\oplus V_{\theta_1} \oplus V_\psi\oplus V_{\theta_2}
$$
be the canonical unitary changing the order of the summands.  Define $v:=w\begin{psmallmatrix} v_1 & 0 \\ 0 & v_2\end{psmallmatrix}$.  Then one checks that 
$$
\|v(\phi_1(s)\oplus \theta(s))v^*-\phi_2(s)\oplus \theta(s)\|<\epsilon
$$
for all $s\in S$, completing the proof.
\end{proof}

\begin{definition}\label{equiv cl}
Let $\Gamma$ be a group, let $S$ be a finite subset of $\Gamma$ containing the identity, and let $\epsilon>0$.  With terminology as in Definitions \ref{qr} and \ref{qr equiv}, we define $V_\epsilon(S)$ to be the set of all $(S,\epsilon)$-representations modulo stable $(S,\epsilon)$-equivalence.
\end{definition}

We leave the details of the next lemma to the reader.

\begin{lemma}
Let $V_\epsilon(S)$ be as in Definition \ref{equiv cl}.  Then the operation ``\,$\oplus$'' defined on $(S,\epsilon)$-representations in Definition \ref{qr} descends to a well-defined binary operation on $V_\epsilon(S)$ that makes it into a commutative monoid, with identity the class of the zero quasi-representation. \qed
\end{lemma}

\begin{definition}\label{gp gp}
Let $\Gamma$ be a group, let $S$ be a finite subset of $\Gamma$ containing the identity, and let $\epsilon>0$.  We write $K_\epsilon(S)$ for the Grothendieck group of the monoid $V_\epsilon(S)$ from Definition \ref{equiv cl}. 
\end{definition}

\begin{definition}\label{gp order}
Let $\Gamma$ be a group.  We write $I_\Gamma$ for the directed set consisting of pairs $(S,\epsilon)$ with $S$ a finite subset of $\Gamma$, $\epsilon>0$, and where we say $(S,\epsilon)\leq (T,\delta)$ if $S\subseteq T$ and $\delta\leq \epsilon$.
\end{definition}

Note that if $(S,\epsilon)\leq (T,\delta)$ in $I_\Gamma$, then there is a canonical `forget control' map 
\begin{equation}\label{fc}
K_\delta(T)\to K_\epsilon(S)
\end{equation}
defined by considering $(T,\delta)$-representations as $(S,\epsilon)$-representations.

\begin{definition}\label{inv lim approx}
Let $\Gamma$ be a countable group, and let $I_\Gamma$ be as in Definition \ref{gp order}.   These forget control maps of line \eqref{fc} make the collection $(K_\epsilon(S))_{(S,\epsilon)\in I_\Gamma}$ into an inverse system, and we define 
$$
QR(\Gamma):=\lim_{\leftarrow}K_\epsilon(S)
$$
to be their inverse limit in the category of topological abelian groups (compare Definition \ref{inv lim def}).
\end{definition}

\begin{remark}
In \cite{Weinberger:2026aa}, Weinberger, Wu, and Yu independently introduce the same notation ``$QR(\Gamma)$'' for what turns out to be the same object, although the definition is slightly different.  Their $QR(\Gamma)$ is built as an inverse limit of groups they call ``$QR(S,\epsilon)$''.  These are slightly different from our $K_\epsilon(S)$: for example, their $(S,\epsilon)$-representations are defined on $S^2$, so that the expression ``$\phi(s)\phi(t)-\phi(st)$'' always make sense for $s,t\in S$; and the equivalence relation used is very slightly different.  However, none of these variants make any difference in the inverse limit (at least for the group structure: Weinberger-Wu-Yu do not seem to introduce a ring structure on their $QR(\Gamma)$, but this could be done in the same way we do, of course).
\end{remark}

Our next goal is to use tensor products to make $QR(\Gamma)$ into a ring using tensor products of quasi-representations.  This is a little more complicated than the group structure as the ring structure is not present in the groups $K_\epsilon(S)$, only appearing `in the limit'.

\begin{definition}\label{tp qr}
Let $\Gamma$ be a group.  Let $(S,\epsilon)$, $(T,\delta)$ be elements of the set $I_\Gamma$ of Definition \ref{gp order}, and let $\phi$ and $\psi$ be an $(S,\epsilon)$-representation and a $(T,\delta)$-representation respectively.  Their \emph{tensor product} is the map 
$$
\phi\otimes \psi:S\cap T\to U(V_\phi\otimes V_\psi),\quad s\mapsto \phi(s)\otimes \psi(s).
$$
\end{definition}

\begin{lemma}\label{tp qr lem}
With notation as in Definition \ref{tp qr}, if $\phi$ is an $(S,\epsilon)$-representation and $\psi$ is a $(T,\delta)$ representation, then $\phi\otimes \psi$ is a $(S\cap T,\epsilon+\delta)$-representation.  Moreover, if $\phi_1$ and $\phi_2$ are stably $(S,\epsilon)$-equivalent $(S,\epsilon)$-representations, then $\phi_1\otimes \psi$ and $\phi_2\otimes \psi$ are stably $(S\cap T,\epsilon)$-equivalent, and similarly with the roles of $(S,\epsilon)$ and $(T,\delta)$ reversed.
\end{lemma}

\begin{proof}
We compute that whenever $s,t,st\in S\cap T$,
\begin{align*}
\|(\phi(st)&\otimes \psi(st)) -(\phi(s)\otimes \psi(s))(\phi(t)\otimes \psi(t))\| \\ &\leq \|\phi(st)\otimes (\psi(st)-\psi(s)\psi(t))\|+\|(\phi(st)-\phi(s)\phi(t))\otimes \psi(st)\|\\
& \leq   \|\phi(st)\|\|(\psi(st)-\psi(s)\psi(t))\|+\|(\phi(st)-\phi(s)\phi(t))\|\|\psi(st)\|
\\ &<\epsilon+\delta.
\end{align*}
Computations checking the equivalence relations behave as claimed are similar.
\end{proof}

\begin{proposition}\label{klg ring}
Let $\Gamma$ be a group, and let $(S,\epsilon),(T,\delta)$ be elements of the set $I_\Gamma$ of Definition \ref{gp order}.  Then the assignment 
\begin{equation}\label{tp v}
 [\phi]\otimes [\psi]\mapsto [\phi\otimes \psi]
\end{equation}
uniquely determines a well-defined homomorphism
\begin{equation}\label{tp small group}
K_\epsilon(S)\otimes K_\delta(T)\to K_{\epsilon+\delta}(S\cap T).
\end{equation}
These maps fit together to induce a product map $QR(\Gamma)\otimes QR(\Gamma)\to QR(\Gamma)$, making $QR(\Gamma)$ into a commutative, unital topological ring (with multiplicative identity the class of the trivial representation - compare Remark \ref{act rep rem}). 
\end{proposition}

\begin{proof}
The formula in line \eqref{tp v} induces a well-defined ($\N$-)bilinear map $V_\epsilon(S)\times V_\delta(T)\to V_{\epsilon+\delta}(S\cap T)$ by Lemma \ref{tp qr lem}.  It uniquely determines a well-defined homomorphism as in line \eqref{tp small group} by the universal properties of the Grothendieck group, and of the tensor product.  

We now provisionally define a homomorphism $QR(\Gamma)\otimes QR(\Gamma)\to QR(\Gamma)$ as follows.  Recall (see Definition \ref{inv lim def} and \ref{inv lim approx}) that an element of $QR(\Gamma)$ consists of a tuple $([\phi_{S,\epsilon}])_{(S,\epsilon)\in I_\Gamma}$ that is compatible with the forget control maps.  Given two such tuples $[\phi]=([\phi_{(S,\epsilon)}])$ and $[\psi]=([\psi_{(S,\epsilon)}])$, we provisionally define their product to be the tuple $(\theta_{(S,\epsilon)})$, where $\theta_{(S,\epsilon)}:=\phi_{(S,\epsilon/2)}\otimes \psi_{(S,\epsilon/2)}$.  We then have to check that if $(T,\delta)\geq (S,\delta)$ in the set $I_\Gamma$ of Definition \ref{gp order} that the classes defined by $\phi_{(S,\epsilon/2)}\otimes \psi_{(S,\epsilon/2)}$ and $\phi_{(T,\delta/2)}\otimes \psi_{(T,\delta/2)}$ in $K_\epsilon(S)$ are the same.  This follows directly from Lemma \ref{tp qr lem} and that the classes defined by $\phi_{(T,\delta/2)}$ and $\phi_{(S,\epsilon/2)}$ are the same in $K_0(S,\epsilon/2)$, and similarly for the corresponding `$\psi$-classes'.

Having got that the map $QR(\Gamma)\otimes QR(\Gamma)\to QR(\Gamma)$ is well-defined, checking that $QR(\Gamma)$ satisfies the properties of a commutative, unital, topological ring is routine based on standard commutativity, associativity, and distributivity properties of tensor products and direct sums (that hold up to unitary equivalence): we leave the details to the reader.
\end{proof}

\begin{remark}\label{act rep rem}
Let $\Gamma$ be a group, and let $R(\Gamma)$ denote the ring of unitary equivalence classes of finite-dimensional unitary representations of $\Gamma$: the addition operation is defined by direct sum, and the multiplication by tensor product.  There is then a natural ring homomorphism $R(\Gamma)\to QR(\Gamma)$ defined by treating an honest representation as an $(S,\epsilon)$-representation for all $S$ and $\epsilon$.  Partly to get used to the definitions, let us show directly that this homomorphism is in general a long way from being injective.  

For example, if $\Gamma=\Z$, then $R(\Gamma)$ is isomorphic to the ring of Laurent polynomials in a single variable with integer coefficients, but we claim that the image of $R(\Gamma)$ in $QR(\Gamma)$ is just integer multiples of the trivial representation.  Indeed, it suffices to show that the class of any irreducible representation of $\Z$ in $QR(\Gamma)$ is the same as the class of the trivial representation $\pi_0$ in any $K_\epsilon(S)$.  Any such irreducible representation is canonically isomorphic to the representation $\pi_z:\Gamma\to U(\C)$ defined by $\pi_z(n)=z^n$ for all $n\in \Z$ and some fixed $z\in \C$ with modulus one.  Given $(S,\epsilon)$, choose a finite string of points $1=z_0,z_1,\cdots,z_N=z$ with the property that $|z^n_i-z_{i+1}^n|<\epsilon$ for all $n\in S$ and all $i\in \{0,...,N-1\}$.  Let $\phi:=\bigoplus_{i=1}^{N-1} \pi_{z_i}$; this is an honest representation, so certainly an $(S,\epsilon)$-representation for any $S$ and $\epsilon$.  Then for an appropriately chosen permutation unitary $v$, one sees that 
$$
\|v(\pi_0(s)\oplus \phi(s))v^*-\pi_N(s)\oplus \phi(s)\|<\epsilon
$$
for all $s\in S$.  Hence $\pi_0$ and $\pi_z$ are stably $(S,\epsilon)$-equivalent for all $(S,\epsilon)$ as claimed.  

A similar direct argument shows that for any finitely generated torsion-free abelian group the image of the group $R(\Gamma)$ in $QR(\Gamma)$ just consists of multiples of the trivial representation: the essential point is that the space of characters is metrizable and path connected.  It will follow from the results of this paper (for example, from Theorem \ref{intro main}) that $KL(\Z)$ itself just reduces to integer multiples of the trivial representation; this is, however, not true for $\Z^n$ for $n\geq 2$ and other more complicated abelian groups.
\end{remark}

In the remainder of this section, we construct a map relating the ring $QR(\Gamma)$ to the $K$-theory of any reasonable space with fundamental group $\Gamma$.  The following definitions follows Atiyah \cite[Section 4]{Atiyah:1961uq}.\footnote{Our notation however follows Kasparov \cite[Section 9]{Kasparov:1993hc}; Atiyah uses ``$\mathcal{K}(X)$'' for what we have called ``$LK^0(X)$''.}

\begin{definition}\label{lim kt}
Let $X$ be a Hausdorff topological space   We define the \emph{limit $K$-theory} of $X$ to be the inverse limit 
$$
LK^0(X):= \lim_{\leftarrow} K^0(Y)
$$
where $Y$ ranges over all compact subspaces of $X$, and the connecting maps in the inverse system are the usual restriction maps functorially induced by inclusion.  Here each $K^0(Y)$ is considered as a topological ring (with the discrete topology), and the inverse limit is defined in the category of topological rings, so $LK^0(X)$ is considered as a topological ring (compare Definition \ref{inv lim def}).
\end{definition}

\begin{remark}\label{bun rem}
Note that if $E$ is an honest (locally trivial, finite-dimensional) complex vector bundle over $X$, then $E$ canonically defines a class $[E]\in LK^0X)$: indeed, its restrictions to compact subspaces define a coherent family, so an element of the inverse limit group.  In particular, if $X$ is path connected and admits a universal cover, $\Gamma$ is the fundamental group of $X$, and if $\pi:\pi_1(X)\to U(V)$ is a finite-dimensional unitary representation, then we get the associated flat bundle $E_\pi:=\widetilde{X}\times_\Gamma V$ over $X$, and associated class $[E_\pi]\in LK^0X)$ (here $\widetilde{X}$ is the universal cover of $X$ with action of $\Gamma$ by deck transformations, and ``$\widetilde{X}\times_\Gamma V$'' means the quotient of $\widetilde{X}\times V$ by the diagonal action $g(x,v)=(gx,\pi(g)v)$).
\end{remark}

Let now $X$ be a paracompact, Hausdorff, path connected, locally path connected, and semilocally simply connected topological space; we will only actually use the case where $X=B\Gamma$ is a CW complex model for the classifying space for a discrete group $\Gamma$, but it seemed simpler to work with some extra generality.  The connectedness assumptions imply that there is a universal cover $\pi:\widetilde{X}\to X$ of $X$ (see for example \cite[Theorem 82.1]{Munkres:2000fy}).  Let $\Gamma$ be the fundamental group of $X$, and have $\Gamma$ act on $\widetilde{X}$ on the left by deck transformations.  The following ad-hoc definition will be useful.

\begin{definition}\label{cov dat}
Let $X$, $\Gamma$, and $\pi:\widetilde{X}\to X$ be as in the discussion above.  A collection of \emph{covering coordinate data} for $X$ consists of the following:
\begin{enumerate}[(i)]
\item \label{cov co} a locally finite cover $\mathcal{U}$ of $X$ and for each $U\in \mathcal{U}$ a $\Gamma$-equivariant homeomorphism $\psi_U:\pi^{-1}(U)\to U\times \Gamma$ such that $\pi(\psi_U^{-1}(x,g))=s$ for all $x\in U$ and $g\in \Gamma$ (here $\Gamma$ equivariance is with respect to the action on $\pi^{-1}(U)$ by deck transformations, and the product action on $U\times \Gamma$ coming from the trivial action on $U$, and the left translation action of $\Gamma$ on itself);
\item a partition of unity $(\chi_U)_{U\in \mathcal{U}}$ on $X$ subordinate to $\mathcal{U}$.
\end{enumerate}
\end{definition}

Note that the assumptions imply that a collection of covering coordinate data always exists: indeed, by definition of the universal cover there exists a cover and homeomorphisms with the properties in part \eqref{cov co}, other than the cover may not be locally finite; as $X$ is paracompact and Hausdorff, this cover then has a locally finite refinement that admits a partition of unity (see for example \cite[Theorem 41.7]{Munkres:2000fy}).

Now, fix a collection of covering coordinate data for $X$, and let $U,V\in \mathcal{U}$ with $U\cap V\neq \varnothing$.  Then the composition 
$$
\psi_V\circ \psi_U^{-1}:(U\cap V)\times \Gamma\to (U\cap V)\times \Gamma
$$
is necessarily of the form $(x,h)\mapsto (x,hg_{VU})$ for some element $g_{VU}\in \Gamma$.  Note that if $U\cap V\cap W\neq \varnothing$, then 
$$
g_{VU}\circ g_{UW}=\pi_V\circ \pi_U^{-1}\circ \pi_U\circ \pi_W^{-1}=g_{VW},
$$
i.e.\ the family $(g_{VU})$ satisfies the cocycle identity.  

Variants of the next technical lemma are well-known: we just need the second part now, but record the first part for later use.  

\begin{definition}\label{delta qp}
Let $C$ be a $C^*$-algebra, and $\delta\in (0,1/2)$.  A self-adjoint element $c\in C$ is a \emph{$\delta$-quasiprojection} if the spectrum of $c$ does not intersect $[\delta,1-\delta]$.
\end{definition}

\begin{lemma}\label{chi com}
Let $C$ be a $C^*$-algebra. Let $\delta\in (0,1/2)$, and let $E_\delta$ denote the set of $\delta$-quasiprojections in $C$.  Let $\chi$ be the characteristic function of $(1/2,\infty)$.  Then the following hold:
\begin{enumerate}[(i)]
\item \label{chi com est} for any $x\in E_\delta$ and $c\in C$,
$$
\|[\chi(x),c]\|\leq \frac{\|[x,c]\|}{1-2\delta};
$$
\item \label{chi lip} the map $E_\delta\to C$ defined by $x\mapsto \chi(x)$ is $(1-2\delta)^{-1}$-Lipschitz.
\end{enumerate}
\end{lemma}

Examples in $M_2(\C)$ based on the pair $x=\begin{psmallmatrix} \delta & 0 \\ 0 & 1-\delta\end{psmallmatrix}$ and $c=\begin{psmallmatrix} 0 & 1 \\ 1 & 0 \end{psmallmatrix}$ show that the estimates obtained are in some sense optimal: we leave the precise statements to the interested reader.

\begin{proof}[Proof of Lemma \ref{chi com}]
Part \eqref{chi com est} is \cite[Lemma 2.2.4]{Willett:2021te}.  Part \eqref{chi lip} follows from essentially the same argument as in \cite[Lemma 2.2.4]{Willett:2021te}; we give the argument for the reader's convenience.  Indeed, given $x,y\in E_\delta$, we use the holomorphic functional calculus to write 
$$
\chi(x)-\chi(y)=\frac{1}{2\pi i} \int_\gamma (z-x)^{-1}-(z-y)^{-1}dz,
$$
where $\gamma$ is the positively oriented contour with vertices $\frac{1}{2}\pm iN$ and $N+\frac{1}{2}\pm iN$ for some $N$ satisfying $N>\max\{\|y\|,\|x\|\}$.  Using the resolvent identity
$$
(z-x)^{-1}-(z-y)^{-1}=(z-x)^{-1}(x-y)(z-y)^{-1}
$$ 
we estimate 
$$
\|\chi(x)-\chi(y)\|\leq \frac{1}{2\pi} \int_\gamma \|(z-x)^{-1}\|\|x-y\|\|(z-y)^{-1}\|d|z|. 
$$
Let $\gamma_1$ be the side of $\gamma$ given by $\{\frac{1}{2}+it\mid -N\leq t\leq N\}$, and $\gamma_2$ be the union of the other three sides, so we have 
\begin{align}\label{chi est1}
\|\chi(x)-\chi(y)\|\leq \frac{\|x-y\|}{2\pi} \Bigg(\int_{\gamma_1} \nonumber & \|(z-x)^{-1}\|\|(z-y)^{-1}\|d|z| \\ &+\int_{\gamma_2}\|(z-x)^{-1}\|\|(z-y)^{-1}\|d|z|\Bigg).
\end{align}  
The continuous functional calculus implies that $\|(z-x)^{-1}\|\leq ((\frac{1}{2}-\delta)^2+t^2)^{-1/2}$ for $z$ in the image of $\gamma_1$, and similarly for $\|(z-y)^{-1}\|$, whence
\begin{align}\label{gam1 est}
\int_{\gamma_1}\|(z-x)^{-1}\|\|(z-y)^{-1}\|d|z| & \leq \int_{-N}^N\frac{1}{(\frac{1}{2}-\delta)^2+t^2}dt\leq \int_{-\infty}^\infty\frac{1}{(\frac{1}{2}-\delta)^2+t^2}dt \nonumber \\ & =\frac{\pi}{\frac{1}{2}-\delta}.
\end{align}
On the other hand, for $z$ in the image of $\gamma_2$, we have that $\|(z-x)^{-1}\|\leq (N-\|x\|)$ and similarly for $y$.  As $\gamma_2$ has length $3N$, we thus see that 
\begin{equation}\label{gam2 est}
\int_{\gamma_2}\|(z-x)^{-1}\|\|(z-y)^{-1}\|d|z|\leq \frac{3N}{(N-\|x\|)(N-\|y\|)}.
\end{equation}
Combining lines \eqref{chi est1}, \eqref{gam1 est} and \eqref{gam2 est} and letting $N\to\infty$, the result follows.
\end{proof}

\begin{lemma}\label{coass tech lem}
Let now $X$ be a paracompact, Hausdorff, path connected, locally path connected, and semilocally simply connected topological space with fundamental group $\Gamma$.  let  $(\mathcal{U},(\psi_U),(\chi_U))$ be a fixed collection of covering coordinate data as in Definition \ref{cov dat} above.  

For a compact subspace $Y$ of $X$, let $U_1,...,U_n$ be an enumeration\footnote{Local finiteness of the cover and compactness of $Y$ imply that there are only finitely many such sets.} of the elements of $\mathcal{U}$ that non-trivially intersect $Y$.  Let $\chi$ be the characteristic function of $(1/2,\infty)$.  Write $g_{ij}$ as shorthand for $g_{U_iU_j}$ and $\chi_{i}$ as shorthand for $\chi_{U_i}$.  Let $S_Y:=\{g_{ij}\mid i,j=1,...,n\}\subseteq \Gamma$, and let $\epsilon_Y:=1/24n^3$.  The following hold.
\begin{enumerate}[(i)]
\item \label{con proj} Let $\phi:S_Y\to U(V)$ be an $(S_Y,\epsilon_Y)$-representation.  Let $\psi:S_Y\to U(V)$ be a unital function such that $\|\psi(s)-\phi(s)\|<\epsilon_Y$ for all $s\in S_Y$, and such that $\psi(s^{-1})=\psi(s^*)$ for all $s\in S_Y$ such that $s^{-1}$ is also in $S_Y$; such a $\psi$ exists by Lemma \ref{sa approx}.  Then the formula 
$$
p_\phi:=\chi\Bigg(\sum_{i,j=1}^n e_{ij}\otimes \psi(g_{ij})\otimes \chi_i^{1/2}\chi_j^{1/2}\Bigg)
$$
defines a projection in $M_n(\C)\otimes \mathcal{B}(V)\otimes C(Y)$ (here $e_{ij}$ is the usual matrix unit of $M_n(\C)$ and we abuse notation slightly by identifying each $\chi_i$ with its restriction to $Y$). 
\item \label{coass base hom} With notation as in part \eqref{con proj}, if $(S,\epsilon)\geq (S_Y,\epsilon_Y)$ in the directed set $I_\Gamma$ of Definition \ref{gp order}, then the assignment $[\phi]\to [p_\phi]$ determines a well-defined group homomorphism $\alpha:K_\epsilon(S)\to K^0(Y)$ (depending on the choice of covering coordinate data).
\item \label{coass nat} If $Y\subseteq Z$ are finite subcomplexes of $X$ and $S_Y$, $S_Z$ and $\epsilon_Y$, $\epsilon_Z$ are as defined above, then for any $(S,\epsilon)\geq (S_Y,\epsilon_Y)$ and any $(T,\delta)\geq (S\cup S_Z,\min\{\epsilon_Z,\epsilon))$ in the directed set $I_\Gamma$ of Definition \ref{gp order}, the diagram 
$$
\xymatrix{K_\delta(T) \ar[r] \ar[d] &  K_\epsilon(S) \ar[d] \\
K^0(Z) \ar[r] & K^0(Y) },
$$
commutes, where: the top horizontal map is the usual forget control map; the bottom horizontal map is the usual restriction morphism; and the vertical maps are the homomorphisms from part \eqref{coass base hom}.
\end{enumerate}
\end{lemma}

\begin{proof}
For part \eqref{con proj}, define 
\begin{equation}\label{app}
a:=\sum_{i,j=1}^n e_{ij}\otimes \psi(g_{ij})\otimes \chi_i^{1/2}\chi_j^{1/2}\in M_n(\C)\otimes \mathcal{B}(V)\otimes C(Y).
\end{equation}
As $g_{ij}^{-1}=g_{ji}$ and as $\psi(g^{-1})=\psi(g)^*$ for all $g\in S_Y$, we see that $a$ is self-adjoint.  Moreover, 
$$
a^2=\sum_{i,j,k=1}^n e_{ij}\otimes \psi(g_{ik})\psi(g_{kj})\otimes \chi_i^{1/2}\chi_k\chi_j^{1/2}.
$$
As $g_{ik}g_{kj}=g_{ij}$ on the support of $\chi_i^{1/2}\chi_j\chi_k^{1/2}$ and as $\psi$ is an $(S,3/(24n^3))$-representation, this is less than $1/8$ away from 
$$
\sum_{i,j,k=1}^n e_{ij}\otimes \psi(g_{ij})\otimes \chi_i^{1/2}\chi_k\chi_j^{1/2}=a.
$$ 
As $\|a^2-a\|<1/8$, we have that the spectrum of $a$ avoids $[1/4,3/4]$, and therefore $\chi(a)$ is a well-defined projection in $M_n(\C)\otimes \mathcal{B}(V)\otimes C(Y)$ as claimed.

For part \eqref{coass base hom}, we first show that $p_\phi$ does not depend on the choice of $\psi$.  Indeed, if $\psi$ and $\psi'$ are any two choices satisfying the conditions in part \eqref{con proj}, then $\|\psi(s)-\psi'(s)\|<1/(12n^3)$ for all $s\in S_Y$.  Hence if $a$ and $a'$ are as in line \eqref{app} associated to $\psi$ and $\psi'$ respectively, we have $\|a-a'\|<1/(12n)\leq 1/12$.  Using Lemma \ref{chi com} part \eqref{chi lip} and the fact observed above that the spectra of $a$ and $a'$ both avoid $[1/4,3/4]$, we see that $\|\chi(a)-\chi(a')\|<2/12$.  Hence the projections $\chi(a)$ and $\chi(a')$ define the same class in the $K$-theory of $Y$ as claimed.  

We next show that the map $\phi\mapsto [p_\phi]$ descends to a well-defined monoid homomorphism $V_{\epsilon}(S)\to K^0(Y)$; this monoid homomorphism will then induce the desired homomorphism on the Grothendieck group.  Indeed, as this assignment clearly sends block sums to block sums, it suffices to show that if $\phi:S\to V$ and $\phi':S\to V$ are $(S,\epsilon)$ representations and $u\in U(V)$ is a unitary such that $\|u\phi(s)u^*-\phi'(s)\|<\epsilon_Y$ for all $s\in S_Y$, then $[p_\phi]$ and $[p_{\phi'}]$ are the same class in $K^0(Y)$.  Indeed, let $\psi$ and $\psi'$ be as in part \eqref{con proj} for $\phi$ and $\phi'$ respectively, and let $a$ and $a'$ be built from $\psi$ and $\psi'$ as in line \eqref{app}.  Then 
\begin{align*}
\|u\psi(s)u^*-\psi'(s)\| & \leq \|u(\phi(s)-\psi(s))u^*\|+\|u\phi(s)u^*-\phi'(s)\|+\|\phi'(s)-\psi'(s)\| \\ & <3\epsilon\leq 3\epsilon_Y=1/(8n^3).
\end{align*}
Hence $\|uau^*-a'\|< 1/(8n)\leq 1/8$ (here we abuse notation, identifying $u$ with the unitary $1_{M_n(\C)}\otimes u\otimes 1_{C(X)}$ in $M_n(\C)\otimes \mathcal{B}(V)\otimes C(X)$).  As argued above, Lemma \ref{chi com} part \eqref{chi lip} thus implies that $\|\chi(uau^*)-\chi(a')\|<2/8$, and thus the projections $\chi(uau^*)$ and $\chi(a')$ define the same class in $K^0(Y)$.  Finally, note that by naturality of the functional calculus, we have that $\chi(uau^*)=u\chi(a)u^*$, whence $\chi(uau^*)$ defines the same class in $K$-theory as $\chi(a)$, and we are done with this part of the argument.

Part \eqref{coass nat} is straightforward from the definitions: we leave the details to the reader.
\end{proof}

We are finally ready for the main result of this section.

\begin{proposition}\label{coass}
Let $X$ be a paracompact, Hausdorff, path connected, locally path connected, and semilocally simply connected topological space with fundamental group $\Gamma$.  Let  $(\mathcal{U},(\psi_U),(\chi_U))$ be a fixed collection of covering coordinate data as in Definition \ref{cov dat} above.  The following hold.
\begin{enumerate}[(i)]
\item \label{hom exists}
The homomorphisms $K_\epsilon(S)\to K^0(Y)$ of Lemma \ref{coass tech lem} part \eqref{coass base hom} fit together to give a well-defined continuous ring homomorphism 
$$
\alpha:QR(\Gamma)\to LK^0X)
$$
on the level of inverse limits (where the left hand side is as in Definition \ref{inv lim approx} and the right hand side is as in Definition \ref{lim kt}).  
\item The homomorphism $\alpha$ from part \eqref{hom exists} depends only on $X$, and not on the collection of covering coordinate data made at the start of the construction.
\item The homomorphism $\alpha$ from part \eqref{hom exists} takes the class of an actual representation (compare Remark \ref{act rep rem}) to the class of the associated flat bundle in $LK^0X)$ (see Remark \ref{bun rem}).
\end{enumerate}
\end{proposition}

\begin{proof}
The fact that we have a continuous homomorphism $\alpha:QR(\Gamma)\to LK^0X)$ of abelian groups (a priori depending on the collection of covering coordinate data) follows directly from Lemma \ref{coass tech lem}.  

We now show that this homomorphism does not depend on any of the choices involved; we will show that it is a ring homomorphism.  

To show that the homomorphism $QR(\Gamma)\to LK^0X)$ does not depend on the choices involved, let $(\mathcal{U}^{(i)},(\psi_U^{(i)})_{U\in \mathcal{U}},(\chi_U^{(i)})_{U\in \mathcal{U}})$ for $i\in \{1,2\}$ be two collections of covering coordinate data as in Definition \ref{cov dat}.   Let $\alpha^{(i)}:QR(\Gamma)\to LK^0X)$ be the homomorphisms constructed from this data as in part \eqref{hom exists}.  Let $Y$ be a compact subset of $X$, and let $(S_1,\epsilon_1)$ and $(S_2,\epsilon_2)$ be such that the maps $\alpha^{(i)}:K_{\epsilon_i}(S_i)\to K^0(Y)$ from Lemma \ref{coass tech lem} are well-defined based on the two sets of data.  We need to find a pair $(T,\delta)$ in the directed set $I_\Gamma$ of Definition \ref{gp order} that dominates $(S_1,\epsilon_1)$ and $(S_2,\epsilon_2)$, and with the property that the diagram
$$
\xymatrix{ & K_{\epsilon_1}(S_1)\ar[dr]^{\alpha^{(1)}} & \\
K_\delta(T) \ar[ur] \ar[dr] & & K^0(Y) \\
& K_{\epsilon_2}(S_2) \ar[ur]_{\alpha^{(2)}} & }
$$
commutes (the unlabeled arrows on the left are the usual forget control maps).  

For this, consider a new collection of covering coordinate data with: open sets of the form $U\cap V$ with $U\in \mathcal{U}^{(1)}$, $V\in \mathcal{U}^{(2)}$, and where the index set is considered to be $\mathcal{U}^{(1)}\times \mathcal{U}^{(2)}$; the coordinate function associated to a pair $(U,V)$ is the restriction of $\psi_U^{(1)}$ to $U\cap V$; and the partition of unity is $(\psi_U\psi_V)_{(U,V)\in \mathcal{U}\times \mathcal{V}}$.  Let now $(T,\delta)$ be such that $\delta\leq \min\{\epsilon_1,\epsilon_2,1/(16(n_1n_2)^2)\}$, and such that $T$ contains all elements of $S_1$, all elements of $S_2$, and also all group elements corresponding to the compositions $\psi^{(1)}_{U_i}\circ (\psi^{(2)}_{V_k})^{-1}$ as $i$ ranges over $\{1,...,n_1\}$ and $k$ over $\{1,...,n_2\}$, and their inverses, and also all products of at most three elements from these sets; enlarging $T$ and shrinking $\delta$, we may also assume that the map $K_\delta(T)\to K^0(Y)$ of Lemma \ref{coass tech lem} is well-defined for this new `product' covering coordinate data.  We claim that this $T$ and $\delta$ have the right properties.

Now, let $\phi:S\to U(V)$ be a $(T,\delta)$-representation.  Let $U_1,...,U_{n_1}$ be the elements of $\mathcal{U}^{(1)}$ that have non-empty intersection with $Y$, and similarly for $V_1,...,V_{n_2}$ and $\mathcal{U}^{(2)}$.  For $i,j\in \{1,...,n_1\}$, let $g_{ij}$ be the group element corresponding to $\psi_{U_i}^{(1)}\circ (\psi_{U_j}^{(1)})^{-1}$, and for $k,l\in \{1,...,n_2\}$, let $h_{kl}$ be the group element corresponding to $\psi_{V_k}^{(2)}\circ (\psi_{V_l}^{(2)})^{-1}$.  With $\chi$ as usual the characteristic function of $(1/2,\infty)$, we want to show that the classes of the projections
\begin{equation}\label{sps}
\chi\Bigg(\sum_{i,j=1}^{n_1}e_{ij}\otimes \phi(g_{ij})\otimes (\chi^{(1)}_i\chi_j^{(1)})^{1/2}\Bigg)\quad \text{and}\quad \chi\Bigg(\sum_{k,l=1}^{n_2}e_{kl}\otimes \phi(h_{kl})\otimes (\chi^{(2)}_k\chi_l^{(2)})^{1/2}\Bigg)
\end{equation}
are the same.  We will do this by showing that they both agree with the class of the projection 
\begin{equation}\label{bp}
\chi\Bigg(\sum_{i,j=1;k,l=1}^{n_1;n_2}e_{(i,k)(j,l)}\otimes \phi(g_{ij})\otimes (\chi^{(1)}_i\chi^{(2)}_k)^{1/2}(\chi_j^{(1)}\chi^{(2)}_l)^{1/2}\Bigg)
\end{equation}
which we think of as the image of the class of $\phi$ under the homomorphism $K_\delta(T)\to K^0(Y)$ defined using the `product data' (here we think of $e_{(i,j)(k,l)}$ as a matrix unit in $M_{n_1\times n_2}(\C)$).

We first show that the first element in line \eqref{sps} equals the element in line \eqref{bp}.  Indeed, let $\delta_1,...,\delta_n$ be the standard basis of $\C^{n_1}$, and let $\C^{n_1\times n_2}$ denote the vector space with basis $\{\delta_{(i,j)}\mid i\in \{1,...,n_1\},j\in \{1,...,n_2\}\}$ and define a map 
\begin{equation}\label{isometry}
v:\C^{n_1}\otimes V\otimes C(X)\to \C^{n_1\times n_2}\otimes V\otimes  C(X),\quad \delta_i\otimes \xi\otimes  a\mapsto \sum_{j=1}^{n_2}\delta_{(i,j)}\otimes \xi\otimes \chi_j^{1/2}a.
\end{equation}
Then one checks that $v$ is an isometry (of Hilbert $C(X)$-modules) and moreover that
\begin{align*}
v\Bigg(\sum_{i,j=1}^{n_1}&e_{ij}\otimes \phi(g_{ij})\otimes (\chi^{(1)}_i\chi_j^{(1)})^{1/2}\Bigg)v^* \\ & =\Bigg(\sum_{i,j=1;k,l=1}^{n_1;n_2}e_{(i,k)(j,l)}\otimes \phi(g_{ij})\otimes (\chi^{(1)}_i\chi^{(2)}_k)^{1/2}(\chi_j^{(1)}\chi^{(2)}_l)^{1/2}\Bigg).
\end{align*}
As isometries act as the identity on $K$-theory, and as the functional calculus is natural, this gives the equality of the class defined by the the first element in line \eqref{sps} and the element in line \eqref{bp}.  

Similarly, we have the equality of class defined by the second element in line \eqref{sps} with the class of the element 
\begin{equation}\label{pw h}
\chi\Bigg(\sum_{i,j=1;k,l=1}^{n_1;n_2}e_{(i,k)(j,l)}\otimes \phi(h_{kl})\otimes (\chi^{(1)}_i\chi^{(2)}_k)^{1/2}(\chi_j^{(1)}\chi^{(2)}_l)^{1/2}\Bigg).
\end{equation}
Now, for each $i\in \{1,...,n_1\}$ and $k\in \{1,...,n_2\}$, define $f_{ik}$ to be the group element corresponding to $\psi^{(1)}_{U_i}\circ (\psi^{(2)}_{V_k})^{-1}$.  Define a unitary in $M_{n_1\times n_2}(\C)\otimes \mathcal{B}(V)$ by 
\begin{equation}\label{unitary}
u:=\sum_{i,j=1;k,l=1}^{n_1;n_2} e_{(i,i)(k,k)}\otimes \phi(f_{ik}).
\end{equation}
Then one checks that for each $i,k$, $\|\phi(f_{ik})\phi(h_{kl})\phi(f_{jl}^{-1})-\phi(g_{ij})\|<2\delta<1/8(n_1n_2)^2$, whence the difference
\begin{align*}
u\Bigg(\sum_{i,j=1;k,l=1}^{n_1;n_2} & e_{(i,k)(j,l)}\otimes \phi(h_{kl})\otimes (\chi^{(1)}_i\chi^{(2)}_k)^{1/2}(\chi_j^{(1)}\chi^{(2)}_l)^{1/2}\Bigg)\\&-\Bigg(\sum_{i,j=1;k,l=1}^{n_1;n_2}e_{(i,k)(j,l)}\otimes \phi(g_{ij})\otimes (\chi^{(1)}_i\chi^{(2)}_k)^{1/2}(\chi_j^{(1)}\chi^{(2)}_l)^{1/2}\Bigg)
\end{align*}
has norm at most $1/8$.  Using Lemma \ref{chi com} part \eqref{chi lip} and naturality of the functional caclulus, we therefore see that the class of the element in line \eqref{pw h} equals the class of the element in line \eqref{bp} in $K^0(Y)$.  This completes the proof of well-definedness.

We next check that the abelian group homomorphism that we have constructed is a ring homomorphism.  We need to show that for any compact subset $Y$ of $X$ there is a finite subset $T$ of $\Gamma$ containing $S_Y$ and $\delta<\epsilon_Y$ such that if $\phi$ and $\psi$ are in $K_\delta(T)$, then the classes of $p_{\phi\otimes \psi}$ and $\Delta^*(p_\phi\otimes p_\psi)$ are the same in $K^0(Y)$ (here $p_\phi\otimes p_\psi$ defines a class in a matrix algebra over $C(Y\times Y)$, and $\Delta^*$ is the dual map to the diagonal inclusion $\Delta:Y\to Y\times Y$, amplified to matrices as necessary).  We compute: with our usual conventions
\begin{equation}\label{p phi psi}
p_{\phi\otimes \psi}=\chi\Bigg(\sum_{i,j=1}^n e_{ij}\otimes \phi(g_{ij})\otimes\psi(g_{ij})\otimes \chi_i^{1/2}\chi_j^{1/2}\Bigg)
\end{equation}
and 
$$
\Delta^*(p_\phi\otimes p_\psi)=\chi\Bigg(\sum_{i,j,k,l=1}^n e_{(i,j)(k,l)}\otimes \phi(g_{ij})\otimes \psi(g_{kl})\otimes \chi_{i}^{1/2}\chi_j^{1/2}\chi_k^{1/2}\chi_l^{1/2}\Bigg).
$$
The proof is now similar to that of well-definedness.  Indeed, using conjugation by an appropriate isometry analogous to the one in line \eqref{isometry}, the element in line \eqref{p phi psi} can be replaced with 
$$
\chi\Bigg(\sum_{i,j,k,l=1}^n e_{(i,j)(k,l)}\otimes \phi(g_{ij})\otimes \phi(g_{ij})\otimes \chi_{i}^{1/2}\chi_j^{1/2}\chi_k^{1/2}\chi_l^{1/2}\Bigg).
$$
Then up to conjugation by an approximate unitary analogous to the one in line \eqref{unitary}, and an approximation, we can replace the second tensor factor $\phi(g_{ij})$ above with $\phi(g_{kl})$ as long as $T$ is suitably large, and $\delta$ suitably small.  We leave the details to the reader.

For the last statement, the point is that if $\pi$ is an honest representation, the the projection 
$$
p_\pi=\sum_{i,j=1}^n e_{ij}\otimes \pi(g_{ij})\otimes \chi_i^{1/2}\chi_j^{1/2}
$$
represents the flat bundle $E_\pi$ under the usual Serre-Swan correspondence.  We leave the remaining details to the reader.
\end{proof}

Specializing to the case when $X$ is a CW-complex model for the classifying space of $\Gamma$, we have the following definition.

\begin{definition}\label{aas}
The \emph{approximate Atiyah-Segal homomorphism} is the homomorphism 
$$
\alpha:QR(\Gamma)\to LK^0B\Gamma)
$$
arising from specializing the map in Proposition \ref{coass} to the case $X=B\Gamma$.
\end{definition}

\begin{remark}
The classical Atiyah-Segal map is built originally as a map $\alpha:R(\Gamma)\to LK^0B\Gamma)$ for a finite group $\Gamma$, and then extended to a certain completion of the left hand side.  Precisely, one defines the \emph{augmentation ideal} $I(\Gamma)$ as the kernel of the map $\text{dim}:R(\Gamma)\to \Z$ which takes a formal difference of representations to the corresponding difference of dimensions of the underlying Hilbert spaces.  One then shows that $\alpha:R(\Gamma)\to LK^0B\Gamma)$ is continuous when the left hand side is equipped with the $I(\Gamma)$-adic topology (see for example \cite[page 105]{Atiyah:1969ab}) and the right hand side with the inverse limit topology (compare for example \cite[page 46]{Atiyah:1961uq}), and so extends to the completion.

Now, let $\Gamma$ be a possibly infinite group.  One can show that there is still a well-defined `dimension' homomorphism $\text{dim}:QR(\Gamma)\to \Z$, and associated `augmentation ideal' $QI(\Gamma)$ defined as its kernel.  Moreover, one can show that the map $\alpha:QR(\Gamma)\to LK^0\Gamma)$ is continuous for the $QI(\Gamma)$-adic topology: the key point (which is the same in the classical case) is that $\alpha(QI(\Gamma))$ is contained in the reduced $K$-theory $\widetilde{K}^0(B\Gamma)$, and that elements of $\widetilde{K}^0(B\Gamma)$ are multiplicatively nilpotent when restricted to finite subcomplexes (see for example \cite[Corollary 3.1.6 (and page 120 for notation)]{Atiyah:1967kn}).  Thus in the quasi-representation context, we again get a ring homomorphism
$$
\alpha:\widehat{QR(\Gamma)}\to LK^0B\Gamma).
$$
This is a more direct generalization of the classical case, and seems likely to have better properties than the map from Definition \ref{aas} in the presence of torsion.  As we are mainly interested in torsio-free groups, we do not develop this in detail and focus instead on the homomorphism from line \eqref{aas}.  It seems an interesting avenue for further study.
\end{remark}

\section{Controlled $KK$-theory and quasidiagonality}\label{qd sec}

We now switch subjects: Sections \ref{qd sec} through \ref{pair sec} are written purely in the language of $C^*$-algebra $K$-theory; we will only go back to quasi-representations and vector bundles at the end of the paper.   We also make use of significantly more advanced $C^*$-algebra background: for example the theory of absorbing representations on Hilbert modules, completely positive maps, and quasidiagonality.

In this section we recall the definition of the controlled $KK$-theory groups from \cite{Willett:2020aa}: more precisely, we use the variant from \cite[Appendix A.1]{Willett:2020aa} that is tailored to the case that the algebra in the first (`$K$-homology') variable is unital.  We then reformulate the definition in the presence of a quasidiagonality assumption.

Throughout this section, $A$ and $B$ denote separable $C^*$-algebras with $A$ unital.  All Hilbert modules will be over $B$ and countably generated unless otherwise stated.  We write $\LL(E)$, respectively $\K(E)$, for the $C^*$-algebras of adjointable, respectively compact, operators on a Hilbert module $E$.  Our conventions on Hilbert modules follow \cite{Lance:1995ys}.

The next definition recalls the notion of a unitally absorbing representation.  Thomsen \cite[Theorem 2.4]{Thomsen:2000aa} shows that unitally absorbing representations on Hilbert modules always exist (in the separable case).

\begin{definition}\label{ab unital def}
A  unital representation $\sigma:A \to \LL(F)$ of $A$ on a Hilbert $B$-module is \emph{unitally absorbing} (for the pair $(A,B)$) if for any Hilbert $B$-module $E$ and ucp map $\phi:A \to \LL(E)$, there is a sequence $(v_n)$ of isometries in $\LL(E,F)$ such that:
\begin{enumerate}[(i)]
\item $\phi(a)-v_n^*\sigma(a)v_n\in \K(E)$ for all $a\in A$ and $n\in\N$;
\item $\|\phi(a)-v_n^*\sigma(a)v_n\|\to 0$ as $n\to\infty$ for all $a\in A$.
\end{enumerate}
A unital representation $\pi:A\to \LL(E)$ is \emph{strongly unitally absorbing} if it comes with a decomposition as a countably infinite direct sum $(\pi,E)=\bigoplus_{n=1}^\infty(\sigma,F)$, where $(\sigma,F)$ is a unitally absorbing representation.
\end{definition}

Note that a unitally absorbing representation is faithful; as such if $\pi:A\to \mathcal{L}(E)$ is a (strongly) unitally absorbing representation we generally just write $a\in \LL(E)$ for the image of $a\in A$ under this representation to avoid notational clutter. 
\begin{remark}\label{bl2 rem}
Note that if $(\pi,E)$ is strongly unitally absorbing then the decomposition $(\pi,E)=\bigoplus_{n=1}^\infty(\sigma,F)$ from Definition \ref{ab unital def} canonically induces an isomorphism $E\cong \ell^2\otimes F$ (here $\ell^2$ is shorthand for $\ell^2(\N))$ and therefore a canonical inclusion 
\begin{equation}\label{bl2 inc}
\mathcal{B}(\ell^2)=\mathcal{B}(\ell^2)\otimes \C1_F\subseteq \LL(E).
\end{equation}
In this way we identity $\mathcal{B}(\ell^2)$ with a subalgebra of $\LL(F)$ that commutes with the image of $A$.
\end{remark}

\begin{remark}\label{ab unique}
If $(\pi,E)$ is a unitally absorbing representation of $A$ on a Hilbert $B$-module, then $E$ is necessarily isomorphic to the standard Hilbert $B$-module $\ell^2\otimes B$.  Indeed, Thomsen \cite[Theorem 2.4]{Thomsen:2000aa} shows that a unitally absorbing representation on $\ell^2\otimes B$ always exists, and it follows from \cite[Theorem 2.1]{Thomsen:2000aa} that the modules underlying any two unitally absorbing representations are unitarily isomorphic (this is not true of the representations themselves: there one has only approximate unitary equivalence).
\end{remark}

\begin{definition}\label{kku}
Let $\pi:A\to \LL(E)$ be a strongly unitally absorbing representation of $A$ on a Hilbert $B$-module.  Let $X$ be a finite subset of the unit ball $A_1$ of $A$, and let $\epsilon>0$.  For $b\in \LL(E)$, we write $b^{(2)}$ for the diagonal matrix $\begin{psmallmatrix} b & 0 \\ 0 & b \end{psmallmatrix}$ in $M_2(\LL(E))$.

Define\footnote{The ``$u$'' is for ``unital''.} $\mathcal{P}^{\pi,u}_{\epsilon}(X,B)$ to be the set of projections in $M_2(\LL(E))$ satisfying the following conditions:
\begin{enumerate}[(i)]
\item $p-\begin{psmallmatrix} 1 & 0 \\ 0 & 0 \end{psmallmatrix}$ is in $\K(E)$;
\item $\|[p,a^{(2)}]\|<\epsilon$ for all $a\in X$.
\end{enumerate}
Equip $\mathcal{P}^{\pi,u}_{\epsilon}(X,B)$ with the norm topology it inherits from $M_2(\LL(E))$, and define
$$
KK^{\pi,u}_\epsilon(X,B):=\pi_0(\mathcal{P}^{\pi,u}_{\epsilon}(X,B)),
$$
i.e.\ $KK^{\pi,u}_\epsilon(X,B)$ is the set of path components of $\mathcal{P}^{\pi,u}_{\epsilon}(X,B)$.
\end{definition}

With respect to the inclusion in Remark \ref{bl2 rem}, let $s_1,s_2\in \mathcal{B}(\ell^2)\subseteq \LL(E)$ be a pair of isometries that commute with the image of $A$ and satisfy the \emph{Cuntz relation} 
\begin{equation}\label{cun rel}
s_1s_1^*+s_2s_2^*=1;
\end{equation}
we then say that $(s_1,s_2)$ is a \emph{Cuntz pair} in $\mathcal{B}(\ell^2)$.  With notation as in Definition \ref{kku}, define an operation (called \emph{Cuntz sum}) on $KK^{\pi,u}_\epsilon(X,B)$ by 
$$
[p]+[q]:=[s_1ps_1^*+s_2ps_2^*].
$$
It is shown in \cite[Lemma A.5]{Willett:2020aa} that this operation makes $KK^{\pi,u}_\epsilon(X,B)$ into an abelian group with identity element the class of $\begin{psmallmatrix} 1 & 0 \\ 0 & 0 \end{psmallmatrix}$, and moreover that the group operation does not depend on the choice of Cuntz pair $(s_1,s_2)$.  The argument of \cite[Lemma 6.6]{Willett:2020aa} shows that the group $KK^{\pi,u}_\epsilon(X,B)$ does not depend on the choice of strongly unitally absorbing representation up to canonical isomorphism.

We now introduce an alternative formulation: this is different in general, but gives the same group as $KK^{\pi,u}_\epsilon(X,B)$ in the presence of a quasidiagonality assumption that we will introduce later.

\begin{definition}\label{vk}
Let $\pi:A\to \LL(E)$ be a strongly unitally absorbing representation of $A$ on a Hilbert $B$-module.  Let $X$ be a finite subset of the unit ball $A_1$ of $A$, and let $\epsilon>0$.  

Define $\mathcal{P}^{\pi,\K}_{\epsilon}(X,B)$\footnote{The ``$\K$'' is for ``compact''.} to be the set of projections in $\K(E)$ such that $\|[p,a]\|<\epsilon$ for all $a\in X$.  Equip $\mathcal{P}^{\pi,\K}_{\epsilon}(X,B)$ with the norm topology it inherits from $\K(E)$, and define
$$
V^{\pi,\K}_\epsilon(X,B):=\pi_0(\mathcal{P}^{\pi,\K}_{\epsilon}(X,B)),
$$
i.e.\ $V^{\pi,\K}_\epsilon(X,B)$ is the set of path components of $\mathcal{P}^{\pi,\K}_{\epsilon}(X,B)$.
\end{definition}

We next define an operation on the set $V^{\pi,\K}_\epsilon(A,B)$.  We need a definition.

\begin{definition}\label{so plus}
Let $\pi:A\to \LL(E)$ be a strongly unitally absorbing representation of $A$ on a Hilbert $B$-module, and let $\mathcal{B}(\ell^2)\subseteq \LL(E)$ be the canonical inclusion from Remark \ref{bl2 rem}.  Choose two isometries $s_1,s_2\in \mathcal{B}(\ell^2)$ satisfying the Cuntz relation as in line \eqref{cun rel}.  For $[p],[q]\in V^{\pi,\K}_\epsilon(X,B)$, define 
$$
[p]+[q]:=[s_1ps_1^*+s_2qs_2^*].
$$
\end{definition}




The next lemma is standard: we leave the details to the reader (compare \cite[Lemma A.5]{Willett:2020aa} as cited above).
 
\begin{lemma}\label{v op}
The operation from Definition \ref{so plus} makes the set $V^{\pi,\K}_\epsilon(X,B)$ of Definition \ref{vk} into a commutative monoid with identity element the class of the zero projection.  The monoid stricture does not depend on the choice of Cuntz pair $(s_1,s_2)$.  \qed
\end{lemma}

\begin{lemma}\label{k to u}
With notation as in Definitions \ref{kku} and \ref{vk}, there is a well-defined monoid homomorphism determined by  
$$
V^{\pi,\K}_\epsilon(X,B)\to KK^{\pi,u}_\epsilon(X,B),\quad [p]\mapsto \begin{bmatrix} 1 & 0 \\ 0 & p \end{bmatrix}
$$
(here we use the operation on the left hand side defined by Lemma \ref{v op}).
\end{lemma}

\begin{proof}
It is clear that we have a well-defined set-map taking $\mathcal{P}^{\pi,\K}_\epsilon(X,B)$ to $\mathcal{P}^{\pi,u}_\epsilon(X,B)$ given by $p\mapsto\begin{psmallmatrix} 1 & 0 \\ 0 & p \end{psmallmatrix}$.  This map takes homotopies to homotopies, so descends to a set map as in the statement; it moreover clearly takes the identity $[0]$ on the left hand side to the identity $\begin{bsmallmatrix} 1 & 0 \\ 0 & 0 \end{bsmallmatrix}$ on the right hand side.  As the operations on both sides are given by conjugation by Cuntz sums (and we can use the same Cuntz pair of isometries to define them) it is also straightforward to check that the operation is preserved.
\end{proof}

\begin{definition}\label{kkuk}
With $V^{\pi,\K}_\epsilon(X,B)$ as in Definition \ref{vk} equipped with the operation from Lemma \ref{v op}, we let $KK^{\pi,\mathcal{K}}_\epsilon(X,B)$ denote its Grothendieck group.
\end{definition}

\begin{remark}\label{fc1}
Let $I_A$ be the set of all pairs $(X,\epsilon)$, where $X$ is a finite subset of the unit ball of $A$, and $\epsilon>0$.  Define a partial order ``$\leq_u$'' on $I_A$ by stipulating that $(X,\epsilon)\leq_u (Y,\delta)$ if for any separable $C^*$-algebra $B$, and any strongly unitally absorbing representation $\pi$ of $A$ on a Hilbert $B$-module, there is an inclusion 
$$
\mathcal{P}_\delta^{\pi,u}(Y,B)\subseteq \mathcal{P}_\epsilon^{\pi,u}(X,B)
$$
of the sets from Definition \ref{kku}.  In this case, the inclusion induces a canonical \emph{forget control map}
$$
KK^{\pi,u}_\delta(X,B)\to KK^{\pi,u}_\epsilon(X,B).
$$

We may similarly define an order ``$\leq_\K$'' by on $I_A$ by replacing the role of $\mathcal{P}_\epsilon^{\pi,u}(X,B)$ with $\mathcal{P}_\epsilon^{\pi,\K}(X,B)$, and get forget control maps 
$$
KK^{\pi,\K}_\delta(X,B)\to KK^{\pi,\K}_\epsilon(X,B)
$$
in that case similarly.
\end{remark}

Now, using the universal property of the Grothendieck group, the homomorphism of Lemma \ref{k to u} induces a unique group homomorphism 
\begin{equation}\label{k to u gp}
KK^{\pi,\mathcal{K}}_\epsilon(X,B)\to KK^{\pi,u}_\epsilon(X,B).
\end{equation}
In general, this map will be far from being an isomorphism: for example if $A=\C$, then $KK^{\pi,u}_\epsilon(X,B)$ is always isomorphic to the $K$-theory group $K_0(B)$, but if $B\otimes \K$ is projectionless (for example, if $B=C_0(\R^2)$), then $KK^{\pi,\mathcal{K}}_\epsilon(X,B)$ will be zero.  There are also more subtle obstructions coming from failure of quasidiagonality: it may be that $B\otimes \K$ does not contain any projections that approximately commute with the representation of $A$.  For example, let $B=\C$, and let $A$ be the universal $C^*$-algebra generated by a single isometry $s$.  Then if $X=\{s\}$, there cannot be any compact projections that commute very well with $s$.

Our goal in the remainder of this section is to show, however, that the map in line \eqref{k to u gp} is an isomorphism in the presence of an appropriate quasidiagonality assumption.  The following definition is a Hilbert module version of the standard definition of a quasidiagonal representation (compare for example \cite[Definition 2.5]{Dadarlat:2002aa}).

\begin{definition}\label{qd rep}
A representation $\pi:A\to \LL(E)$ of $A$ on a Hilbert $B$-module is \emph{quasidiagonal} if there is an increasing sequence $(e_n)$ of projections in $\K(E)$ that converges strictly to the identity, and such that $\|[\pi(a),e_n]\|\to 0$ for all $a\in A$.
\end{definition}

\begin{example}\label{qd ex}
A separable unital $C^*$-algebra $A$ is said to be \emph{quasidiagonal} if it admits a quasidiagonal representation on the Hilbert $\C$-module (i.e.\ Hilbert space) $\ell^2$ in the sense of Definition \ref{qd rep} above: compare for example \cite[Chapter 7]{Brown:2008qy} for background.   

Let $\pi:A\to \mathcal{B}(\ell^2)$ be a unital representation such that $\pi^{-1}(\mathcal{K}(\ell^2))=\{0\}$.  If $A$ is a quasidiagonal $C^*$-algebra, then this will always be a quasidiagonal representation: see for example \cite[Theorem 7.2.5]{Brown:2008qy}.  On the other hand, for a separable $C^*$-algebra $B$ we may consider the amplified representation $\pi\otimes 1_B:A\to \mathcal{L}(\ell^2\otimes B)$ built from $\pi$.  This will inherit quasidiagonality from $\pi$ as long as $B$ is unital, or more generally, has an approximate unit consisting of projections: indeed, $\K(\ell^2\otimes B)\cong \K(\ell^2)\otimes B$, and if $(e_n)$ is as in the definition of quasdiagonality and $(f_n)$ is a approximate unit for $B$ consisting of projections, then $(e_n\otimes f_n)$ has the properties needed to show quasidiagonality of $\pi\otimes 1_B$.

Note moreover that if $A$ or $B$ is nuclear, the representation $\pi\otimes 1_B$ will be unitally absorbing by Kasparov's Hilbert module version of Voiculescu's theorem \cite[Theorem 6]{Kasparov:1980sp}; if moreover $\pi$ has infinite multiplicity, $\pi\otimes 1_B$ will be strongly unitally absorbing.

In particular, if $A$ is separable, unital, and quasidiagonal, and if $B$ is separable, nuclear and has an approximate unit consisting of projections, then $A$ admits a strongly unitally absorbing and quasidiagonal representation on a Hilbert $B$-module.  This will be our main source of examples.
\end{example}

\begin{remark}\label{sub com rem}
Let $\pi:A\to \LL(E)$ be a quasidiagonal representation of $A$ on a Hilbert $B$-module with $(e_n)$ a sequence of projections in $\K(E)$ as in the definition.  Assume moreover that the representation is strongly unitally absorbing, and let $\mathcal{B}(\ell^2)\subseteq \LL(E)$ be the canonical copy from Remark \ref{bl2 rem}.  We may assume that for each $n$ there is a projection $p_n\in \mathcal{K}(\ell^2)$ such that $p_ne_n=e_n=e_np_n$.  Indeed, as each $e_n$ is compact there exists $p_n\in \K(\ell^2)$ such that $\|p_ne_n-e_n\|<1/(3n+1)$ for each $n$.  Then if $\chi$ is the characteristic function of $(1/2,\infty)$, and we define $f_n:=\chi(p_ne_np_n)$, then one checks that $(f_n)$ still has the properties required by quasidiagonality, and moreover that $p_nf_n=f_np_n=f_n$ for all $n$.
\end{remark}

Finally, we are ready for the main theorem of this section.

\begin{theorem}\label{k to u the}
Let $A$ be a separable unital $C^*$-algebra, let $B$ be a separable $C^*$-algebra, and assume that there exists a strongly unitally absorbing and quasidiagonal representation $\pi:A\to \LL(E)$ of $A$ on a Hilbert $B$-module.  Then the homomorphism
$$
KK^{\pi,\mathcal{K}}_\epsilon(X,B)\to KK^{\pi,u}_\epsilon(X,B)
$$ 
of line \eqref{k to u gp} is an isomorphism.  

Moreover, with notation as in Remark \ref{fc1}, if $(X,\epsilon)\leq_u (Y,\delta)$ and also $(X,\epsilon)\leq_\K(Y,\delta)$ then the diagram  
$$
\xymatrix{KK^{\pi,\mathcal{K}}_\delta(Y,B) \ar[d]^-\cong  \ar[r] & KK^{\pi,\mathcal{K}}_\epsilon(X,B) \ar[d]^-\cong \\
KK^{\pi,u}_\delta(Y,B)  \ar[r] & KK^{\pi,u}_\epsilon(X,B) }
$$
commutes, where the vertical maps are the isomorphisms constructed above, and the horizontal maps are the forget control maps of Remark \ref{fc1}.
\end{theorem}

We need an ancillary lemma.

\begin{lemma}\label{close lem}
With notation as in Definition \ref{kku}, let $p\in \mathcal{P}^{\pi,u}_\epsilon(X,B)$.  Let $\chi$ be the characteristic function of $(1/2,\infty)$  Then there exists $\delta>0$ (depending on $p$) such that if $b\in \LL(E)$ is self-adjoint and satisfies $\|b-p\|<\delta$ and $b-\begin{psmallmatrix} 1 & 0 \\ 0 & 0 \end{psmallmatrix}$, then $\chi(b)$ is in $\mathcal{P}^{\pi,u}_\epsilon(X,B)$, and is moreover in the same path component of $\mathcal{P}^{\pi,u}_\epsilon(X,B)$ as $p$.
\end{lemma}

\begin{proof}
We may assume $\delta<1/2$, and consider the path $p_t:=\chi(t b+(1-t)p)$ for $t\in [0,1]$.  This then connects $\chi(b)$ and $p$, and consists of projections in $\LL(E)$ that satisfy $p_t-\begin{psmallmatrix} 1 & 0 \\ 0 & 0 \end{psmallmatrix}\in \K(E)$ for all $t$.  Moreover, using Lemma \ref{chi com} part \eqref{chi com est}, for any $a\in X$,
$$
\|[p_t,a^{(2)}]\|\leq \frac{\|[tb+(1-t)p,a^{(2)}]\|}{1-2\delta}<\frac{\|b-p\|}{1-2\delta}+\frac{\|[p,a^{(2)}]\|}{1-2\delta}<\frac{\delta}{1-2\delta}+\frac{\|[p,a^{(2)}]\|}{1-2\delta}
$$
As $\|[p,a^{(2)}]\|<\epsilon$ for all $a\in X$, we can ensure that this is also smaller than $\epsilon$ for all $t$ by making $\delta$ small enough.  Hence the homotopy $(p_t)_{t\in [0,1]}$ passes through $\mathcal{P}^{\pi,u}_\epsilon(X,B)$ and we are done.
\end{proof} 

\begin{proof}[Proof of Theorem \ref{k to u the}]
Let $(e_n)$ be a sequence of projections in $\K(E)$ as in the definition of a quasidiagonal representation.  We assume that $(e_n)$ also satisfies the condition in Remark \ref{sub com rem}.  We write $e_n^{(2)}$ for the operator $\begin{psmallmatrix} e_n & 0 \\ 0 & e_n\end{psmallmatrix}$ in $M_2(\K(E))$.

We first show that the homomorphism is surjective, so let $p\in \mathcal{P}^{\pi,u}_\epsilon(X,B)$ be given.  Let $\chi$ be the characteristic function of $(1/2,\infty)$.  Choose any unitary $u:\ell^2\oplus \ell^2\to \ell^2$, so conjugation by $u$ induces an isomorphism $M_2(\LL(E))\cong \LL(E)$ that intertwines the representations of $A$ on $E\oplus E$ and $E$.  We then define a class
\begin{equation}\label{pre p}
[u\chi(e_n^{(2)}pe_n^{(2)})u^*]-[e_n]\in KK^{\pi,\mathcal{K}}_\epsilon(X,B)
\end{equation}
(this makes sense for $n$ large enough by Lemma \ref{chi com}) and claim this maps to $[p]$ under the class in the statement.  We first note that the class in line \eqref{pre p} does not depend on the choice of $u$ (as any two choices are homotopic in the unitary group of $\mathcal{B}(\ell^2\oplus \ell^2,\ell^2)$).  

Let now $p_n\in \K(\ell^2)\subseteq \LL(E)$ be any projection that dominates $e_n$.  As the choice of $u$ does not matter, we may as well assume that $up_n^{(2)}u^*$ is orthogonal to $p_n$; note moreover that this implies that $ue_n^{(2)}u^*$ is orthogonal to $e_n$.

Now, to show that the class in line \eqref{pre p} maps to $[p]$, we need to show that $[p]+\begin{bsmallmatrix} 1 & 0 \\ 0 & e_n\end{bsmallmatrix} =\begin{bsmallmatrix} 1 & 0 \\ 0 & u\chi(e_n^{(2)}pe_n^{(2)})u^*\end{bsmallmatrix}$, or in other words that 
\begin{equation}\label{surj desid}
s_1\begin{pmatrix} 1 & 0 \\ 0 & e_n \end{pmatrix}s_1^*+s_2ps_2^* \quad \text{and}\quad \begin{pmatrix} 1 & 0 \\ 0 & u\chi(e_n^{(2)}pe_n^{(2)})u^*\end{pmatrix}
\end{equation}
are homotopic in $\mathcal{P}_\epsilon^{\pi,u}(X,B)$.  Indeed, we note first that as $p-\begin{psmallmatrix} 1 & 0 \\ 0 & 0 \end{psmallmatrix}\in M_2(\K(E))$, we have that $\chi(e_n^{(2)}pe_n^{(2)})+\begin{psmallmatrix} 1-e_n & 0 \\ 0 & 0 \end{psmallmatrix}\to p$ as $n\to\infty$.  Hence using Lemma \ref{close lem}, the element on the left hand side of line \eqref{surj desid} is homotopic in $\mathcal{P}_\epsilon^{\pi,u}(X,B)$ to 
\begin{equation}\label{surj desid 2}
s_1\begin{pmatrix} 1 & 0 \\ 0 & e_n \end{pmatrix}s_1^*+s_2\Bigg(\begin{pmatrix} 1-e_n & 0 \\ 0 & 0 \end{pmatrix}+\chi(e_n^{(2)}pe_n^{(2)})\Bigg)s_2^*.
\end{equation}

Let now $i:\ell^2\to \ell^2\oplus \ell^2$ be the second factor inclusion, and consider the isometry $u_0:=iu\in \mathcal{B}(\ell^2\oplus \ell^2)\subseteq \L(E\oplus E)=M_2(\LL(E))$.  We claim that there is a unitary isomorphism $v:eE_n\oplus E_n\to eE_n\oplus E_n$ whose restriction to $e_nE\oplus e_nE$ agrees with the restriction of $iu$ to $e_nE\oplus e_nE$.  It suffices to show that 
\begin{enumerate}[(a)]
\item \label{orth dom} the orthogonal complement of $e_nE\oplus e_nE$ in $e_nE\oplus E$ is isomorphic to $E$;
\item \label{orth ran} the orthogonal complement of $u_0(e_nE\oplus e_nE)$ in $e_nE\oplus E$ is isomorphic to $E$.
\end{enumerate}
Indeed, let $p_n\in \K(\ell^2)\subseteq \LL(E)$ denote a compact projection that dominates $e_n$.  For part \eqref{orth dom}, the orthogonal complement is (canonically) isomorphic to $(1-p_n)E\oplus (p_n-e_n)E$.  As $p_n\in \K(\ell^2)$ is finite rank, $(1-p_n)E$ is just isomorphic to $E$.  Hence by Kasparov's stabilization theorem \cite[Theorem 2]{Kasparov:1980sp} (or see \cite[Chapter 6]{Lance:1995ys} for a textbook treatment) and Remark \ref{ab unique}, $(1-p_n)E\oplus (p_n-e_n)E$ is isomorphic to $E$ and we are done with part \eqref{orth dom}.  For part \eqref{orth ran}, consider $iup_n^{(2)}E$, which is a subspace of the second summand of $e_nE\oplus E$.  Again, as $up_n^{(2)}u^*$ is a finite rank projection in $\mathcal{B}(\ell^2)$, the orthogonal complement of this subspace in the second summand is isomorphic to $E$.  Hence the orthogonal complement of the image of $u_0(e_nE\oplus e_nE)$ in $e_nE\oplus E$ is isomorphic to $e_nE\oplus (u(p_n^{(2)}-e_n^{(2)})u^*)E\oplus E$, which is again isomorphic to $E$ by Kasparov's stabilization theorem.  We are done with the claim.

Now, recalling that $e_nE\oplus E\cong B\otimes \ell^2$ by Remark \ref{ab unique} and Kasparov's stabilization theory, we see that the unitary group of $\LL(e_nE\oplus E)$ is connected (even contractible) by the main result of \cite{Cuntz:1987ly}.  We may thus find a continuous path $(v_t)_{t\in [0,1]}$ connecting $v$ and the identity.  Let $j:e_n\oplus E\to E\oplus E$ be the inclusion, and define $w_t:=jv_tj^*\in \LL(E)$.  Consider the path 
$$
t\mapsto s_1\begin{pmatrix} 1 & 0 \\ 0 & e_n \end{pmatrix}s_1^*+s_2\Bigg(\begin{pmatrix} 1-e_n & 0 \\ 0 & 0 \end{pmatrix}+w_t\chi(e_n^{(2)}pe_n^{(2)})w_t^*\Bigg)s_2^*.
$$
This connects the element in line \eqref{surj desid 2} to 
\begin{equation}\label{surj desid 3}
s_1\begin{pmatrix} 1 & 0 \\ 0 & e_n \end{pmatrix}s_1^*+s_2\begin{pmatrix} 1-e_n & 0 \\ 0 & u\chi(e_n^{(2)}pe_n^{(2)})u^* \end{pmatrix}s_2^*
\end{equation}
inside $\mathcal{P}_\epsilon^{\pi,u}(X,B)$ (for $n$ suitably large).  

To complete the proof, we define $r=s_1s_2^*+s_2s_1^*$, and $r_t:=\frac{1}{2}r(1+e^{\pi i t})+\frac{1}{2}(1-e^{\pi i t})$ for $t\in [0,1]$, so $(r_t)_{t\in [0,1]}$ is a path of unitaries connecting $r$ to the identity in $\mathcal{B}(\ell^2)$.  We then consider the path 
$$
s_1\begin{pmatrix} 1 & 0 \\ 0 & 0 \end{pmatrix}s_1^*+r_ts_1\begin{pmatrix} 0 & 0 \\ 0 & e_n\end{pmatrix} s_1^*r_t^*.
$$
One checks that this is a norm continuous path of projections between $s_1\begin{psmallmatrix} 1 & 0 \\ 0 & e_n \end{psmallmatrix}s_1^*$ and $s_1\begin{psmallmatrix} 1 & 0 \\ 0 & 0 \end{psmallmatrix}s_1^*+ s_2\begin{psmallmatrix} 0 & 0 \\ 0 & e_n \end{psmallmatrix}s_2^*$, and moreover that for $n$ suitably large the path 
$$
s_1\begin{pmatrix} 1 & 0 \\ 0 & 0 \end{pmatrix}s_1^*+r_ts_1\begin{pmatrix} 0 & 0 \\ 0 & e_n\end{pmatrix} s_1^*r_t^*+s_2\begin{pmatrix} 1-e_n & 0 \\ 0 & u\chi(e_n^{(2)}pe_n^{(2)})u^* \end{pmatrix}s_2^*
$$
stays within $\mathcal{P}^{\pi,\K}_\epsilon(X,B)$ (we here use our assumption that $ue_n^{(2)}u^*$ is orthogonal to $e_n$).  We thus see that the element in line \eqref{surj desid 3} is homotopic to 
$$
s_1\begin{pmatrix} 1 & 0 \\ 0 & 0 \end{pmatrix}s_1^*+s_2\begin{pmatrix} 0 & 0 \\ 0 & e_n \end{pmatrix}s_2^*+s_2\begin{pmatrix} 1-e_n & 0 \\ 0 & u\chi(e_n^{(2)}pe_n^{(2)}u^* \end{pmatrix}s_2^*.
$$
A final rotation homotopy on the second matrix appearing above shows that this is homotopic to the element on the right hand side of line \eqref{surj desid} above, and we are done with surjectivity.

We now turn to injectivity.  Say then that $p,q\in \mathcal{P}^{\pi,\K}_\epsilon(X,B)$ are such that $[p]-[q]\mapsto 0$, or in other words that $\begin{psmallmatrix} 1 & 0 \\ 0 & p \end{psmallmatrix}$ and $\begin{psmallmatrix} 1 & 0 \\ 0 & q \end{psmallmatrix}$ are homotopic in $\mathcal{P}^{\pi,u}_\epsilon(X,B)$.  Let $(p_t)_{t\in [0,1]}$ be a homotopy connecting them.  As $p_t-\begin{psmallmatrix} 1 & 0 \\ 0 & 0 \end{psmallmatrix}\in M_2(\K(E))$ for all $t$, we see that for $n$ suitably large, $(\chi(e_n^{(2)}p_te_n^{(2)}))_{t\in [0,1]}$ is a homotopy between $\begin{psmallmatrix} e_n & 0 \\ 0 & \chi(e_npe_n) \end{psmallmatrix}$ and $\begin{psmallmatrix} e_n & 0 \\ 0 & \chi(e_nqe_n) \end{psmallmatrix}$ in $\mathcal{P}^{\pi\oplus\pi,\K}_\epsilon(X,B)$.  Using Lemma \ref{close lem} this implies that for $n$ suitably large, $\begin{psmallmatrix} e_n & 0 \\ 0 & p \end{psmallmatrix}$ and $\begin{psmallmatrix} e_n & 0 \\ 0 & q\end{psmallmatrix}$ are homotopic in $\mathcal{P}^{\pi\oplus\pi,\K}_\epsilon(X,B)$.  From here it is straightforward to deduce that $s_1ps_1^*+s_2e_ns_2^*$ and $s_1qs_1^*+s_2e_ns_2^*$ are homotopic in $\mathcal{P}^{\pi,\K}_\epsilon(X,B)$ or in other words that $[p]+[e_n]=[q]+[e_n]$ so we are done.

The statement on forget control maps is straightforward: we leave this to the reader.
\end{proof}

\section{Murray von Neumann equivalence}\label{mvn sec}

In this section, we give a reformulation of the equivalence relation on the controlled $KK$-groups $KK^{\pi,\K}_\epsilon(X,B)$ of Definition \ref{kkuk} in terms of Murray von Neumann equivalence (although at the price of changing the parameter $\epsilon$).

Throughout this section, $A$ and $B$ denote separable $C^*$-algebras with $A$ unital.  All Hilbert modules will be over $B$ and countably generated unless otherwise stated.  We write $\LL(E)$, respectively $\K(E)$, for the $C^*$-algebras of adjointable, respectively compact, operators on a Hilbert module $E$.  We will work mainly with strongly unitally absorbing representations of $A$ on Hilbert $B$-modules as in Definition \ref{ab unital def}; as in Remark \ref{bl2 rem} above, this structure induces a canonical inclusion $\mathcal{B}(\ell^2)\subseteq \LL(E)$.  We also fix two isometries $s_1,s_2\in \mathcal{B}(\ell^2)$ satisfying the Cuntz relation $s_1s_1^*+s_2s_s^*=1$ as in line \eqref{cun rel}.

\begin{definition}\label{kk mvn}
Let $\pi:A\to \LL(E)$ be a strongly unitally absorbing representation of $A$ on a Hilbert $B$-module, let $\mathcal{B}(\ell^2)\subseteq \LL(E)$ be the associated inclusion as in Remark \ref{bl2 inc}, and let $s_1,s_2\in \mathcal{B}(\ell^2)$ satisfying the Cuntz relation $s_1s_1^*+s_2s_s^*=1$ as in line \eqref{cun rel}.  

Let $\mathcal{P}_\epsilon^{\pi,\K}(X,B)$ be as in Definition \ref{vk} for some finite subset $X$ of the unit ball of $A$ and $\epsilon>0$.  For $p,q\in \mathcal{P}_\epsilon^{\pi,\K}(X,B)$, write $p\sim q$ if there exist $r\in \mathcal{P}_\epsilon^{\pi,\K}(X,B)$ and a partial isometry $v\in \K(E)$ such that $\|[v,a]\|<\epsilon$ for all $a\in X\cup X^*$ and such that 
$$
v^*v=s_1ps_1^*+s_2rs_2^* \quad \text{and}\quad vv^*=s_1qs_1^*+s_2rs_2^*.
$$
\end{definition}

\begin{lemma}\label{con kk set up}
With notation as in Definition \ref{kk mvn}, $\sim$ is an equivalence relation on $\mathcal{P}_{\epsilon}^{\pi,\K}(X,B)$ that does not depend on the choice of $s_1$ and $s_2$.  
\end{lemma}

\begin{proof}
Assume $p_1,q,p_2\in \mathcal{P}_\epsilon^{\pi,\K}(X,B)$ are such that $p_1\sim q$ and $q\sim p_2$.  Then there exist $v,w\in \K(E)$ and $r_1,r_2\in \mathcal{P}_\epsilon^{\pi,\K}(X,B)$ such that $\|[v,a]\|<\epsilon$ and $\|[w,a]\|<\epsilon$ for any $x\in X\cup X^*$, and so that  $v^*v=s_1p_1s_1^*+s_2r_1s_2^*$, $vv^*=s_1qs_1^*+s_2r_1s_2^*$, $w^*w=s_1qs_1^*+s_2r_2s_2^*$, and $ww^*=s_1p_2s_1^*+s_2r_2s_2^*$.  We define 
$$
u_l:=s_1s_1^*s_2^*+s_2s_1s_1^*s_1^*+s_2s_2s_1s_2^*s_1^*+s_2s_2s_2s_2^*s_2^*
$$
and
$$
u_r^*:=s_1s_1^*s_1^*+s_2s_2s_1s_2^*s_1^*+s_2s_1s_1^*s_2^*+s_2s_2s_2s_2^*s_2^*,
$$
which ones checks are unitary elements of $\mathcal{B}(\ell^2)$.  We then define 
$$
t:=u_l(s_1vs_1^*+s_2ws_2^*)u_r.
$$
Note that as $s_1$ and $s_2$ (and therefore $u_l$ and $u_r$) commute with $A$, we have that for any $a\in X\cup X^*$, 
$$
\|[t,a]\|=\|s_1[v,a]s_1^*+s_2[w,a]s_2^*\|=\max\{\|[v,a]\|,\|[w,a]\|\}<\epsilon,
$$
where the second equality uses orthogonality of $s_1[v,a]s_1^*$ and $s_2[w,a]s_2^*$.  One then computes that 
$$
t^*t=s_1p_1s_1^*+s_2(s_1qs_1^*+s_2(s_1rs_1^*+s_2r_2s_2^*)s_2^*)s_2^*
$$
and that 
$$
tt^*=s_1p_2s_1^*+s_2(s_1qs_1^*+s_2(s_1rs_1^*+s_2r_2s_2^*)s_2^*)s_2^*,
$$
which completes the proof that $\sim$ is an equivalence relation.

If $(t_1,t_2)$ is another pair of Cuntz isometries in $\mathcal{B}(\ell^2)$, we note that $u:=s_1t_1^*+s_2t_2^*$ is a unitary element that conjugates each $t_i$ to $s_i$ (and commutes with $A$).  Conjugating by $u$ can thus be used to translate between the `$s_i$-equivalence relation' and the `$t_i$-equivalence relation': we leave the details to the reader.
\end{proof}

Thus we may make the following definition.

\begin{definition}\label{mvn mon}
With notation as in Definition \ref{kk mvn}, we define $V^{\pi,mvn}_\epsilon(X,B)$ to be the quotient set $\mathcal{P}^{\pi,\K}_\epsilon(X,B)/\sim$.
\end{definition}

The next lemma uses similar ideas to the proof of Lemma \ref{con kk set up} (compare also \cite[Lemma 4.7]{Willett:2020aa}).  We leave the proof to the reader.

\begin{lemma}\label{mvn op}
With notation as in Definition \ref{kk mvn}, the operation 
$$
[p]+[q]:=[s_1ps_1^*+s_2qs_2^*]
$$
defines a commutative monoid operation on $\mathcal{P}^{\pi,\K}_\epsilon(X,B)/\sim$ that does not depend on the choice of Cuntz isometries $s_1$ and $s_2$. \qed
\end{lemma}

\begin{definition}\label{kk mvn gp}
With notation as in Definitions \ref{vk} and \ref{kk mvn}, define $KK_{\epsilon}^{\pi,mvn}(X,B)$ to be the Grothendieck group of $\mathcal{P}^{\pi,\K}_\epsilon(X,B)/\sim$.
\end{definition}

\begin{remark}\label{fc2}
Just as in Remark \ref{fc1} above, if $\mathcal{P}^{\pi,\K}_\epsilon(X,B)\subseteq \mathcal{P}^{\pi,\K}_\delta(Y,B)$, then there is a canonical \emph{forget control homomorphism}
$$
KK^{\pi,mvn}_\epsilon(X,B)\to KK^{\pi,mvn}_\delta(Y,B).
$$
\end{remark}

Our goal in the remainder of this section is to show that the groups $KK^{\pi,\mathcal{K}}_\epsilon(X,B)$ and $KK_{\epsilon}^{\pi,mvn}(X,B)$  of Definitions \ref{kkuk} and \ref{kk mvn gp} are essentially the same.  Unfortunately, this is only true `in the limit', so the actual statements are a little more technical than this.  

We proceed via several lemmas.

\begin{lemma}\label{equiv to comp}
With notation as in Definition \ref{kk mvn}, if $p\in \mathcal{P}_\epsilon^{\pi,\K}(X,B)$ and $s\in \mathcal{B}(\ell^2)$ is an isometry, then $p$ and $sps^*$ are in the same path component of $\mathcal{P}_\epsilon^{\pi,\K}(X,B)$.
\end{lemma}

\begin{proof}
Let $(q_n)$ be an increasing sequence of finite rank projections in $\mathcal{B}(\ell^2)\subseteq \LL(E)$ that converge in the strong operator topology to the identity.  As operators on $\LL(E)$, the sequence $(q_n)$ converges strictly to the identity.  Hence Lemma \ref{close lem} implies that for $n$ suitably large, if $\chi$ is the characteristic function of $(1/2,\infty)$, then $\chi(q_npq_n)$ is in $\mathcal{P}_\epsilon^{\pi,\K}(X,B)$, and also path connected to $p$ in this space; similarly, $s_1\chi(q_npq_n)s_1^*$ is connected to $s_1ps_1^*$ for $n$ large enough; fix an $n$ satisfying these conditions.  On the other hand, there is a (norm continuous) path of isometries $(s_t)_{t\in [0,1]}$ in $\mathcal{B}(\ell^2)$ such that $s_0=s$ and such that $s_1$ acts as the identity on $q_n\ell^2$.  Then the path $s_t\chi(q_npq_n)s_t^*$ connects $s\chi(q_npq_n)s^*$ to $\chi(q_npq_n)$ in $\mathcal{P}_\epsilon^{\pi,\K}(X,B)$, and we are done.
\end{proof}

\begin{lemma}\label{mvn to hom}
With notation as in Definition \ref{vk}, assume $X$ is self-adjoint (i.e.\ an element $a$ is in $X$ if and only if $a^*$ is in $X$).  Let $p,q\in \mathcal{P}_\epsilon^{\pi,K}(X,B)$, and assume there exists $v\in \LL(E)$ such that $v^*v=p$, $vv^*=q$ and $\|[v,a]\|<\epsilon$ for all $a\in X$.  Then there exists a path in $\mathcal{P}_{6\epsilon}^{\pi,K}(X,B)$ that connects $p$ to $q$.
\end{lemma}

\begin{proof}
Let $s_1$ and $s_2$ be the Cuntz isometries defining the operation on $V^{\pi,\K}_{\epsilon}(X,B)$.  Using Lemma \ref{equiv to comp}, we see that $p$ is connected in $\mathcal{P}_\epsilon^{\pi,\K}(X,B)$ to $s_1ps_1^*$, and similarly for $q$.  We may thus replace $p$ and $q$ with $s_1ps_1^*$ and $s_2qs_2^*$ respectively.  Define\footnote{It may be more intuitive to use $s_1$ and $s_2$ to identity $\LL(E)$ with $2\times 2$ matrices over itself via the correspondence $\begin{psmallmatrix} a & b \\ c & d \end{psmallmatrix}\mapsto s_1as_1^*+s_1bs_2^*+s_2cs_1^*+s_2ds_2^*$; having done this, $u$ corresponds to the matrix $\begin{psmallmatrix} v & 1-q \\ 1-p & v^* \end{psmallmatrix}$.} now
$$
u:=s_1vs_1^*+s_1(1-q)s_2^*+s_2(1-p)s_1^*+s_2v^*s_2^*.
$$
Then one checks that $u$ is unitary, and that $u(s_1ps_1^*)u^*=s_1qs_1^*$.  It thus suffices to show that there is a path of unitaries $(u_t)_{t\in [0,1]}$ in $\LL(E)$ such that $u_0=s_1s_2^*+s_2s_1^*$ and $u_1=u$ and such that $\|[u_t,a]\|<6\epsilon$ for all $a\in X$ (note that $u_0$ is connected to the identity via unitaries that commute with $A$).  For $t\in [0,1]$, define 
$$
u_t:=s_1tvs_1^*+s_1(1-(1-\sqrt{1-t^2})q)s_2^*+s_2(1-(1-\sqrt{1-t^2})p)s_1^*+s_2tv^*s_2^*.
$$
Direct computations (we use here that $X$ is self-adjoint) show that this has the claimed properties.
\end{proof}

The next corollary follows directly from Lemma \ref{mvn to hom}.

\begin{corollary}\label{mvn to he map}
Let $X$ be a self-adjoint subset of the unit ball of $A$ and let $\epsilon>0$.  With notation as in Definitions \ref{kkuk} and \ref{kk mvn gp}, the identity map on cycles induces a well-defined homomorphism
$$
\psi_{(X,\epsilon)}:KK_{\epsilon}^{\pi,mvn}(X,B)\to KK_{6\epsilon}^{\pi,\K}(X,B). \eqno\qed
$$
\end{corollary}

We now build a map in the opposite direction.  The following lemma is an analogue of \cite[Lemma 4.4.2]{Willett:2021te}; we give details for the reader's convenience.

\begin{lemma}\label{lip path}
With notation as in Definition \ref{vk}, assume that $p_0$ and $p_1$ are elements in the same path component of $\mathcal{P}^{\pi,\K}_\epsilon(X,B)$.  Then there exists $r\in \mathcal{P}^{\pi,\K}_\epsilon(X,B)$ and a unitary $u\in \mathcal{B}(\ell^2)$ such that $s_1p_0s_1^*+s_2rs_2^*$ and $u(s_1p_1s_1^*+s_2rs_2^*)u^*$ are connected in $\mathcal{P}^{\pi,\K}_{2\epsilon}(X,B)$ via a $2$-Lipschitz path parametrized by $[0,1]$.
\end{lemma}

\begin{proof}
Let $(p_t)_{t\in [0,1]}$ be a path in $\mathcal{P}^{\pi,\K}_\epsilon(X,B)$ that connects $p_0$ and $p_1$.  Using compactness of $[0,1]$, there are points $0=t_0<t_1<t_2<\cdots <t_n=1$ in $[0,1]$ such that $\|p_{t_i}-p_{t_{i-1}}\|<\frac{1}{4}$ for all $i\in \{1,...,n\}$.  Write $p^{(i)}:=p_{t_i}$ for each $i$ and define  
$$
r:=\sum_{i=0}^{n-3} s_2^is_1p^{(i+1)}s_1^*(s_2^*)^{i}+s_2^{n-2}p^{(n-1)}(s_2^*)^{n-2}.
$$
We claim this has this properties in the statement.  We leave the direct checks that $r$ is in $\mathcal{P}^{\pi,\K}_\epsilon(X,B)$ to the reader.  Let $\chi$ be the characteristic function of $(1/2,\infty)$ and for each $i\in \{0,...,n-1\}$ and $s\in[0,1]$ define 
$$
p^{(i)}_s:=\chi((1-s)p^{(i)}+sp^{(i+1)}).
$$
For each $i\in \{0,...,n-1\}$, $(1-s)p^{(i)}+sp^{(i+1)}$ is a self-adjoint contraction that is within distance $1/4$ of a projection.  It thus has spectrum contained in $[0,1/4)\cup (3/4,1]$, and so each $p^{(i)}_s$ is a well defined projection in $\K(E)$.  Moreover, Lemma \ref{chi com} implies that each $p^{(i)}_s$ is in $\mathcal{P}^{\pi,\K}_{2\epsilon}(X,B)$ and the path $s\mapsto p^{(i)}_s$ is $2$-Lipschitz.  It follows that the path 
$$
s\mapsto s_1p^{(0)}_ss_1^*+\sum_{i=1}^{n-2} s_2^is_1p^{(i)}_ss_1^*(s_2^*)^{i}+s_2^{n-1}p^{(n-1)}_s(s_2^*)^{n-1}
$$
is $2$-Lipschitz, passes through $\mathcal{P}^{\pi,\K}_{2\epsilon}(X,B)$ and connects $s_1ps_1^*+s_2rs_2^*$ to 
\begin{equation}\label{halfway}
\sum_{i=0}^{n-1}s_2^is_1p^{(i+1)}s_1^*(s_2^*)^i+s_2^{n-1}p^{(n)}(s_2^*)^{n-1}.
\end{equation}
Define now 
$$
u^*:=\sum_{i=0}^{n-2} s_2^{i+1}s_1s_1^*(s_2^*)^i+s_1(s_2^*)^{n-1}.
$$
One checks that conjugating the element in line \eqref{halfway} by $u^*$ gives $s_1p_1s_1^*+s_2rs_2^*$, so we are done.
\end{proof}

Variants of the following lemma are well-known.  The idea of the proof is essentially the same as \cite[Lemma 4.8]{Willett:2021te}.

\begin{lemma}\label{-1/2 com}
Let $C$ be a $C^*$-algebra, let $\delta>0$, and let $P_\delta$ be the collection of positive elements $x\in C$ that satisfy $x\geq \delta$.  The following hold:
\begin{enumerate}[(i)]
\item \label{mh com} for any $c\in C$, 
$$
\|[c,x^{-1/2}]\|\leq \frac{\|[c,x]\|}{2\delta^{3/2}};
$$
\item \label{mh lip} the map $P_\delta\to C$ defined by $x\mapsto x^{-1/2}$ is $1/(2\delta^{3/2})$-Lipschitz.
\end{enumerate}
\end{lemma}

\begin{remark}
The estimates is in some sense optimal: consider for example the commutator of $c:=\begin{psmallmatrix} 0 & 1 \\ 1 & 0 \end{psmallmatrix}$ with matrices of the form $\begin{psmallmatrix} \delta & 0 \\ 0 & \delta+\epsilon \end{psmallmatrix}$ in $M_2(\C)$.
\end{remark}

\begin{proof}[Proof of Lemma \ref{-1/2 com}]
We just give the proof of part \eqref{mh com}; part \eqref{mh lip} is very similar (or see  \cite[Lemma 4.8]{Willett:2021te}, which contains exactly this statement).  For any $t>0$ we have the integral identity
$$
t^{-1/2}=\frac{2}{\pi}\int_0^\infty \frac{1}{\lambda^2+t}d\lambda.
$$
As $x$ is bounded below, we thus have a norm convergent integral
$$
x^{-1/2}=\frac{2}{\pi}\int_0^\infty (\lambda^2+x)^{-1}d\lambda.
$$
Hence for any $c\in C$, 
\begin{equation}\label{root com}
\|[c,x^{-1/2}]\|\leq \frac{2}{\pi}\int_0^\infty \|[c, (\lambda^2+x)^{-1}]\|d\lambda.
\end{equation}
Using the resolvent identity 
$$
[c, (\lambda^2+x)^{-1}]=(\lambda^2+x)^{-1}[x,c](\lambda^2+x)^{-1}
$$
and the fact that $x\geq \delta$, we have that 
$$
\|[c, (\lambda^2+x)^{-1}]\|\leq \|[x,c]\|(\lambda^2+\delta)^{-2},
$$
and so from line \eqref{root com} we have 
$$
\|[c,x^{-1/2}]\|\leq \frac{2\|[x,c]\|}{\pi}\int_0^\infty \frac{1}{(\lambda^2+\delta)^{-2}}d\lambda.
$$
Computing, the integral on the right equals $\frac{\pi}{4}\delta^{-3/2}$, which gives the desired inequality.
\end{proof}

Variants of the following lemma are again well-known: it is a version of (for example) \cite[Proposition 2.2.6]{Rordam:2000mz} with care paid to commutator estimates.

\begin{lemma}\label{uni com}
Let $p$ and $q$ be projections in a unital $C^*$-algebra $C$ that satisfy $\|p-q\|\leq 1/4$.  Then there is a unitary $u\in C$ such that $upu^*=q$.  If moreover $c\in C$ satisfies $\|[c,p]\|<\epsilon$ and $\|[c,q]\|<\epsilon$ for some $\epsilon>0$ then 
$$
\|[u,c]\|< 15\epsilon.
$$
\end{lemma}

\begin{proof}
We define $x:=(1-q)(1-p)+qp$.  Note then that 
\begin{equation}\label{x and 1}
\|1-x\|=\|(1-2q)(q-p)\|\leq \|q-p\|\leq 1/4.
\end{equation}
Hence $x$ is invertible and we may thus define $u:=x(x^*x)^{-1/2}$.  We claim this has the properties in the statement.  

One sees that $upu^*=q$ just as in the proof of \cite[Proposition 2.2.6]{Rordam:2000mz}: we leave this to the reader.  For the commutator estimate, we note that 
\begin{align}\label{x com}
\|[x,c]\| & =\|[1-x,c]\|=\|[(1-2q)(q-p),c]\| \nonumber  \\ & \leq \|1-2q\|\|[q-p,c]\|+\|[1-2q,c]\|\|q-p\|  <(5/2)\epsilon.
\end{align}
Similarly, $\|[x^*,c]\|< (5/2)\epsilon$, whence 
\begin{equation}\label{v*v com}
\|[x^*x,c]\|< 5\epsilon.
\end{equation}
On the other hand, note that line \eqref{x and 1} implies that 
\begin{equation}\label{x norm}
\|x\|\leq 5/4
\end{equation}
whence $\|x^*x-1\|\leq \|x^*(x-1)\|+\|x^*-1\|\leq 9/16$ and so 
\begin{equation}\label{x*x low}
x^*x\geq 7/16.
\end{equation}
From line \eqref{x*x low}, Lemma \ref{-1/2 com}, and line \eqref{v*v com}, we thus see that 
\begin{equation}\label{x*x rt com}
\|(x^*x)^{-1/2},c]\|<\Bigg(\frac{16}{7}\Bigg)^{3/2}\frac{1}{2}\cdot 5\epsilon
\end{equation}
We finally estimate 
$$
\|[u,c]\|\leq \|[c,x]\|\|(x^*x)^{-1/2}\|+\|[c,(x^*x)^{-1/2}]\|\|x\|.
$$
Line \eqref{x*x low} implies that $\|(x^*x)^{-1/2}\|\leq (7/16)^{-1/2}$, so the above plus lines \eqref{x com}, \eqref{x norm} and \eqref{x*x rt com} gives
$$
\|[u,c]\|< \frac{5}{2}\epsilon\cdot \Bigg(\frac{16}{7}\Bigg)^{1/2}+\Bigg(\frac{16}{7}\Bigg)^{3/2}\frac{5}{2}\epsilon\cdot \frac{5}{4},
$$
which implies the given estimate.
\end{proof}

The following lemma is an analogue of \cite[Proposition 4.4.4]{Willett:2021te}; we again give details for the reader's convenience.

\begin{lemma}\label{lip path to ue}
Let $L>0$.  With notation as in Definition \ref{vk}, let $p_0,p_1\in \mathcal{P}^{\pi,\K}_\epsilon(X,B)$ be connected by an $L$-Lipschitz path parametrized by $[0,1]$.  Then there exists a unitary $u$ in the unitization $\K(E)^+$ of $\K(E)$ such that $upu^*=q$, and with the property that $\|[u,a]\|<15(4L+1)\epsilon$ for all $a\in X$.
\end{lemma}

\begin{proof}
Let $(p_t)_{t\in [0,1]}$ be an $L$-Lipschitz path of projections in $\mathcal{P}^{\pi,\K}_\epsilon(X,B)$ connecting $p_0$ and $p_1$.  As the path is $L$-Lipschitz, we may find points $0=t_0<t_1<\cdots <t_n=1$ with $\|p_{t_i}-p_{t_{i+1}}\|\leq 1/4$ for all $i\in \{0,...,n-1\}$, and $n\leq 4L+1$.  From Lemma \ref{uni com}, for each $i\in \{0,...,n-1\}$ there exists a unitary $u_i$ in the unitization of $\K(E)$ such that $u_ip_{t_i}u_i^*=p_{t_{i+1}}$ and such that $\|[u_i,a]\|<15\epsilon$ for all $a\in X$.  Define now $u:=u_{n-1}u_{n-2}\cdots u_1u_0$.  Then $u$ is a unitary in $\K(E)^+$, $up_0u^*=p_1$ and for any $a\in X$,
$$
\|[u,a]\|=\Bigg\|\sum_{i=0}^n \Bigg(\prod_{0\leq j<i}u_j\Bigg)[u_i,a]\Bigg(\prod_{i<j\leq n}u_j\Bigg)\Bigg\|<15n\epsilon.
$$
As $n\leq 4L+1$, this gives the desired estimate.
\end{proof}

The next corollary is straightforward from Lemmas \ref{lip path} and \ref{lip path to ue}: we leave the details to the reader.

\begin{corollary}\label{he to mvn map}
With notation as in Definitions \ref{kkuk} and \ref{kk mvn gp}, the identity map on cycles induces a well-defined homomorphism
$$
\phi_{(X,\epsilon)}:KK_{\epsilon}^{\pi,\K}(X,B)\to KK_{135\epsilon}^{\pi,mvn}(X,B). \eqno\qed
$$
\end{corollary}

The following is now the main result of this section.  As the maps in the statement are all induced by the identity on cycles, the proof is immediate.

\begin{proposition}\label{mvn vs hom}
Let $A$ be a separable unital $C^*$-algebra, let $B$ be a separable $C^*$-algebra, let $X$ be a finite self-adjoint subset of the unit ball of $A$, and let $\epsilon>0$.  Let $\pi:A\to \LL(E)$ be a strongly unitally absorbing representation of $A$ on a Hilbert $B$-module.  With notation as in Corollaries \ref{mvn to he map} and \ref{he to mvn map}, the diagrams
$$
\xymatrix{ KK_{\epsilon}^{\pi,mvn}(X,B) \ar[rr] \ar[dr]^-{\psi_{(X,\epsilon)}} & & K_{810\epsilon}^{\pi,mvn}(X,B) \\
& KK_{6\epsilon}^{\pi,\K}(X,B) \ar[ur]^-{\phi_{(X,6\epsilon)}} & }
$$
and 
$$
\xymatrix{ & KK_{135\epsilon}^{\pi,mvn}(X,B) \ar[dr]^-{\psi_{(X,M\epsilon)}} &  \\
KK_{\epsilon}^{\pi,\K}(X,B) \ar[rr] \ar[ur]^-{\phi_{(X,\epsilon)}} & & K_{810\epsilon}^{\pi,\K}(X,B) }.
$$
commute (here unlabelled maps are the forget control homomorphisms of Remarks \ref{fc1} and \ref{fc2}).\qed
\end{proposition}

\section{Almost multiplicative ucp maps}\label{ucp sec}

In this section, we give a reformulation of the controlled $KK$-groups $KK^{\pi,mvn}_\epsilon(X,B)$ of Definition \ref{kk mvn gp} in terms of almost multiplicative ucp maps(although at the price of changing the parameter $\epsilon$).

Throughout this section, $A$ and $B$ denote separable $C^*$-algebras with $A$ unital.  All Hilbert modules will be over $B$ and countably generated unless otherwise stated.  We write $\LL(E)$, respectively $\K(E)$, for the $C^*$-algebras of adjointable, respectively compact, operators on a Hilbert module $E$.  We will work mainly with strongly unitally absorbing representations of $A$ on Hilbert $B$-modules as in Definition \ref{ab unital def}; as in Remark \ref{bl2 rem} above, this structure induces a canonical inclusion $\mathcal{B}(\ell^2)\subseteq \LL(E)$.  We also fix two isometries $s_1,s_2\in \mathcal{B}(\ell^2)$ satisfying the Cuntz relation $s_1s_1^*+s_2s_s^*=1$ as in line \eqref{cun rel}.

We next need to recall some well-known facts about `small' Hilbert $B$-modules.

\begin{lemma}\label{hm equiv}
For a Hilbert $B$-module $F$, the following are equivalent:
\begin{enumerate}[(i)]
\item \label{proj im} there exists a Hilbert $B$-module $E$ and a projection $p\in \K(E)$ such that $F$ is isomorphic to $pE$;
\item \label{id in} the identity operator on $F$ is contained in $\K(F)$;
\item \label{all com} for any Hilbert $B$-module $E$, $\LL(F,E)$ equals $\K(F,E)$.
\end{enumerate}
\end{lemma}

\begin{proof}
To see \eqref{proj im}$\Rightarrow$\eqref{id in}, we note that $p$ is in $\K(pE)$ and is the identity.  To see \eqref{id in}$\Rightarrow$\eqref{all com}, we note that if $b\in \LL(F,G)$, then $b=b\circ \text{id}_F$; as $\text{id}_F$ is compact $b$ is too.  Finally, \eqref{all com}$\Rightarrow$\eqref{proj im} as we may take $F=E$, whence $p=\text{id}_F$ is compact by \eqref{all com}.
\end{proof}

\begin{definition}\label{com hm}
We say a Hilbert module $F$ is \emph{compact} if it satisfies the equivalent conditions in Lemma \ref{hm equiv}.
\end{definition}

\begin{remark}\label{hm fg rem}
The conditions in Lemma \ref{hm equiv} are also equivalent to other well-known conditions on Hilbert $B$-modules: for example they are equivalent to   $F$ being algebraically finitely generated, and that $F$ is of the form $pA^n$ for some $n$ and $p\in \K(A^n)$.  We will not use this, but see for example \cite[Theorem 8.1.27]{Blecher:2004aa} for more background.  We could therefore just say `finitely generated' instead of `compact'; we choose the latter terminology as it is more suggestive of the condition we actually use.   
\end{remark}

Compare \cite[Definition 3.2]{Willett:2024aa} for the following.

\begin{definition}\label{alm mult def}
Let $\mathcal{A}$ be a $*$-algebra, and let $C$ be a $C^*$-algebra.  Let $X$ be a finite subset of $\mathcal{A}$, and let $\epsilon>0$.  A linear map $\phi:\mathcal{A}\to C$ is said to be an \emph{$(X,\epsilon)$-$*$-homomorphism} if 
\begin{equation}\label{am con}
\|\phi(a^*)-\phi(a)^*\|<\epsilon \quad \text{and}\quad \|\phi(ab)-\phi(a)\phi(b)\|<\epsilon
\end{equation}
for all $a,b\in X$.
\end{definition}

\begin{remark}
In this section, we will only apply Definition \ref{alm mult def} to ucp maps.  These are automatically adjoint preserving, so the condition ``$\|\phi(a^*)-\phi(a)^*\|<\epsilon$'' appearing in Definition \ref{alm mult def} is vacuous in that case; we include it as we will want to work in more generality later.
\end{remark}

\begin{definition}\label{eps ucp}
Let $A$ be a separable unital $C^*$-algebra, and let $B$ be a separable $C^*$-algebra.  Let $\epsilon>0$ and let $X$ be a finite subset of the unit ball of $A$.  We write $\emu_{\epsilon}(X,B)$ for the collection\footnote{It is not necessarily a set.  One can get around this by instead considering all compact submodules of some suitably large Hilbert $B$-module, such as $B\otimes \ell^2$.} of all ucp $(X,\epsilon)$-$*$-homorphisms $\phi:A\to \K(F)$, where $F$ is a compact Hilbert $B$-module in the sense of Definition \ref{com hm}.  Given $\phi\in \emu_{\epsilon}(X,B)$, we will generally write $F_\phi$ for the associated module if we need to disambiguate between modules associated to different maps.

For $\phi,\psi\in \emu_{\epsilon}(X,B)$, we define their \emph{block sum} $\phi\oplus \psi:A\to \LL(F_\phi\oplus F_\psi)$ by $(\phi\oplus \psi)(a):=\phi(a)\oplus\psi(a)$.
\end{definition}

\begin{definition}\label{delta sim}
Let $\epsilon>0$ and let $X$ be a finite subset of the unit ball of $A$. Let $\emu_{\epsilon}(X,B)$ be as in Definition \ref{eps ucp}.  For $\phi,\psi\in \emu_{\epsilon}(X,B)$ we write $\phi\sim\psi$ if there exists $\theta\in \emu_{\epsilon}(X,B)$ and a unitary $u\in \LL(F_\phi\oplus F_\theta,F_\psi\oplus F_\theta)$ such that $\|u(\phi(a)\oplus\theta(a))u^*-(\psi(a)\oplus\theta(a))\|<\epsilon$ for all $a\in X$.  
\end{definition}

\begin{lemma}\label{ucp equiv}
With notation as in Definitions \ref{eps ucp} and \ref{delta sim}, $\sim$ is an equivalence relation on $\emu_{\epsilon}(X,B)$.
\end{lemma}

\begin{proof}
We check that $\sim$ is transitive (the other conditions are immediate).  Let $\phi_1\sim\psi$ and $\psi\sim\phi_2$.  Let $\theta_1,\theta_2$ and $u_1,u_2$ be such that 
$$
\|u_1(\phi_1(a)\oplus \theta_1(a))u_1^*-(\psi(a)\oplus \theta_1(a))\|<\epsilon
$$
and 
$$
\|u_2(\psi(a)\oplus \theta_2(a))u_2^*-(\phi_2(a)\oplus \theta_2(a))\|<\epsilon
$$
for all $a\in X$.  Set now $\theta:=\theta_1\oplus \psi\oplus \theta_2$.  Define 
$$
v:=\begin{psmallmatrix} u_1 & 0 \\ 0 & u_2 \end{psmallmatrix}\in \LL(F_{\phi_1}\oplus F_{\theta_1}\oplus F_{\psi}\oplus F_{\theta_2},F_{\psi}\oplus F_{\theta_1}\oplus F_{\phi_2}\oplus F_{\theta_2}),
$$
and let $w:F_{\psi}\oplus F_{\theta_1}\oplus F_{\phi_2}\oplus F_{\theta_2}\to F_{\phi_2}\oplus F_{\theta_1}\oplus F_{\psi}\oplus F_{\theta_2})$ be the canonical unitary that switches the order of the summands.  Define now $u:=wv$; one checks that 
$$
\|u(\phi_1(a)\oplus \theta(a))u^*-\phi_2(a)\oplus \theta(a)\|<\epsilon
$$
for all $a\in X$ which completes the proof.
\end{proof}

\begin{definition}\label{ucp gp}
With notation as in Definitions \ref{eps ucp} and \ref{delta sim}, let $V^{ucp}_{\epsilon}(A,B)$ be the quotient of $\emu_{\epsilon}(X,B)$ by the equivalence relation $\sim$, equipped with the operation $\oplus$.  It is straightforward to check that this is an abelian monoid, and we let $KK^{ucp}_{\epsilon}(X,B)$ be the associated Grothendieck group.
\end{definition}

\begin{remark}\label{fc3}
As in Remark \ref{fc1}, let $I_A$ be the set of all pairs $(X,\epsilon)$ where $X$ is a finite subset of the unit ball of $A$, and $\epsilon>0$.  Analogously to Remarks \ref{fc1} and \ref{fc2} we may define a partial order ``$\leq_{ucp}$'' on $I_A$ by stipulating that $(X,\epsilon)\leq_{ucp}(Y,\delta)$ if $\emu_\epsilon(X,B)\subseteq \emu_\delta(Y,B)$ for all $C^*$-algebras $B$.  In this case, we get a canonically induced \emph{forget control map}
$$
KK^{ucp}_{\delta}(Y,B)\to KK^{ucp}_{\epsilon}(X,B).
$$
\end{remark}

We now turn to relating the groups $KK^{ucp}_\epsilon(X,B)$ from Definition \ref{ucp gp} and $KK^{\pi,mvn}_\epsilon(X,B)$ from Definition \ref{kk mvn gp}.  We need some preliminary lemmas relating almost multiplicative ucp maps and almost commuting projections.  

\begin{lemma}\label{proj real}
Let $\pi:A\to \LL(E)$ be a unitally absorbing representation of $A$ on a Hilbert $B$-module; as usual we use $\pi$ to identify $A$ with its image.  Let $\epsilon\in (0,1]$ and let $X$ be a finite self-adjoint\footnote{We mean that ``for every $a\in X$, we also have $a^*\in X$'' and not the stronger condition ``$a=a^*$ for all $a\in X$''.} subset of the unit ball of $A$. Then with notation as in Definition \ref{eps ucp}, for every $\phi \in \emu_{\epsilon}(X,B)$ there exists an isometry $v_\phi\in \K(F_\phi,E)$ with the following properties:
\begin{enumerate}[(i)]
\item \label{isom approx} for all $a\in X$,
$$
\|\phi(a)-v^*_\phi av_\phi\|<\epsilon;
$$
\item \label{p com} if we define $p_\phi:=v_\phi v_\phi^*$, then
$$
\|[p_\phi,a]\|<\epsilon^{1/2}
$$
for all $a\in X$.  
\end{enumerate}
Moreover, if $\phi,\psi\in \emu_{\epsilon}(X,B)$, and $v_\phi$, $v_\psi$ have the properties in parts \eqref{isom approx} and \eqref{p com} above, then the following holds: if there is a unitary isomorphism $u\in \LL(F_\phi,F_\psi)$ satisfying $\|u\phi(a)u^*-\psi(a)\|<\epsilon$ for all $a\in X$, then there exists $v\in \LL(E)$ satisfying $vv^*=p_\psi$, $v^*v=p_\phi$, and 
$$
\|[v,a]\|<5\epsilon^{1/2}
$$
for all $a\in X$.
\end{lemma}

\begin{proof}
The definition of a unitally absorbing representation (see Definition \ref{ab unital def}) applied to $\phi$ gives a sequence of isometries $(v_m)$ in $\LL(F_\phi,E)$ such that $v_mav_m^*-\phi(a)\to 0$ as $m\to\infty$.  Moreover, as $F_\phi$ is compact, $\LL(F_\phi,E)=\K(F_\phi,E)$ (see Definition \ref{com hm} and Lemma \ref{hm equiv} above).  Hence for $m$ suitably large, $v_m$ has the property required by part \eqref{isom approx}.

Now look at part \eqref{p com}.  We claim that for some suitably large $m$, the projection $p_m:=v_mv_m^*$ satisfies the desired estimate.  We consider the ucp map $\psi_m:A\to \LL(p_mE)$ defined by $\psi_m(a):=p_map_m$.  For simplicity of notation, let us drop the subscripts $m$ for now.  Computing,
\begin{align}\label{p pi com 0}
\|[p,a]\| & = \|pa(1-p)+(1-p)a \| \nonumber\\
& =\|pa(1-p)a^*p+(1-p)apa^*(1-p)\|^{1/2} \nonumber\\
& = \max\{\|pa(1-p)a^*p\|^{1/2}, \|(1-p)apa^*(1-p)\|^{1/2}\},
\end{align}
where the first equality is a direct check, the second is the $C^*$-identity, and the third follows as the two terms in the norm preceding it are orthogonal.  On the other hand, we have 
\begin{equation}\label{p pi com 1}
pa(1-p)a^*p=\psi(aa^*)-\psi(a)\psi(a^*)
\end{equation}
by definition of $\psi$, and that (by the $C^*$-identity, twice)
\begin{align}\label{p pi com 2}
\|(1-p)apa^*(1-p)\| & =\|(1-p)ap\|^2 \nonumber \\ & =\|pa^*(1-p)ap\| \nonumber \\ & =\|\psi(a^*a)-\psi(a^*)\psi(a)\|.
\end{align}
Combining lines \eqref{p pi com 0}, \eqref{p pi com 1}, and \eqref{p pi com 2} and putting the subscript $m$ back in gives  
\begin{equation}\label{psi good}
\|[p_m,a]\|\leq \max\{\|\psi_m(aa^*)-\psi_m(a)\psi_m(a^*)\|^{1/2},\|\psi_m(a^*a)-\psi_m(a^*)\psi_m(a)\|^{1/2}\}.
\end{equation}
On the other hand for any $a\in A$, 
\begin{align*}
\|\psi_m(aa^*)-\psi_m(a)\psi_m(a^*)\| & =\|v_mv_m^*a^*av_mv_m^*-v_mv_m^*av_mv_m^*a^*v_mv_m^*\| \\
& = \|v_m^*a^*av_m-(v_m^*av_m)(v_m^*a^*v_m)\|.
\end{align*}
The right hand side converges to $\|\phi(aa^*)-\phi(a)\phi(a)^*\|$ as $m\to\infty$, and thus if $\phi$ is an $(X,\epsilon)$-$*$-homomorphism we see that 
\begin{equation}\label{ls2}
\limsup_{m\to\infty}\|\psi_m(aa^*)-\psi_m(a)\psi_m(a^*)\|<\epsilon
\end{equation}
for all $a\in X$ (here we use that $X$ is self-adjoint).  Similarly 
\begin{equation}\label{ls1}
\limsup_{m\to\infty}\|\psi_m(a^*a)-\psi_m(a^*)\psi_m(a)\|<\epsilon
\end{equation}
for all $a\in X$.  The desired estimate now follows from lines \eqref{psi good}, \eqref{ls1}, and \eqref{ls2} on taking $m$ suitably large.

Finally, we look at part the part dealing with unitary equivalences  We define $v:=v_\psi uv_\phi^*$ from which the identities $vv^*=p_\psi$ and $v^*v=p_\phi$ follow directly.  For the commutator estimate, we compute 
$$
\|[v,a]\|=\|v_\psi uv_\phi^*a -av_\psi u v_\phi^*\|=\|v_\psi uv_\phi^*p_\phi a-a p_\psi v_\psi u v_\phi^*\|.
$$
Using the assumption that $\|[p_\phi,a]\|<\epsilon^{1/2}$ for all $a\in X$, we see that 
\begin{align*}
\|[v,a]\| &<\|v_\psi uv_\phi^* ap_\phi-p_\psi a v_\psi u v_\phi^*\|+2\epsilon^{1/2} \\
& =\|v_\psi (uv_\phi^* a v_\phi u^*-v_\psi^* a v_\psi)uv_\phi^*\| +2\epsilon^{1/2} \\
& \leq \|uv_\phi^* a v_\phi u^*-v_\psi^* a v_\psi\| +2\epsilon^{1/2} \\
& < \|u\phi(a)u^*-\psi(a)\| +2\epsilon^{1/2}+2\epsilon,
\end{align*}
and this is bounded above by $2\epsilon^{1/2}+2\epsilon+\epsilon\leq 5\epsilon^{1/2}$ for $a\in X$ as claimed.
\end{proof}

\begin{corollary}\label{ucp to mvn}
Let $\pi:A\to \LL(E)$ be a strongly unitally absorbing representation of $A$ on a Hilbert $B$-module $E$; as usual, we identify $A$ with its image under $\pi$.  Let $\epsilon\in (0,1]$ and let $X$ be a finite, self-adjoint, subset of the unit ball of $A$.  Then there exists a well-defined homomorphism 
$$
\phi_{(X,\epsilon)}:KK^{\pi,ucp}_\epsilon(X,B)\to KK^{\pi,mvn}_{5\epsilon^{1/2}}(X,B)
$$
determined on cycles by sending $[\phi]$ to the class of any projection $p_\phi:=v_\phi v_\phi^*$ in $\LL(E)$ where $v_\phi$ has the properties from Lemma \ref{proj real}, parts \eqref{isom approx} and \eqref{p com}.
\end{corollary}

\begin{proof}
Lemma \ref{proj real} part \eqref{p com} show that the projection $p_\phi$ exists and is an element of $\mathcal{P}^{\pi,\K}_{\epsilon^{1/2}}(X,B)$.  If $p_\phi'$ is another projection with the same properties, then the part of Lemma \ref{proj real} dealing with unitary equivalences (applied with $\phi=\psi$ and $u$ the identity) implies that (with notation as in Definition \ref{mvn mon}) $[p_\phi]=[p_\phi']$.  We thus have a well-defined set map 
$$
\emu_{\epsilon}(X,B)\to V^{\pi,mvn}_{5\epsilon^{1/2}}(X,B).
$$
The part of Lemma \ref{proj real} dealing with unitary equivalences implies that this in fact descends to a set map
\begin{equation}\label{ucp to mvn mon map}
V^{ucp}_\epsilon(X,B)\to V^{\pi,mvn}_{5\epsilon^{1/2}}(X,B).
\end{equation}
Note also that if $\phi,\psi\in \emu_{\epsilon}(X,B)$ with associated isometries $v_\phi:F_\phi\to E$ and $v_\psi:F_\psi\to E$, and if $s_1,s_2\in \LL(E)$ are Cuntz isometries defining the semigroup operation on $V^{\pi,mvn}_{5\epsilon^{1/2}}(X,B)$, then $s_1v_\phi$ and $s_2v_\psi$ satisfy the same conditions that $v_\phi$ and $v_\psi$ do.  As we may represent $\phi\oplus \psi$ by $s_1v_\phi+s_2v_\psi$, this shows that the map in line \eqref{ucp to mvn mon map} is a semigroup homomorphism.  It thus descends to a homomorphism on Grothendieck groups, which is the desired conclusion.
\end{proof}

\begin{lemma}\label{mvn to ucp lem}
Let $\pi:A\to \LL(E)$ be a strongly unitally absorbing representation of $A$ on a Hilbert $B$-module $E$; as usual, we identify $A$ with its image under $\pi$.  Let $X$ be a finite subset of the unit ball of $A$.  With notation as in Definition \ref{vk}, let $p\in \mathcal{P}^{\pi,\K}_\epsilon(X,B)$.   The following hold:
\begin{enumerate}[(i)]
\item \label{induce ucp}  If $v_p:pE\to E$ is the canonical inclusion, then the map 
$$
\phi_p:A\to \LL(pE),\quad a\mapsto v_p^*av_p
$$
is in $\emu_{\epsilon^2}(X,B)$.    
\item \label{mvn equiv} Assume in addition $q\in \mathcal{P}^{\pi,\K}_\epsilon(X,B)$, and let $v_q$ and $\phi_q$ be defined analogously to part \eqref{induce ucp}.  Assume moreover that there is a partial isometry $v\in \LL(E)$ such that $v^*v=p$, $vv^*=q$, and such that $\|[v,a]\|<\epsilon$ for all $a\in X\cup X^*$.   Then there exists a unitary $u\in \LL(pE,qE)$ such that 
$$
\|u\phi_p(a)u^*-\phi_q(a)\|<\epsilon
$$
for all $a\in X$.
\end{enumerate}
\end{lemma}

\begin{proof}
For part \eqref{induce ucp}, direct checks show that $\phi_p$ is ucp.  To check that it is $(X,\epsilon^2)$-multiplicative, we note that for any $a,b\in X$,
\begin{align*}
\|\phi_p(ab)-\phi_p(a)\phi_p(b)\| & = \|v_p^*abv_p-v_p^*av_pv_p^*bv_p\| \\ & =\|pabp-papbp\| \\ &=\|pa(1-p)bp\| \\ & =\|[p,a](1-p)[b,p]\| \\ & \leq \|[p,a]\|\|[b,p]\|
\end{align*}
and the last line is less than $\epsilon^2$ as $\|[p,c]\|<\epsilon$ for all $c\in X$.

For part \eqref{mvn equiv}, we define $u:=v_q^*vv_p$.  Direct computations show that this is unitary.  For the claimed estimate, we have that 
\begin{align*}
\|u\phi_p(a)u^*-\phi_q(a)\| & = \|u(v_p^*av_p)u^*-v_q^*av_q\| \\
& = \|v_q^*(vp ap v^*v_q-v_q^*aq v_q\| \\
& =\|v_q^*(vav^*-aq)v_q\|\\
& \leq \|vav^*-avv^*\| \\
& \leq \|va-av\|.
\end{align*}
This now implies the desired estimate from the assumption on $v$.
\end{proof}

\begin{corollary}\label{mvn to ucp}
Let $\pi:A\to \LL(E)$ be a strongly unitally absorbing representation of $A$ on a Hilbert $B$-module $E$; as usual, we identify $A$ with its image under $\pi$.  Let $\epsilon>0$ and let $X$ be a finite subset of the unit ball of $A$.  Then there exists a well-defined homomorphism 
$$
\psi_{(X,\epsilon)}:KK^{\pi,mvn}_\epsilon(X,B)\to KK^{ucp}_\epsilon(X,B)
$$
determined on cycles by sending $[p]$ to the class of the ucp map 
$$
\phi:A\to \LL(pE),\quad a\mapsto v^*a v,
$$
where $v:pE\to E$ is the canonical inclusion.
\end{corollary}

\begin{proof}
The fact that $\phi$ as defined is an element of $\emu_{\epsilon}(X,B)$ follows directly from Lemma \ref{mvn to ucp lem} part \eqref{induce ucp}.  Thus we get a map $\mathcal{P}_\epsilon^{\pi,\K}(X,B)\to \emu_{\epsilon}(X,B)$.  Lemma \ref{mvn to ucp lem} part \eqref{mvn equiv} implies that this map descends to a well-defined set map 
\begin{equation}\label{mon map}
V^{\pi,mvn}_\epsilon(X,B)\to V^{ucp}_\epsilon(A,B)
\end{equation}
(see Definitions \ref{kk mvn} and \ref{ucp gp} for notation).  

Let now $p,q\in \mathcal{P}_\epsilon^{\pi,\K}(X,B)$ with associated ucp maps $\phi_p,\phi_q$, and let $s_1,s_2\in \LL(E)$ be Cuntz isometries defining the operation on $V^{\pi,mvn}_\epsilon(X,B)$.  Then we have that $[p]+[q]$ is represented in $V^{\pi,mvn}_\epsilon(X,B)$ by the projection $s_1ps_1^*+s_2qs_2^*$, and the associated ucp map is
$$
\phi_{s_1ps_1^*+s_2qs_2^*}:A\to \LL((s_1ps_1^*+s_2qs_2^*)E).
$$
Note that any element in $(s_1ps_1^*+s_2qs_2^*)E$ can be written uniquely as $s_1y+s_2z$ for some $y\in pE$ and $z\in q$.  One then checks that the map
$$
(s_1ps_1^*+s_2qs_2^*)E\to pE\oplus qE,\quad s_1y+s_2z\mapsto y\oplus z
$$
is a unitary isomorphism that intertwines $\phi_{s_1ps_1^*+s_2qs_2^*}$ and $\phi_p\oplus \phi_q$.  This shows that the map in line \eqref{mon map} is a semigroup homomorphism.  It thus descends uniquely to a homomorphism on Grothendieck groups, which is the homomorphism in the statement.
\end{proof}

\begin{proposition}\label{mvn and ucp}
Let $A$ be a separable unital $C^*$-algebra, and let $B$ be a separable $C^*$-algebra.  Let $\pi:A\to \LL(E)$ be a strongly unitally absorbing representation of $A$ on a Hilbert $B$-module.   Let $X$ be a finite and self-adjoint subset of the unit ball of $A$, and let $\epsilon>0$.  The maps from Lemmas \ref{ucp to mvn} and \ref{mvn to ucp} fit into commutative diagrams
\begin{equation}\label{mvn ucp t1}
\xymatrix{ KK^{ucp}_\epsilon(X,B) \ar[dr]_-{\phi_{(X,\epsilon)}} \ar[rr] & & KK^{ucp}_{5\epsilon^{1/2}}(X,B) \\
& KK^{\pi,mvn}_{5\epsilon^{1/2}}(X,B) \ar[ur]_-{\psi_{(X,5\epsilon^{1/2})}} & }
\end{equation}
and 
\begin{equation}\label{mvn ucp t2}
\xymatrix{ & KK^{ucp}_\epsilon(X,B) \ar[dr]^-{\phi_{(X,\epsilon)}} & \\
KK^{\pi,mvn}_\epsilon(X,B) \ar[ur]^-{\psi_{(X,\epsilon)}} \ar[rr] & & KK^{\pi,mvn}_{5\epsilon^1/2}(X,B) },
\end{equation}
where in both cases the unlabeled arrows are the canonical forget control maps as in Remarks \ref{fc2} and \ref{fc3}.
\end{proposition}

\begin{proof}
We first look the triangle in line \eqref{mvn ucp t1}.  Given a ucp map $\phi:A\to \K(F)$ representing a class in $KK^{\pi,ucp}_\epsilon(X,B)$, the bottom composition in line \eqref{mvn ucp t1} takes this to the class of the map $\psi:A\to \K(pE)$ defined by $\psi(a):=w^*aw$ where: $v:F\to E$ is an isometry with properties \eqref{isom approx} and \eqref{p com} from Lemma \ref{proj real} with respect to $\phi$, $p=vv^*$, and $w:pE\to E$ is the canonical inclusion.  Consider now the map $u:=w^*v:F\to pE$.  One checks that this is a unitary isomorphism, and moreover 
\begin{align*}
\|u\phi(a)u^*-\psi(a)\| & =\|w^*v\phi(a)v^*w-w^*aw\|=\|w^*(v\phi(a)v^*-vv^*avv^*)w\|  \\ & \leq \|\phi(a)-v^*av\|,
\end{align*}
and the right hand side is smaller than $\epsilon$ for all $a\in X$ by Lemma \ref{proj real} part \eqref{isom approx}.  This implies that $\psi$ and $\phi$ are equivalent in $KK^{\pi,ucp}_{5\epsilon^{1/2}}(X,B)$, which shows that the diagram commutes.

We now look at triangle \eqref{mvn ucp t2}.  Given a projection $p\in \mathcal{P}^{\pi,\K}_{\epsilon}(X,B)$ representing a class in $KK^{\pi,mvn}_\epsilon(X, B)$, the top composition in line \eqref{mvn ucp t2} takes this to the class of the projection $q\in \mathcal{P}^{\pi,mvn}_{5\epsilon^{1/2}}(X,B)$, where $q:=vv^*$, where: $v:pE\to E$ is any isometry with the properties in Lemma \ref{proj real}, parts \eqref{isom approx} and \eqref{p com}, relative to the ucp map $\phi:A\to \K(pE)$, $a\mapsto w^*aw$, with $w:pE\to E$ the canonical inclusion.  However, we may just take $w=v$, in which case we have equality of $p$ and $q$ even on the level of cycles.  
\end{proof}

\section{Almost multiplicative maps}\label{llp sec}

In this section, we give a reformulation of the controlled $KK$-groups $KK^{ucp}_\epsilon(X,B)$ of Definition \ref{ucp gp} that in some sense drops the ucp condition.  This needs extra structure: we assume that $A$ comes equipped with a suitably nice countable subset with dense span; this is an unnatural assumption from a $C^*$-algebra point of view, but works well in our main examples.  We also need additional assumptions on $A$: we will assume that $A$ has Kirchberg's \emph{lifting property} (LP) (in many situations one can also get away with the weaker \emph{local lifting property} (LLP)).  See \cite{Kirchberg:1993aa} for the original systematic study of the LP and LLP, and \cite[Chapter 13]{Brown:2008qy} and \cite[Chapter 9]{Pisier:2020aa} for textbook expositions.  As usual, this reformulation comes at the price of changing the parameters - in this case, we must change both the finite set $X$ and the error parameter $\epsilon$.  

Throughout this section, $A$ and $B$ denote separable $C^*$-algebras with $A$ unital.  All Hilbert modules will be over $B$ and countably generated unless otherwise stated.  We write $\LL(E)$, respectively $\K(E)$, for the $C^*$-algebras of adjointable, respectively compact, operators on a Hilbert module $E$.  We will work primarily with compact Hilbert $B$-modules as in Definition \ref{com hm}.

For applications, it will be convenient to work with a countable subset of $A$ satisfying appropriate conditions.

\begin{definition}\label{admiss}
Let $\mathcal{X}$ be a subset of a unital $C^*$-algebra $A$, and let $\mathcal{A}$ denote the span of $\mathcal{X}$.  The subset $\mathcal{X}$ is said to be \emph{admissible} if it satisfies the following conditions:
\begin{enumerate}[(i)]
\item $\mathcal{X}$ is countable;
\item $\mathcal{X}$ contains the unit of $A$;
\item all elements of $\mathcal{X}$ have norm at most one;
\item $\mathcal{A}$ is a dense $*$-subalgebra of $A$;
\item any unital $*$-homomorphism $\mathcal{A}\to B$ from $\mathcal{A}$ to a $C^*$-algebra $B$ extends uniquely to a $*$-homomorphism $A\to B$.
\end{enumerate}
\end{definition}

The most important example of an admissible subset of a $C^*$-algebra for us is the canonical copy of the group $\Gamma$ inside the maximal group $C^*$-algebra $C^*\Gamma$.  Other universal $C^*$-algebras equipped with a reasonable countable set of generators such as Cuntz algebras, Cuntz-Krieger algebras, semigroup $C^*$-algebras, graph $C^*$-algebras and so on provide additional examples.

\begin{definition}\label{non ucp}
Let $A$ and $B$ be separable $C^*$-algebras with $A$ unital.  Let $\mathcal{X}$ be an admissible subset of $A$ in the sense of Definition \ref{admiss}, and let $\mathcal{A}$ be its span.  Let $\epsilon\in (0,1]$, let $X$ be a finite subset of $\mathcal{X}$.  Define $\lin_{\epsilon}(X,B)$ to be the collection\footnote{It is not a set: as usual we leave details to get around this to the reader.} of all linear maps $\psi:\mathcal{A}\to \K(F)$ satisfying the following conditions:
\begin{enumerate}[(i)]
\item $F$ is a compact Hilbert $B$-module in the sense of Definition \ref{com hm};
\item $\psi(1_A)=1_F$;
\item $\|\psi(a)^*-\psi(a^*)\|<\epsilon$ for all $a\in X$;
\item $\|\psi(ab)-\psi(a)\psi(b)\|<\epsilon$ for all $a,b\in X$;
\item $\|\psi(a)\|\leq 1$ for all $a\in \mathcal{X}$.
\end{enumerate}

For $\phi,\psi\in \lin_{\epsilon}(X,B)$, we write $\phi\sim \psi$ if there is $\theta\in \lin_{\epsilon}(X,B)$ and a unitary isomorphism $u:F_\phi\oplus F_\theta\to F_\psi\oplus F_\theta$ such that $\|u(\phi(a)\oplus \theta(a))u^*-\psi(a)\oplus \theta(a)\|$ for all $a\in X$; analogously to Lemma \ref{ucp equiv} above this is an equivalence relation.

We then define $V^{lin}_{\epsilon}(X,B)$ to be $\lin_{\epsilon}(X,B)/\sim$.  This becomes a commutative monoid when equipped with block sum for addition, and we let $KK^{\lin}_{\epsilon}(X,B)$ denote its Grothendieck group.
\end{definition}

\begin{remark}\label{fc4}
Analogously to Remarks \ref{fc1}, \ref{fc2}, and \ref{fc3}, if $\mathcal{X}$ is an admissible subset of $A$, if $Y\subseteq X$ are finite subsets of $\mathcal{X}$, and if $0<\epsilon<\delta$ then there is an inclusion $\lin_{\epsilon}(X,B)\to \lin_{\delta}(Y,B)$ that canonically induces a \emph{forget control homomorphism}
$$
KK^{\lin}_{\epsilon}(X,B)\to KK^{\lin}_{\delta}(Y,B).
$$
\end{remark}

Variants of the following lemma are well-known: compare for example \cite[Lemma 2.1]{Lin:1995aa}, \cite[Corollary 1.7]{Ioana:2020aa}, or \cite[Proposition 2.7]{Willett:2024aa}.  We give a proof for the reader's convenience.

\begin{lemma}\label{lp lem}
Let $\mathcal{X}\subseteq A$ be an admissible subset of $A$ in the sense of Definition \ref{admiss}.  Assume that $A$ has the lifting property.  Let $X_1\subseteq X_2\subseteq \cdots $ be an increasing sequence of finite subsets of $\mathcal{X}$ with union $\mathcal{X}$, and let $\epsilon_1>\epsilon_2>\cdots $ be a decreasing sequence that tends to zero.

With notation as in Definition \ref{non ucp}, for all finite subsets $Y$ of $\mathcal{X}$ and all $\delta>0$, there exists $N$ such that for all $n\geq N$ if $\psi\in \lin_{\epsilon_n}(X_n,B)$, then there exists a ucp map $\phi:A\to \K(F_\phi)$ such that $\|\psi(a)-\phi(a)\|<\delta$ for all $a\in Y$.
\end{lemma}

\begin{proof}
Let $\mathcal{A}$ denote the linear span of $\mathcal{X}$, a dense $*$-subalgebra of $A$.  Assume the conclusion does not hold for contradiction, so there exists a finite subset $Y$ of $\mathcal{X}$ and a $\delta>0$ such that for any $n$ there is a map $\phi_n:\mathcal{A}\to \K(F_n)$ in $\lin_{\epsilon_n}(X_n,B)$ such that for any ucp map $\psi:A\to \K(F_n)$ we have that $\|\phi_n(a)-\psi(a)\|\geq \delta$ for some $a\in Y$.

Define now $C:=\prod_{n=1}^\infty \K(F_n)$ to be $C^*$-algebra of all bounded sequences $(c_n)$ with $c_n\in \K(F_n)$, and define $J:=\bigoplus_{n=1}^\infty \K(F_n)$ to be the ideal consisting of sequences $(c_n)$ such that $\|c_n\|\to 0$ as $n\to\infty$.  Note that the condition that $\|\phi(x)\|\leq 1$ for all $x\in \mathcal{X}$ implies that for any $a\in \mathcal{A}$, the sequence $(\phi_n(a))_{n=1}^\infty$ is uniformly bounded, and so a well-defined element of $C$.  The conditions on the $\phi_n$ imply then that the map 
$$
\mathcal{A}\to C/J,\quad a\mapsto (\phi_n(a))_{n=1}^\infty
$$
is a unital $*$-homomorphism, so extends (uniquely) to a unital $*$-homomorphism $\Phi:A\to C/J$ by admissibility.  Hence using the LP, there is a ucp lift $\Psi:A\to C$ of $\Phi$.  The map $\Psi$ is necessarily of the form $\Psi(a)=(\psi_n(a))_{n=1}^\infty$, where each $\psi_n$ is a ucp map $A\to \K(F_n)$.  Moreover, as $\Psi$ is a lift of $\Phi$, we have that $\Phi(a)-\Psi(a)\in J$ for all $a\in A$ and so in particular $\|\phi_n(a)-\psi_n(a)\|\to 0$ as $n\to\infty$ for all $a\in \mathcal{X}$.  This contradicts the assumption that $\phi_n$ is bounded away from ucp maps on the finite set $Y$.
\end{proof}

\begin{remark}\label{lp vs llp}
If $B$ is nuclear, we may replace the LP assumption on $A$ with the LLP, and the conclusion of Lemma \ref{lp lem} still holds.  This uses the notion of a QWEP $C^*$-algebra: see for example \cite{Ozawa:2004ab} for background.  Indeed, if $B$ is nuclear, then so is $\K(F)$ for any (countably generated) $B$-module $F$: this follows for example as Kasparov's stabilization theorem \cite[Theorem 2]{Kasparov:1980sp} (see also \cite[Chapter 6]{Lance:1995ys} for a textbook exposition) implies that $\K(F)$ embeds as a corner in $\K(B\otimes \ell^2)$.  Hence in particular $\K(F)$ is QWEP, whence also any product of such algebras is also QWEP (see for example \cite[Proposition 4.1]{Ozawa:2004ab}), and in particular the $C^*$-algebra $C$ from the proof of Lemma \ref{lp lem} is QWEP.  Hence by \cite[Corollary 3.12]{Ozawa:2004ab}, the map $\Phi:A\to C/J$ would still admit a ucp lift under the assumption that $A$ satisfies the LLP.
\end{remark}

\begin{corollary}\label{lin to ucp}
Assume that $A$ has the lifting property.   Let $\mathcal{X}$ be an admissible subset of $A$ in the sense of Definition \ref{admiss}.  Then for any finite subset $Y$ of $\mathcal{X}$ and any $\delta>0$, there is a finite subset $X$ of $\mathcal{X}$ and $\epsilon\in (0,\delta/3)$ such that there is a well-defined homomorphism.
$$
\psi_{(X,\epsilon)}:KK^{lin}_{\epsilon}(X,B)\to KK^{ucp}_{\delta}(Y,B).
$$
\end{corollary}

\begin{proof}
Given $Y$ and $\delta$, we let $X$ and $\epsilon$ be as in Lemma \ref{lp lem}, although with $\delta$ replaced by $\delta / 3$.  Then Lemma \ref{lp lem} shows that for any $\psi\in \lin_{\epsilon}(X,B)$, there exists a ucp map $\phi:A\to \K(F_\psi)$ such that $\|\phi(a)-\psi(a)\|<\delta/3$ for all $a\in Y$.  Note that if $\phi':A\to \K(F_\psi)$ is another ucp map satisfying the same conditions then $\|\phi(a)-\phi'(a)\|<\delta$ for all $y\in Y$, and thus the classes $[\phi]$ and $[\phi']$ in the set $V^{ucp}_\delta(Y,B)$ of Definition \ref{delta sim} are the same.  We thus have a well-defined map
$$
\lin_{\epsilon}(X,B)\to V^{ucp}_\delta(Y,B).
$$
Note moreover that as $X\supseteq Y$, and as $\epsilon<\delta/3$ we have that this descends to a well-defined map 
$$
V^{\lin}_{\epsilon}(X,B)\to V^{ucp}_\delta(Y,B).
$$
It is clearly additive, so a monoid homomorphism, and so descends to an homomorphism on Grothendieck groups as in the statement. 
\end{proof}

On the other hand, if $\mathcal{X}$ is an admissible subset of $A$ then for any finite subset $X$ of $\mathcal{X}$ and any $\epsilon>0$ there is a well-defined homomorphism 
\begin{equation}\label{ucp to lin}
\phi_{(X,\epsilon)}: KK^{ucp}_{\delta}(X,B)\to KK^{\lin}_{\epsilon}(X,B)
\end{equation}
which is induced by the identity map on the level of cycles.  

\begin{proposition}\label{lin vs ucp}
Let $A$ be a separable unital $C^*$-algebra with the lifting property, and let $B$ be a separable $C^*$-algebra.  Let $\mathcal{X}$ be an admissible subset of $A$.  Let $X$ be a finite subset of $\mathcal{X}$, let $\epsilon>0$ and let $(Y,\delta)$ be as in Corollary \ref{lin to ucp}.  Then the maps from Corollary \ref{lin to ucp} and line \eqref{ucp to lin} fit into commutative diagrams
\begin{equation}\label{lin ucp t1}
\xymatrix{ KK^{ucp}_\epsilon(X,B) \ar[dr]^-{\psi_{(X,\epsilon)}} \ar[rr] & & KK^{ucp}_\delta(Y,B) \\
&  KK^{lin}_{\epsilon}(X,B) \ar[ur]^-{\phi_{(X,\epsilon)}} & }
\end{equation}
and 
\begin{equation}\label{lin ucp t2}
\xymatrix{ & KK^{ucp}_\delta(Y,B) \ar[dr]^-{\phi_{(Y,\delta)}} &  \\
KK^{lin}_{\epsilon}(X,B) \ar[ur]^-{\psi_{(X,\epsilon)}}  \ar[rr] & & KK^{\lin}_{\delta}(Y,B)  },
\end{equation}
where in both cases the unlabeled arrows are the canonical forget control maps from Remarks \ref{fc3} and \ref{fc4}.
\end{proposition}

\begin{proof}
We start with the diagram in line \eqref{lin ucp t1}.  Let $\phi\in \emu_{\epsilon}(X,B)$ represent a class in $KK_\epsilon(X,B)$.  The bottom composition in line \eqref{lin ucp t1} takes this to the class of the same map in $KK^{lin}_{\epsilon}(X,B)$, and then replaces this by the class of any ucp map that is suitably close to it.  However, as $\phi$ is already ucp, we may just $\phi$ itself for this approximant, so commutativity here is automatic.

We now look at the triangle in line \eqref{lin ucp t2}.  Let $\psi:A\to \K(F)$ be an element of $\lin_{\epsilon}(X,B)$ representing a class in $KK^{\lin}_{\epsilon}(X,\epsilon)$.  Then the top composition takes this to the class in $KK^{\lin}_{\delta}(Y,B)$ of any ucp map $\phi:A\to \K(F)$ that approximates $\psi$ within $\delta/3$ on $Y$.  However, any two maps $\phi,\psi:A\to \K(F)$ define the same class in $KK^{\lin}_{\delta}(Y,B)$ if they are within $\delta$ of each other on $Y$, so we are done.
\end{proof}

\section{Main general results}\label{main sec}

We are finally ready to record our main general $C^*$-algebraic results; we will then give some elaborations, and finally circle back to applications to quasi-representations of groups.  

We first need to make some comments about the index sets and forget control maps of Remarks \ref{fc1}, \ref{fc2}, and \ref{fc3}.  For this, recall from Remark \ref{fc1} that if $A$ is a separable unital $C^*$-algebra, then $I_A$ denotes the collection of all pairs $(X,\epsilon)$ where $X$ is a finite subset of the unit ball of $A$ and $\epsilon>0$.  This is equipped with order structures $\leq_u$, $\leq_\K$, and $\leq_{ucp}$ in Remarks \ref{fc1} and \ref{fc3}.

\begin{lemma}\label{gen cof}
Let $A$ be a separable unital $C^*$-algebra.  Let $I_A$ be the set from Remark \ref{fc1} as in the discussion above, and consider the order structures $\leq_u$, $\leq_\K$, or $\leq_{ucp}$ from Remarks \ref{fc1} and \ref{fc3}.  Then the following hold.
\begin{enumerate}[(i)]
\item \label{dos} The order structures $\leq_u$, $\leq_\K$, or $\leq_{ucp}$ are all directed.
\item \label{dcf} Let $X_1\subseteq X_2\subseteq \cdots $ be a nested sequence of finite subsets of the unit ball of $A$ with dense union, and let $\epsilon_1>\epsilon_2>\cdots$ be a decreasing sequence of positive numbers that converge to zero.  Then $((X_n,\epsilon_n))_{n=1}^\infty$ is a cofinal sequence for all of the order structures $\leq_u$, $\leq_\K$, or $\leq_{ucp}$ from Remarks \ref{fc1} and \ref{fc3}.
\item \label{acf} Let $\mathcal{X}$ be an admissible (countable) subset of $A$ as in Definition \ref{admiss} above.  Let $X_1\subseteq X_2\subseteq \cdots $ be a nested sequence of finite subsets of $\mathcal{X}$ whose union is $\mathcal{X}$, and let $\epsilon_1>\epsilon_2>\cdots$ be a decreasing sequence of positive numbers that converge to zero.  Then $((X_n,\epsilon_n))_{n=1}^\infty$ is a cofinal sequence for all of the order structures $\leq_u$, $\leq_\K$, and $\leq_{ucp}$ from Remarks \ref{fc1} and \ref{fc3}.
\end{enumerate}
\end{lemma}

\begin{proof}
For part \eqref{dos}, we note that for any of these order structures, we have that if $(X,\epsilon)$ and $(Y,\delta)$ in $I_A$, then $(X\cup Y,\min\{\epsilon,\delta\})$ is an upper bound.

For part \eqref{dcf}, we define a new order structure on $I_A$ by stipulating that $(X,\epsilon)\leq (Y,\delta)$ if for all $x\in X$ there exists $y\in Y$ such that $\|x-y\|<\epsilon/5$, and such that $\delta<\epsilon/5$.  One then checks that $(X,\epsilon)\leq (Y,\delta)$ implies that $(X,\epsilon)\leq_* (Y,\delta)$ where $*$ is any of the decorations ``$u$'', ``$\K$'', or ``$ucp$'': this follows from direct checks that we leave this to the reader.  The given collection is easily seen to be cofinal in $(I_A,\leq)$, so we are done with this part.

For part \eqref{acf}, the result follows as $\mathcal{X}$ has dense span, and a similar observation to the previous part: we again leave the details to the reader.
\end{proof}

\begin{remark}
Assume $A$ is finitely generated as a $C^*$-algebra, i.e.\ there is a finite subset $X$ of $A$ such that the $*$-algebra generated by $X$ is dense in $A$.  Then if $\epsilon_1>\epsilon_2>\cdots$ is a decreasing sequence tending to zero, the sequence $(X,\epsilon_n)$ is in fact cofinal in $I_A$ for any of the order structures $\leq_u$, $\leq_\K$, or $\leq_{ucp}$.  We will not need to use this, so just mention the key estimates.  For the order structures $\leq_u$ and $\leq_\K$, we use the estimate 
$$
\|[p,ab]\|\leq \|[p,a]\|\|b\|+\|a\|\|[p,b]\|.
$$
The necessary estimates for the order structure $\leq_{ucp}$ can be reduced to this using the Stinespring construction.
\end{remark}

\begin{theorem}\label{main ucp}
Let $A$ be a separable, unital $C^*$-algebra, and let $B$ be a separable $C^*$-algebra.  Assume moreover that $A$ is quasidiagonal and that $B$ is nuclear and admits an approximate unit consisting of projections.  

Then with notation as in Definition \ref{ucp gp} there is a canonical homeomorphic isomorphism 
$$
KL(A,B)\to\lim_{\leftarrow} KK_\epsilon^{ucp}(X,B)
$$
where the inverse limit on the right hand side is taken over the directed set $(I_A,\leq_{ucp})$ of Remark \ref{fc3} (see also Lemma \ref{gen cof}).
\end{theorem}

\begin{proof}
Let $\pi:A\to \LL(E)$ be a strongly unitally absorbing and quasidiagonal representation of $A$ on a Hilbert $B$-modue: such a representation exists under our assumptions by Example \ref{qd ex}.  We then have a chain of isomorphisms
\begin{align*}
KL(A,B) & \cong  \lim_{\leftarrow} KK^{\pi,u}_\epsilon(X,B)\\& \cong \lim_{\leftarrow} KK^{\pi,\mathcal{K}}_\epsilon(X,B) \\&\cong   \lim_{\leftarrow}KK_{\epsilon}^{\pi,mvn}(X,B) \\&\cong \lim_{\leftarrow}KK^{ucp}_\epsilon(X,B). 
\end{align*}
where: the inverse limit ${\displaystyle \lim_{\leftarrow} KK^{\pi,u}_\epsilon(X,B)}$ is taken over the directed set $(I_A,\leq_u)$ of Remark \ref{fc1} and the first isomorphism follows from \cite[Proposition A.11]{Willett:2020aa}; the inverse limit ${\displaystyle\lim_{\leftarrow} KK^{\pi,\K}_\epsilon(X,B)}$ is taken over the directed set $(I_A,\leq_\K)$ of Remark \ref{fc1} and the second isomorphism follows Theorem \ref{k to u the}; the inverse limit ${\displaystyle\lim_{\leftarrow} KK^{\pi,mvn}_\epsilon(X,B)}$ is also taken over the directed set $(I_A,\leq_\K)$ of Remark \ref{fc1} and the third isomorphism follows from Proposition \ref{mvn vs hom} and Lemma \ref{triangle isos}; and finally the inverse limit ${\displaystyle\lim_{\leftarrow} KK^{ucp}_\epsilon(X,B)}$ is taken over the directed set $(I_A,\leq_{ucp})$ of Remark \ref{fc3} and the fourth isomorphism follows from Proposition \ref{mvn and ucp} and Lemma \ref{triangle isos}.  In all cases we also use Lemma \ref{gen cof} to reduce to cofinal sequences in the relevant directed sets.
\end{proof}

We now look at a corresponding result for the groups $KK^{\lin}_\epsilon(X,B)$ of Definition \ref{non ucp}; this holds under an additional lifting property assumption on $A$.

\begin{definition}\label{order2}
Let $A$ be a separable, unital $C^*$-algebra and let $\mathcal{X}$ be an admissible subset of $A$ in the sense of Definition \ref{admiss}.  Let $I_{\mathcal{X}}$ be the set of all pairs $(X,\epsilon)$ where $X$ is a finite subset of $\mathcal{X}$, and $\epsilon>0$.  Define a relation ``$\leq$'' on $I_{\mathcal{X}}$ by stipulating that $(X,\epsilon)\leq (Y,\delta)$ if $Y\supseteq X$ and $\delta<\epsilon$.
\end{definition}

\begin{remarks}
Just as in Lemma \ref{gen cof}, the set $I_{\mathcal{X}}$ of Definition \ref{order2} is directed.  As $\mathcal{X}$ is countable, it clearly admits a cofinal sequence.  
\end{remarks}

\begin{theorem}\label{main no ucp}
Let $A$ be a separable, unital $C^*$-algebra, and let $B$ be a separable $C^*$-algebra.  Let $\mathcal{X}$ be an admissible subset of $A$ in the sense of Definition \ref{admiss}.  Assume that $B$ is nuclear and admits an approximate unit consisting of projections.  Assume that $A$ has the LLP.  

Then with notation as in Definition \ref{non ucp} there is a canonical homeomorphic isomorphism 
$$
KL(A,B)\to\lim_{\leftarrow} KK_\epsilon^{\lin}(X,B)
$$
where the inverse limit on the right hand side is taken over the directed set $I_{\mathcal{X}}$ of Definition \ref{order2} using the forget control maps in Remark \ref{fc4}.
\end{theorem}

\begin{proof}
This follows from Theorem \ref{main ucp}, Proposition \ref{lin vs ucp} (plus Remark \ref{lp vs llp} to replace an LP assumption with an LLP assumption), and Lemma \ref{triangle isos}.
\end{proof}

\section{Honest representations}\label{rfd sec}

In this section, we discuss how to improve the results of Section \ref{main sec} when the $C^*$-algebra $A$ has many finite-dimensional representations.  

Throughout this section, $A$ and $B$ denote separable $C^*$-algebras with $A$ unital.  All Hilbert modules will be over $B$ and countably generated unless otherwise stated.  We write $\LL(E)$, respectively $\K(E)$, for the $C^*$-algebras of adjointable, respectively compact, operators on a Hilbert module $E$.  We will work primarily with compact Hilbert $B$-modules as in Definition \ref{com hm}.

The following is a strengthening of Definition \ref{qd rep} above.

\begin{definition}
Let $A$ and $B$ be separable $C^*$-algebras with $A$ unital.  A unital representation $\pi:A\to \LL(E)$ of $A$ on a Hilbert $B$-module is \emph{residually finite-dimensional}\footnote{Much as a compact operator on a Hilbert module is not typically compact in the usual sense of Banach space, this notion does not have much to do with finite-dimensionality in the usual sense: the finite-dimensionality should be thought of as relative to $B$.} if there is an increasing sequence $(e_n)$ of projections in $\K(E)$ that converges strictly to the identity, and such that $[\pi(a),e_n]=0$ for all $n$ and all $a\in A$.
\end{definition}

\begin{example}\label{rfd ex}
A $C^*$-algebra $A$ is \emph{residually finite-dimensional} if it admits a faithful residually finite-dimensional representation (in the sense above) on a Hilbert space.  Arguing similarly to Example \ref{qd ex}, one sees that if $A$ is separable, unital, and residually finite-dimensional, and if $B$ is nuclear and admits an approximate unit of projections, then $A$ admits a strongly unitally absorbing and residually finite-dimensional representation on a Hilbert $B$-module.
\end{example}

\begin{lemma}\label{in rep}
Let $A$ and $B$ be separable $C^*$-algebras with $A$ unital.  Let $F$ be a compact $B$-module as in Definition \ref{com hm}, and let $\phi:A\to \LL(F)$ be a ucp map.  Assume moreover that $A$ admits a unitally absorbing and residually finite-dimensional representation on a Hilbert $B$-module $E$; as usual we identify $A$ with its image in this representation.

Then for any finite subset $X$ of $A$ and $\epsilon>0$ there exists a compact Hilbert $B$-module $G$, a $*$-homomorphism $\pi:A\to \LL(G)$, and an isometry $v:F\to G$ such that $\|\phi(a)-v^*\pi(a)v\|<\epsilon$ for all $a\in X$.  If moreover $\phi$ is an $(X,\epsilon)$-$*$-homomorphism in the sense of Definition \ref{alm mult def}, then we may also assume that $\|[vv^*,\pi(a)]\|<2\epsilon^{1/2}$ for all $a\in X$.
\end{lemma}

\begin{proof}
Using that the representation on $E$ is unitally absorbing, there is a sequence $(w_n:F\to E)_{n=1}^\infty$ of isometries such that $\|w_n^*aw_n-\phi(a)\|\to 0$ as $n\to\infty$.  As $F$ is compact, each $w_n$ is a compact operator, whence $w_nw_n^*$ is too.  As $(e_n)$ converges strictly to the identity, we may thus assume via a standard perturbation argument that for all $n$ there exists $m_n$ such that $w_nw_n^*\leq e_{m_n}$; in other words, the image of $w_n$ is contained in that of $e_{m_n}$.  Define now $G_n:=e_{m_n}E$, let $x_n:G_n\to E$ be the inclusion isometry, and let $v_n:=x_n^*w_n$.  Define $\pi_n:A\to \LL(G_n)$, $\pi_n(a):=x_n^*ax_n$; as $e_{m_n}$ commutes with $A$, this is a $*$-homomorphism.  We then have that for $n$ suitably large, $v:=v_n$, $G:=G_n$, and $\pi:=\pi_n$ satisfy $\|\phi(a)-v^*\pi(a)v\|<\epsilon$ for all $a\in X$.

If we assume moreover that $\phi$ is an $(X,\epsilon)$-$*$-homomorphism, the commutator conclusion follows from exactly the same argument as in Lemma \ref{proj real} part \eqref{p com}.
\end{proof}

\begin{corollary}\label{stab cor}
Let $A$ and $B$ be separable $C^*$-algebras with $A$ unital.   Assume moreover that $A$ admits a unitally absorbing and residually finite-dimensional representation on a Hilbert $B$-module $E$.   Let $X$ be a finite subset of the unit ball of $A$, let $\epsilon\in (0,1]$, and let $\emu_\epsilon(X,B)$ be as in Definition \ref{eps ucp}.  The following hold:
\begin{enumerate}[(i)]
\item \label{into hom} for any $\phi\in \emu_\epsilon(X,B)$ there exists an $(X,2\epsilon^{1/2})$-$*$-homomorphism $\psi:A\to \LL(F_\psi)$ from $A$ to the adjointable operators on a compact $B$-module, and a unital $*$-homomorphism $\pi:A\to \LL(F_\phi\oplus F_\psi)$ such that 
$$
\|\phi(a)\oplus \psi(a)-\pi(a)\|<3\epsilon^{1/2}
$$
for all $a\in X$.
\item \label{stab good} for any $\phi,\psi\in \emu_{\epsilon}(X,B)$ if $[\phi]=[\psi]$ in $KK^{ucp}_\epsilon(X,B)$, then there exists a unital $*$-homomorphism $\pi:A\to \LL(F_\pi)$ from $A$ to the adjointable operators on a compact $B$-module, and a unitary isomorphism $u:F_\phi\oplus F_\pi\to F_\psi\oplus F_\pi$ such that 
$$
\|u(\phi(a)\oplus \pi(a))u^*-(\psi(a)\oplus \pi(a))\|<6\epsilon^{1/2}
$$
for all $a\in X$.
\end{enumerate}
\end{corollary}

\begin{proof}
We first look at part \eqref{into hom}.  Let $\phi$ be given, and let $\pi_0:A\to \LL(G)$ and $v:F_\phi\to G$ be as in the conclusion of Lemma \ref{in rep}, so in particular if $p=vv^*$ then $\|[p,\pi_0(a)]\|<2\epsilon^{1/2}$ for all $a\in X$.  Define $F_\psi:=pG$, let $w:F_\psi\to G$ be the inclusion isometry, and define $\psi:A\to \LL(F_\psi)$ by $\psi(a):=w^*\pi(a)w$.  Then one computes that $\psi(a)$ is a ucp $(X,2\epsilon^{1/2})$-$*$-homomorphism.  Define $u:F_\phi\to F_\psi\to G$ by $(\xi,\eta)\mapsto v\xi+w\eta$; this is unitary and we define $\pi(a):=u\pi_0(a)u^*$.  Finally, we compute that for $a\in X$,
\begin{align}\label{pi and sum}
&\|\pi(a) -(\phi(a)\oplus \psi(a))\|  =\|\pi_0(a)-u^*(\phi(a)\oplus \psi(a))u\| \notag \\ & 
=\|\pi_0(a)-v\phi(a)v^*-w\psi(a)w^*\| \notag \\ &
=\big\|\big(p\pi_0(a)p-v\phi(a)v^*\big)+\big((1-p)\pi_0(a)(1-p)-w\psi(a)w^*\big)+\big(p\pi_0(a)(1-p)+(1-p)\pi_0(a)p\big)\big\| \notag \\ &
\leq \|p\pi_0(a)p-v\phi(a)v^*\|+0+\|p\pi_0(a)(1-p)+(1-p)\pi_0(a)p\| \notag \\ & 
<\epsilon+\|p\pi_0(a)(1-p)+(1-p)\pi_0(a)p\|.
\end{align}
Using the $C^*$-identity twice and orthogonality,
\begin{align*}
\|p\pi_0(a)(1-p)+(1-p)\pi_0(a)p\|^2 & =\|p\pi_0(a)^*(1-p)\pi_0(a)p+(1-p)\pi_0(a)^*p\pi_0(a)(1-p)\| \\ & =\max\{\|(1-p)\pi_0(a)p\|^2,\|p\pi_0(a)(1-p)\|^2\}.
\end{align*}
The terms appearing in the max on the right are both bounded by $\|[p,\pi_0(a)]\|^2<4\epsilon$.  Combining this with line \eqref{pi and sum} gives
$$
\|\pi(a)-(\phi(a)\oplus \psi(a))\|<\epsilon+2\epsilon^{1/2};
$$
as $\epsilon\leq 1$, this implies the claimed estimate.

Now let us look at part \eqref{stab good}.  Let $\phi,\psi$ satisfy $[\phi]=[\psi]$ in $KK^{ucp}_\epsilon(X,B)$.  Then by definition of the equivalence relation on $KK^{ucp}_\epsilon(X,B)$ there exists a ucp $(X,\epsilon)$-$*$-homomorphism $\theta:A\to F_\theta$ and a unitary $u_0:F_\phi\oplus F_\theta\to F_\psi\oplus F_\theta$ such that 
\begin{equation}\label{u0 ass}
\|u_0(\phi(a)\oplus \theta(a))u_0^*-\psi(a)\oplus \theta(a)\|<\epsilon
\end{equation}
for all $a\in X$.  Applying part \eqref{into hom} there exists a ucp $(X,2\epsilon^{1/2})$-$*$-homomorphism $\sigma:A\to \LL(F_\sigma)$ with $F_\sigma$ a compact $B$-module, and a unital $*$-homomorphism $\pi:A\to \LL(F_\theta\oplus F_\sigma)$ such that 
$$
\|\pi(a)-(\theta(a)\oplus \sigma(a))\|<3\epsilon^{1/2}
$$
for all $a\in X$.  Now define $u:=u_0\oplus \text{id}_{F_\sigma}$; it is not difficult to check that this has the right properties.
\end{proof}

\begin{corollary}\label{stab lin cor}
Let $A$ and $B$ be separable $C^*$-algebras with $A$ unital.   Let $\mathcal{X}$ be an admissible subset of $A$ in the sense of Definition \ref{admiss}.  Assume moreover that $A$ admits a unitally absorbing and residually finite-dimensional representation on a Hilbert $B$-module $E$, and that $A$ has Kirchberg's lifting property.   Let $X$ be a finite subset of $\mathcal{X}$ and let $\epsilon\in (0,1]$.  Then there exists a finite subset $Y$ of $\mathcal{X}$ and $\delta>0$ such that the set $\lin_\delta(Y,B)$ of Definition \ref{non ucp} has the following properties:
\begin{enumerate}[(i)]
\item \label{into hom lp} for any $\phi\in \lin_\delta(Y,B)$ there exists an $(X,\epsilon)$-$*$-homomorphism $\psi:A\to \LL(F_\psi)$ from $A$ to the adjointable operators on a compact $B$-module, and a unital $*$-homomorphism $\pi:A\to \LL(F_\phi\oplus F_\psi)$ such that 
$$
\|\phi(a)\oplus \psi(a)-\pi(a)\|<3\epsilon^{1/2}
$$
for all $a\in X$.
\item \label{stab good lp} for any $\phi,\psi\in \lin_{\delta}(Y,B)$, if we have that $[\phi]=[\psi]$ in $KK^{\lin}_\delta(Y,B)$, then there exists a $*$-homomorphism $\pi:A\to \LL(F_\pi)$ from $A$ to the adjointable operators on a compact $B$-module, and a unitary isomorphism $u:F_\phi\oplus F_\pi\to F_\psi\oplus F_\pi$ such that 
$$
\|u(\phi(a)\oplus \pi(a))u^*-(\psi(a)\oplus \pi(a))\|<\epsilon
$$
for all $a\in X$.
\end{enumerate}
\end{corollary}

\begin{proof}
We just look at part \eqref{stab good lp}; part \eqref{into hom lp} is similar (and simpler).  Let $(X,\epsilon)$ be as in the statement.  Lemma \ref{lp lem} and Corollary \ref{lin to ucp} give us a finite subset $Y$ of $\mathcal{X}$ and $\delta>0$ with the following property: if $\phi$ is an element of $\lin_\delta(Y,B)$ then there exists a ucp map $\phi':A\to \K(F_\phi)$ such that $\|\phi(a)-\phi'(a)\|<(\epsilon/18)^2<\epsilon/3$ for all $a\in X$, and moreover such that there is a well-defined homomorphism 
$$
KK^{\lin}_{\delta}(Y,B)\to KK^{ucp}_{(\epsilon/18)^2}(X,B)
$$
sending the class of $\phi$ to the class of any $\phi'$ approximating it suitably well is a well-defined homomorphism.

Apply this to maps $\phi,\psi\in \lin_\delta(Y,B)$ as in the statement to get ucp maps $\phi':A\to\K(F_{\phi})$ and $\psi':A\to \K(F_\psi)$ such that $\|\phi(a)-\phi'(a)\|<\epsilon/3$ for all $a\in X$ (and similarly for $\psi$ and $\psi'$), and so that $[\phi']=[\psi']$ in $KK^{ucp}_{(\epsilon/18)^2}(X,B)$.  Using Corollary \ref{stab cor}, there is then a unital $*$-homomorphism $\pi:A\to \K(F_\pi)$ with $F_\pi$ a compact $B$-module and a unitary $u:F_\phi\oplus F_\pi\to F_\psi\oplus F_\pi$ such that 
$$
\|u(\phi'(a)\oplus \pi(a))u^*-(\psi'(a)\oplus \pi(a))\|<\epsilon/3
$$
for all $a\in X$.  As $\|\phi'(a)-\phi(a)\|<\epsilon/3$ for all $a\in X$ (and similarly for $\psi$ and $\psi'$), the result follows.
\end{proof}

\begin{remark}\label{other reps rem}
We are mainly interested in the case that $B=\C$.  In this case, we may further restrict the properties of $\psi$.  Indeed, let $\mathcal{R}$ be any family of finite-dimensional representations of $A$ that is closed under isomorphism and finite direct sums, and that separates points (i.e.\ for all $a\in A$ there is $\pi\in \mathcal{R}$ with $\pi(a)\neq 0$).  Then (under the same assumptions, and that $B=\C$) Corollaries \ref{stab cor} and \ref{stab lin cor} hold with the additional stipulation that we may assume that the representation $\pi$ comes from the class $\mathcal{R}$.  The proof is essentially similar: the key point is that the assumptions on $\mathcal{R}$ imply that there exists a strongly absorbing representation $\pi:A\to \LL(H)$ of $A$ on a Hilbert space and an increasing sequence of finite rank projections $(e_n)$ that commute with $A$ and converge strongly to the identity, and moreover are such that compression by $e_n$ defines a representation $A\to \LL(e_nH)$ is in the class $\mathcal{R}$. 
\end{remark}

\section{Action on $K$-theory}\label{pair sec}

In this section, we discuss pairings with $K$-theory.  Our goal is to construct actions 
$$
\lim_{\leftarrow} KK_\epsilon^{ucp}(X,B)\to \text{Hom}(K_0(A\otimes C),K_0(B\otimes C))
$$
and
$$
\lim_{\leftarrow} KK_\epsilon^{\lin}(X,B)\to \text{Hom}(K_0(A\otimes C),K_0(B\otimes C))
$$
of our inverse limit groups on $K$-theory, and discuss how these are compatible with the usual action of $KL(A,B)$ on $K$-theory in the cases that we have access to the isomorphisms of Theorems \ref{main ucp} and \ref{main no ucp}.

The definitions and results in this section are closely related to those \cite[Section 3.2]{Willett:2024aa}: as such we refer to that reference for some details.  For the next definition, recall the notion of a $\delta$-quasiprojection from Definition \ref{delta qp} above.

\begin{definition}\label{x eps com}
Let $A$ and $C$ be $C^*$-algebras with $A$ unital, let $X$ be a finite subset of the unit ball of $A$, and let $\epsilon>0$.  A projection $p\in A\otimes C$ is \emph{ucp $(X,\epsilon)$-compatible} if whenever $\phi:A\to B$ is a ucp $(X,\epsilon)$-$*$-homomorphism to another $C^*$-algebra $B$ in the sense of Definition \ref{alm mult def}, we have that $(\phi\otimes\text{id}_C)(p)$ is a $1/4$-quasiprojection.\footnote{The choice of $1/4$ is not important here: any number in $(0,1/2)$ would work equally well for our purposes.}

Write $\text{Proj}^{X,\epsilon}_{ucp}(A\otimes C)$ for the set of ucp $(X,\epsilon)$-compatible projections in $A\otimes C$.  
\end{definition}

\begin{definition}\label{x eps com 2}
Let $A$ be a unital $C^*$-algebra, let $X$ be a finite self-adjoint subset of the unit ball of $A$, and let $\epsilon>0$.  Let $C$ be a unital $C^*$-algebra.  With notation as in Definition \ref{x eps com}, define
$$
\mathcal{Q}^\epsilon_{ucp}(X;C):=\bigsqcup_{k=1}^\infty \text{Proj}^{X,\epsilon}_{ucp}(A\otimes M_k(\C)\otimes C).
$$
Let $\sim$ be the equivalence relation on $\mathcal{Q}_{ucp}^\epsilon(X)$ generated by the conditions below: 
\begin{enumerate}[(i)]
\item $p\sim q$ if both are in the same path component of some $\text{Proj}^{X,\epsilon}_{ucp}(A\otimes M_k(\C)\otimes C)$;
\item $p\sim q$ if $p\in \text{Proj}^{X,\epsilon}_{ucp}(A\otimes M_k(\C)\otimes C)$ and there is $m\geq 1$ such that $q\in \text{Proj}^{X,\epsilon}_{ucp}(A\otimes M_{k+m}(\C)\otimes C)$ and $q=p\oplus 0_m$.
\end{enumerate}
Define $\mathcal{V}^\epsilon_{ucp}(X;C)$ to be $\mathcal{Q}^\epsilon_{ucp}(X;C)/\sim$, and define $K^\epsilon_{ucp}(X;C)$ to be the Grothendieck group of $\mathcal{V}^\epsilon_{ucp}(X;C)$.
%
\end{definition}

\begin{lemma}\label{func lem}
Let $A$ and $C$ be unital $C^*$-algebras, let $X$ be a finite subset of the unit ball of $A$, and let $\epsilon>0$.  Let $\phi:A\to B$ be a ucp $(X,\epsilon)$-$*$-homomorphism as in Definition \ref{alm mult def}.  Let $\chi$ be the characteristic function of $(1/2,\infty)$.  

For any ucp $(X,\epsilon)$-compatible projection $p\in \text{Proj}^{X,\epsilon}_{ucp}(A\otimes M_k(\C)\otimes C)$, define 
$$
\phi_*(p):=\chi((\phi\otimes \text{id}_{M_k(\C)\otimes C})(p))\in B\otimes M_k(\C)\otimes C.
$$
Then the assignment $[p]\mapsto [\phi_*(p)]$ uniquely determines a homomorphism 
$$
\phi_*:K_{ucp}^\epsilon(X;C)\to K_0(B\otimes C). \eqno\qed
$$
\end{lemma}

We recall now the directed set $(I_A,\leq_{ucp})$ from Remark \ref{fc3}.  Essentially by definition, whenever $(X,\epsilon)\leq_{ucp} (Y,\delta)$ in $I_A$, there is a canonical inclusion $\mathcal{Q}^\epsilon_{ucp}(X;C)\to \mathcal{Q}^\delta_{ucp}(Y;C)$, which induces a group homomorphism
$$
K_{ucp}^\epsilon(X;C)\to K_0^\delta(Y;C).
$$
These homomorphisms make the collection $(K_{ucp}^\epsilon(X;C))_{(X,\epsilon)\in I_A}$ of abelian groups into a directed system.  Moreover, for each $(X,\epsilon)$ there are canonical homomorphisms
\begin{equation}\label{dir lim maps 0}
K_{ucp}^\epsilon(X;C)\to K_0(A\otimes C)
\end{equation}
induced by the identity inclusions of ucp $(X,\epsilon)$-compatible projections into all projections.  These fit together to give maps on the direct limit
\begin{equation}\label{dir lim maps}
\lim_{\rightarrow}K_{ucp}^\epsilon(X;C)\to K_0(A\otimes C).
\end{equation}
This homomorphism is in fact an isomorphism.  We leave the proof of the next lemma to the reader: it consists of direct checks analogous to those used to show that $K$-theory commutes with direct limits of $C^*$-algebras.

\begin{lemma}\label{dir lim k thy}
Let $A$ be a unital $C^*$-algebra.  Then for any $C^*$-algebra $C$, the map of line \eqref{dir lim maps} induces a canonical isomorphisms
$$
\lim_{\rightarrow} K_{ucp}^\epsilon(X;C)\cong  K_0(A\otimes C). \eqno\qed
$$
\end{lemma}

The proof of the next lemma again consists of direct checks that we leave to the reader.

\begin{lemma}\label{phi lem} 
Let $A$ be a unital $C^*$-algebra, let $X$ be a finite subset of the unit ball of $A$, and let $\epsilon>0$.  With notation as in Lemma \ref{func lem}, there is a well-defined homomorphism 
$$
\Phi^{(X,\epsilon)}:KK^{ucp}_\epsilon(X,B)\to \text{Hom}(K_{ucp}^\epsilon(X;C),K_0(B\otimes C))
$$
induced by the assignment 
$$
KK^{ucp}_\epsilon(X,B)\owns [\phi]\mapsto \phi_*.
$$
Applying Lemma \ref{dir lim k thy} these homomorphisms induce a homomorphism
$$
\Phi:{\displaystyle \lim_{\leftarrow}KK^{ucp}_\epsilon(X,B)}\to \text{Hom}({K}_0(A\otimes C),K_0(B\otimes C)). \eqno\qed
$$
\end{lemma}

The proof of the following result is very similar to that of \cite[Theorem 3.29]{Willett:2024aa}: we leave it to the reader.

\begin{theorem}\label{kl ck hom}
Let $A$ be a separable, unital, quasidiagonal $C^*$-algebra, let $B$ be a separable nuclear $C^*$-algebra with an approximate unit consisting of projections, and let $C$ be a unital $C^*$-algebra.  Then the following diagram commutes
$$
\xymatrix{ KL(A,B) \ar[r] \ar[dr]_-{\gamma} & {\displaystyle \lim_{\leftarrow} KK_\epsilon^{ucp}(X,B)} \ar[d] \\ & \text{Hom}(K_0(A;C),K_0(B;C))},
$$
where the diagonal arrow denotes the usual pairing $\gamma$ between $KK$-theory and $K$-theory, the vertical arrow denotes the homomorphism of Lemma \ref{phi lem}, and the horizontal arrow denotes the isomorphism of Theorem \ref{main ucp}. \qed
\end{theorem}

\begin{remark}\label{lin pair}
Let us spell out what this result says slightly more concretely.  Let $A$, $B$, and $C$ be as in the assumptions of Theorem \ref{kl ck hom}, let $x\in KL(A,B)$ be a class, and let $[p]\in K_0(A\otimes C)$ be a class with $p$ a projection in $A\otimes M_k(\C)\otimes C$ for some $k$.  Say we want to compute the usual action $x_*[p]$ of $x$ on $[p]$ as an element of $K_0(B\otimes C)$.  Note that there exists $(X,\epsilon)\in I_A$ such that $p$ is in the image of the canonical map $K_{ucp}^\epsilon(X;C)\to K_0(A\otimes C)$.  On the other hand, under the canonical map $KL(A,B)\to KK_\epsilon^{ucp}(X,B)$, $x$ maps to a formal difference $[\phi]-[\psi]$ of ucp $(X,\epsilon)$-$*$-homomorphisms.  Then with notation as in Lemma \ref{func lem} (so in particular, $\chi$ is the characteristic function of $(1/2,\infty)$, Theorem \ref{kl ck hom} implies that we have 
\begin{equation}\label{conc pair}
\langle x,[p]\rangle:=[\chi((\phi\otimes \text{id}_{M_k(\C)\otimes C})(p))]-[\chi((\psi\otimes \text{id}_{M_k(\C)\otimes C})(p))]\in K_0(B\otimes C).
\end{equation}
\end{remark}

Let us now look at a similar result for $KK_\epsilon^{\lin}(X,B)$.  This is built in a similar way, but the details are a little more technical.

\begin{definition}\label{x eps com 3}
Let $A$ and $C$ be unital $C^*$-algebras.  Let $\mathcal{X}$ be an admissible subset of $A$ in the sense of Definition \ref{admiss}, and let $\mathcal{A}$ be its span, a dense $*$-subalgebra of $A$.  Let $\mathcal{A}\odot C$ denote the algebraic tensor product (over $\C$) of $\mathcal{A}$ and $C$, equipped with the norm it inherits from $A\otimes_{\max}C$ (but not completed).  Let $X$ be a finite subset of $\mathcal{X}$, and let $\epsilon>0$.   For an element $d$ in a $C^*$-algebra $d$, write $\text{Re}(d):=\frac{1}{2}(d+d^*)$ for its real part.

 A $1/4$-quasiprojection $a\in \mathcal{A}\odot C$ is \emph{$(X,\epsilon)$-compatible} if whenever $\phi:\mathcal{A}\to B$ is an $(X,\epsilon)$-$*$-homomorphism as in Definition \ref{alm mult def} to another $C^*$-algebra, we have that $\text{Re}((\phi\otimes \text{id}_C)(a))$ is a $1/4$-quasiprojection\footnote{Analogously to Definition \ref{x eps com}, the precise choice of constant $1/4$ is not important here.} in $B\odot C$, where this $*$-algebra is equipped with the maximal norm. 
 
Write $\text{Proj}^{X,\epsilon}_{\lin}(A\otimes C)$ for the set of $(X,\epsilon)$-compatible projections in $\mathcal{A}\odot C$.  
\end{definition}

\begin{definition}\label{x eps com 4}
Let $A$ be a unital $C^*$-algebra.  Let $\mathcal{X}$ be an admissible subset of $A$ in the sense of Definition \ref{admiss} with span $\mathcal{A}$, let $X$ be a finite subset of $\mathcal{X}$, and let $\epsilon>0$.  Let $C$ be a unital $C^*$-algebra.  Define
$$
\mathcal{Q}^\epsilon_{\lin}(X;C):=\bigsqcup_{k=1}^\infty \text{Proj}^{X,\epsilon}_{\lin}(\mathcal{A}\odot M_k(\C)\odot C).
$$
Let $\sim$ be the equivalence relation on $\mathcal{Q}^\epsilon_{\lin}(X)$ generated by the conditions below: 
\begin{enumerate}[(i)]
\item $p\sim q$ if both are in the same path component of some $\text{Proj}^{X,\epsilon}_{\lin}(\mathcal{A}\odot M_k(\C)\odot C)$;
\item $p\sim q$ if $p\in \text{Proj}^{X,\epsilon}_{\lin}(\mathcal{A}\odot M_k(\C)\odot C)$ and there is $m\geq 1$ such that $q\in \text{Proj}_{\lin}^{X,\epsilon}(\mathcal{A}\odot M_{k+m}(\C)\odot C)$ and $q=p\oplus 0_m$.
\end{enumerate}
Define $\mathcal{V}^\epsilon_{\lin}(X;C)$ to be $\mathcal{Q}^\epsilon_{\lin}(X;C)/\sim$, and define $K^\epsilon_{\lin}(X;C)$ to be the Grothendieck group of $\mathcal{V}^\epsilon_{\lin}(X:C)$.
%
\end{definition}

\begin{lemma}\label{func lem 2}
Let $A$ and $C$ be unital $C^*$-algebras.  Let $\mathcal{X}$ be an admissible subset of $A$ in the sense of Definition \ref{admiss} with span $\mathcal{A}$, let $X$ be a finite subset of $\mathcal{X}$, and let $\epsilon>0$.  Let $\phi:\mathcal{A}\to B$ be an $(X,\epsilon)$-$*$-homomorphism as in Definition \ref{alm mult def}.  Let $\chi$ be the characteristic function of $(1/2,\infty)$.  

For any $(X,\epsilon)$-compatible quasiprojection $a\in \text{Proj}^{X,\epsilon}_{\lin}(\mathcal{A}\odot M_k(\C)\odot C)$, define 
$$
\phi_*(a):=\chi(\text{Re}((\phi\otimes \text{id}_{M_k(\C)\otimes C})(a)))\in (B\otimes M_k(\C))\otimes_{\max} C.
$$
Then the assignment $[p]\mapsto [\phi_*(p)]$ uniquely determines a homomorphism 
$$
\phi_*:K_{\lin}^\epsilon(X;C)\to K_0(B\otimes_{\max} C). \eqno\qed
$$
\end{lemma}

We recall now the directed set $(I_\mathcal{X},\leq)$ from Definition \ref{order2}.  Essentially by definition, whenever $(X,\epsilon)\leq (Y,\delta)$ in $I_{\mathcal{X}}$, there is a canonical inclusion $\mathcal{Q}^\epsilon_{\lin}(X;C)\to \mathcal{Q}^\delta_{\lin}(Y;C)$, which induces a group homomorphism
$$
K_{\lin}^\epsilon(X;C)\to K_{\lin}^\delta(Y;C).
$$
These homomorphisms make the collection $(K_{\lin}^\epsilon(X;C))_{(X,\epsilon)\in I_{\mathcal{X}}}$ of abelian groups into a directed system.  Moreover, for each $(X,\epsilon)$ there are canonical homomorphisms
\begin{equation}\label{dir lim maps 0.2}
K_{\lin}^\epsilon(X;C)\to K_0(A\otimes_{\max} C)
\end{equation}
induced by sending the class of an $(X,\epsilon)$-compatible quasiprojections to the class of the projection $\chi(a)$, where $\chi$ is the characteristic function of $(1/2,\infty)$.  These fit together to give maps on the direct limit
\begin{equation}\label{dir lim maps 2}
\lim_{\rightarrow}K_{\lin}^\epsilon(X;C)\to K_0(A\otimes_{\max} C).
\end{equation}
As for the homomorphism in line \eqref{dir lim maps}, this is an isomorphism.  The key point in the proof is the following technical lemma below.

\begin{lemma}\label{tech lin}
Let $A$ be a separable unital $C^*$-algebra.  Let $\mathcal{X}$ be an admissible subset of $A$ as in Definition \ref{admiss}, and let $\mathcal{A}$ be the span of $\mathcal{X}$ (a dense $*$-subalgebra of $A$).  Let $C$ be a unital $C^*$-algebra, and let  $a\in \mathcal{A}\odot C$ be a $1/4$-quasiprojection.  Then there exists a finite set $Y$ of $\mathcal{X}$ and $\delta>0$ with the following property.

For any $C^*$-algebra $B$ and any $(Y,\delta)$-$*$-homomorphism $\phi:\mathcal{A}\to B$ such that $\|\phi(x)\|\leq 1$ for all $x\in \mathcal{X}$, the element
$$
\text{Re}((\phi\otimes\text{id}_C)(a))\in B\odot C\subseteq B\otimes_{\max} C
$$
is a $1/4$-quasiprojection.
\end{lemma}

\begin{proof}
Assume otherwise.  Then there exist an increasing sequence of finite subsets $Y_1\subseteq Y_2\subseteq \cdots$ of $\mathcal{X}$ with union all of $\mathcal{X}$ and a decreasing sequence $\delta_1>\delta_2>\cdots$ of positive numbers that converge to zero such that for each $n$ there exists a $(Y_n,\delta_n)$-$*$-homomorphism $\phi_n:\mathcal{A}\to B$ such that $\text{Re}((\phi_n\otimes\text{id}_C)(a))$ is not a $1/4$-quasiprojection.

Define $D:=\prod_n (B\otimes_{\max} C)$ be the $C^*$-algebra of bounded sequences in $B\otimes_{\max} C$, and let $J$ be the ideal in $D$ consisting of sequences that tend to zero as $n\to\infty$.  Consider the map 
$$
\Phi:\mathcal{A}\odot C\to D/J,\quad a\mapsto ((\phi_n\otimes \text{id}_C)(a))_{n=1}^\infty
$$
(note that for any $a\in \mathcal{A}\odot C$ the sequence $((\phi_n\otimes \text{id}_C)(a))_{n=1}^\infty$ is bounded by the condition that $\phi_n$ takes $\mathcal{X}$ into the unit ball of $B$). 

The conditions on the $\phi_n$ ensure that $\Phi$ is a $*$-homomorphism.  Using admissibility of $\mathcal{X}$, $\Phi$ extends (uniquely) to a $*$-homomorphism $\Phi:A\odot C\to D/J$; using the universal property of the maximal tensor product, this in turn extends (uniquely) to a $*$-homomorphism $\Phi:A\otimes_{\max} C\to D/J$. 

Consider now the element $b:=\Phi(a)\in D/J$, which is a $1/4$-quasiprojection.  Note that $b$ being a $1/4$-quasiprojection is equivalent to the statement that $b-b^2\leq  3/16-\epsilon$ for some $\epsilon>0$ by the polynomial spectral mapping theorem (and algebra - note that $(1/4)-(1/4)^2=3/64$).  Hence there exists $c\in D/J$ such that $7/64-\epsilon-b+b^2=c^*c$.  Let now $(b_n)$ be a self-adjoint lift of $b$, and let $(c_n)$ be any lift of $c$.  We then have that 
$$
3/16-\epsilon-b_n+b_n^2-c_n^*c_n\to 0\quad \text{as}\quad n\to\infty. 
$$
Hence for all $n$ suitably large, $b_n-b_n^2\leq 3/16-\epsilon/2$, and so $b_n$ is a $1/4$-quasiprojection for all $n$ large enough.  On the other hand, as the sequences $((\phi_n\otimes \text{id}_C)(a))$ and $(b_n)$ in $D$ both map to $\Phi(a)$ in $D/J$, we have that 
$$
(\phi_n\otimes \text{id}_C)(a)-b_n\to 0 \quad \text{as}\quad n\to\infty;
$$
as $b$ is self-adjoint, this implies moreover that 
$$
\text{Re}((\phi_n\otimes \text{id}_C)(a))-b_n\to 0  \quad \text{as}\quad n\to\infty.
$$
As $b_n$ is a $1/4$-quasiprojection, this implies that $\text{Re}((\phi_n\otimes \text{id}_C)(a))$ is a $1/4$-quasiprojection for all large $n$, giving the desired contradiction.
\end{proof} 

Note that Lemma \ref{tech lin} implies that (with notation as in the lemma) any projection $A\otimes_{\max}M_k(\C)\otimes_{\max}(C)$ can be approximated by a $1/4$-quasi-projection in $\text{Proj}^{Y,\delta}_{\lin}(\mathcal{A}\odot M_k(\C)\odot C)$ for some appropriate finite subset $Y$ of $\mathcal{X}$ and $\delta>0$; the same applies to homotopies of projections.  Having made this observation, the next lemma can be proved along the same lines as the continuity of $K$-theory under direct limits; we leave the details to the reader.

\begin{lemma}\label{dir lim k thy 2}
Let $A$ be a unital $C^*$-algebra.  Then for any $C^*$-algebra $C$, the map of line \eqref{dir lim maps} induces a canonical isomorphism
$$
\lim_{\rightarrow} K_{\lin}^\epsilon(X;C)\cong  K_0(A\otimes C).  \eqno\qed
$$
\end{lemma}

The proof of the next lemma again consists of direct checks that we leave to the reader.

\begin{lemma}\label{phi lem 2} 
Let $A$ and $C$ be unital $C^*$-algebras.  Let $\mathcal{X}$ be an admissible subset of $A$ in the sense of Definition \ref{admiss} with span $\mathcal{A}$, let $X$ be a finite subset of $\mathcal{X}$, and let $\epsilon>0$.  With notation as in Lemma \ref{func lem 2}, there is a well-defined homomorphism 
$$
\Phi^{(X,\epsilon)}:KK^{\lin}_\epsilon(X,B)\to \text{Hom}(K_{\lin}^\epsilon(X;C),K_0(B\otimes C))
$$
induced by the assignment 
$$
KK^{\lin}_\epsilon(X,B)\owns \phi\mapsto \phi_*.
$$
Applying Lemma \ref{dir lim k thy} these homomorphisms induce a homomorphism
$$
\Phi:{\displaystyle \lim_{\leftarrow}KK^{\lin}_\epsilon(X,B)}\to \text{Hom}({K}_0(A\otimes C),K_0(B\otimes C)). \eqno\qed
$$
\end{lemma}

The proof of the following result is very similar to that of \cite[Theorem 3.29]{Willett:2024aa}: we leave it to the reader.

\begin{theorem}\label{kl ck hom 2}
Let $A$ be a separable, unital, quasidiagonal $C^*$-algebra with the LLP, let $B$ be a separable nuclear $C^*$-algebra with an approximate unit consisting of projections, and let $C$ be a unital nuclear $C^*$-algebra.  Then the following diagram commutes
$$
\xymatrix{ KL(A,B) \ar[r] \ar[dr]_-{\gamma} & {\displaystyle \lim_{\leftarrow} KK_\epsilon^{ucp}(X,B)} \ar[d] \\ & \text{Hom}(K_0(A;C),K_0(B;C))}.
$$
where the diagonal arrow denotes the usual pairing $\gamma$ between $KK$-theory and $K$-theory, the vertical arrow denotes the homomorphism of Lemma \ref{phi lem}, and the horizontal arrow denotes the isomorphism of Theorem \ref{main ucp}. \qed
\end{theorem}

\begin{remark}\label{lin pair 2}
Let us spell out what this results says slightly more concretely.  Let $A$, $B$, and $C$ be as in the assumptions of Theorem \ref{kl ck hom}, let $x\in KL(A,B)$ be a class, and let $[p]\in K_0(A\otimes C)$ be a class with $p$ a projection in $A\otimes M_k(\C)\otimes C$ for some $k$.  Say we want to compute the usual action $x_*[p]$ of $x$ on $[p]$ as an element of $K_0(B\otimes C)$.  Note that there exists $(X,\epsilon)\in I_{\mathcal{X}}$ such that $p$ is the image of the class $[a]$ of a $1/4$-quasiprojection $a$ under the canonical map $K_{\lin}^\epsilon(X;C)\to K_0(A\otimes C)$.  On the other hand, under the canonical map $KL(A,B)\to KK^{\lin}_\epsilon(X,B)$, $x$ maps to a formal difference $[\phi]-[\psi]$ of $(X,\epsilon)$-$*$-homomorphisms.  Then with notation as in Lemma \ref{func lem} (so in particular, $\chi$ is the characteristic function of $(1/2,\infty)$), Theorem \ref{kl ck hom} implies that we have 
\begin{equation}\label{conc pair}
\langle x,[p]\rangle:=[\chi(\text{Re}(\phi\otimes \text{id}_{M_k(\C)\otimes C})(a))]-[\chi(\text{Re}(\psi\otimes \text{id}_{M_k(\C)\otimes C})(a))]\in K_0(B\otimes C).
\end{equation}
\end{remark}



\section{Back to approximate representations and vector bundles}\label{qr finish}

In this section, we finally tie our abstract $KK$-theoretic results back to the coassembly map of Section \ref{qr bun sec} and establish Theorem \ref{intro main}.

We first identify the group $QR(\Gamma)$ of Definition \ref{inv lim approx} with the $KL$-homology group of the (maximal) group $C^*$-algebra. The key technical ingredient is given below.

\begin{proposition}\label{gp lin}
Let $\Gamma$ be a countable group, and let $I_\Gamma$ be as in Definition \ref{gp order}.  Let $\mathcal{X}$ equal $\Gamma$, considered as an admissible subset of the maximal group $C^*$-algebra as in Definition \ref{admiss}.  Then with notation as in Definitions \ref{gp gp} and \ref{non ucp}, for each $(S,\epsilon)\in I_\Gamma$ there are homomorphisms
$$
\phi_{S,\epsilon}:KK^{\lin}_{\epsilon/10}(S\cup S^{-1},\C)\to K_\epsilon(S)\quad \text{and}\quad \psi_{S,\epsilon}:K_\epsilon((S\cup S^{-1})^2)\to KK^{\lin}_\epsilon(S,\C)
$$
that fit into commutative diagrams
$$
\xymatrix{ KK^{\lin}_{\epsilon/10}((S\cup S^{-1})^2,\C) \ar[dr]_-{\phi_{(S\cup S^{-1})^2,\epsilon}} \ar[rr] & & KK^\lin_\epsilon(S,\C) \\
& K_\epsilon((S\cup S^{-1})^2) \ar[ur]_-{\psi_{S,\epsilon}} & }
$$
and 
$$
\xymatrix{ & KK^\lin_{\epsilon/10}(S\cup S^{-1},\C) \ar[dr]^-{\phi_{S,\epsilon}} & \\
K_{\epsilon/10}((S\cup S^{-1})^2) \ar[rr]\ar[ur]^-{\psi_{S\cup S^{-1},\epsilon/10}~~~} & & K_\epsilon(S) }.
$$
In particular, taking the inverse limits over the directed set $I_\Gamma$ of Definition \ref{gp order} and with notation as in Definition \ref{inv lim approx}, we have a canonical isomorphism 
$$
QR(\Gamma)\cong \lim_{\leftarrow}KK^{\lin}_\epsilon(S,\C).
$$
\end{proposition}

We first need a lemma.

\begin{lemma}\label{uni close}
Let $x$ and $y$ be positive contractions in a unital $C^*$-algebra and $\delta>0$ be such that $x\geq 1/2$ and $y\geq 1/2$.  Let $\epsilon>0$, and let $u\in C$ be a unitary element such that $\|uxu^*-y\|<\epsilon$.  Then 
$$
\|u(x(x^*x)^{-1/2})u^*-y(y^*y)^{-1/2}\|<4\epsilon.
$$
\end{lemma}

\begin{proof}
We first note that 
$$
\|u(x^*x)u^*-y^*y\|\leq \|(ux^*u^*-y^*)uxu^*\|+\|y^*(uxu^*-y)\|<2\epsilon.
$$
Hence using Lemma \ref{-1/2 com}, part \eqref{mh lip}, we have that 
$$
\|u(x^*x)^{-1/2}u^*-(y^*y)^{-1/2}\|\leq \frac{2\epsilon}{2(1/2)^{3/2})}<3\epsilon.
$$
From this we see that 
\begin{align*}
\|u&(x(x^*x)^{-1/2})u^*-y(y^*y)^{-1/2}\| \\&\leq \|(uxu^*-y)u(x^*x)^{-1/2}u^*\|+\|y(u(x^*x)^{-1/2}u^*-(y^*y)^{-1/2}\|<4\epsilon
\end{align*}
as claimed.
\end{proof}

\begin{proof}[Proof of Proposition \ref{gp lin}]
Recall from Definition \ref{non ucp} that $KK_\epsilon^{\lin}(S,\C)$ is built as follows.  Cycles for $KK_\epsilon(S,\C)$ are linear maps $\psi:\C[\Gamma]\to \K(F)$ where $F$ is a compact Hilbert $\C$-module as in Definition \ref{com hm}, and $\psi$ satisfies certain estimates; one then looks at such cycles modulo approximate stable unitary equivalence, and takes the Grothendieck group.  Note that a compact Hilbert $\C$-module is just a copy of $\C^n$ for some $n$, whence we may just look at linear maps $\psi:\C[\Gamma]\to M_n(\C)$ for some $n$; moreover as $\psi$ is linear, it is uniquely determined by its restriction to $\Gamma$.  Having made these adjustments, the estimates $\psi$ has to satisfy can be written 
\begin{equation}\label{psi cons}
\psi(1_\Gamma)=1_{M_n(\C)}, \quad \|\psi(s^{-1})-\psi(s)^*\|<\epsilon,\quad \|\psi(st)-\psi(s)\psi(t)\|<\epsilon \quad \forall s,t\in S.
\end{equation}

Now, let $(S,\epsilon)$ be given, and let $\theta:(S\cup S^{-1})^2\to U(\C^n)$ be a cycle for $K_\epsilon((S\cup S^{-1})^2)$ (see Definition \ref{gp gp}).  Extend $\theta$ to all of $\Gamma$ by defining $\theta(g)=1_{M_n(\C)}$ for all $g\in \Gamma\setminus (S\cup S^{-1})^2$, and then extend to a map $\psi:\C[\Gamma]\to M_n(\C)$ by linearity.  Note that $\psi(1_\Gamma)=1_{M_n(\C)}$ and for all $s,t\in S$, $\|\psi(st)-\psi(s)\psi(t)\|<\epsilon$ by definition (note that if $s,t\in S$, then $st\in S^2$).  Moreover, for any $s\in S$, 
$$
\|\psi(s)^*-\psi(s^{-1})\|=\|1_{M_n(\C)}-\psi(s)\psi(s^{-1})\|=\|\psi(ss^{-1})-\psi(s)\psi(s)^*\|<\epsilon
$$
(the first equality uses that $\psi(s)$ is unitary).  Hence $\psi$ satisfies the conditions in line \eqref{psi cons}.  It follows from this that the map on cycles given by the above extension process gives a well-defined homomorphism 
$$
\psi_{S,\epsilon}:K_\epsilon((S\cup S^{-1})^2)\to KK_\epsilon^{\lin}(S,\C),\quad [\theta]\mapsto [\psi]
$$
as in the statement.

On the other hand, let $S$ and $\epsilon\in (0,1]$ be given.   Assume that $\psi$ is a cycle for $KK_{\delta}^{\lin}(S\cup S^{-1},\C)$.  For each $s\in S\cup S^{-1}$, the analogues of the estimates in line \eqref{psi cons} imply that
$$
\|\psi(s)^*\psi(s)-1_{M_n(\C)}\|\leq \|\psi(s)^*-\psi(s^{-1})\|\|\psi(s)\|+\|\psi(s^{-1})\psi(s)-\psi(s^{-1}s)\|<2\delta,
$$
whence $\psi(s)^*\psi(s)$ is invertible, and so $\psi(s)$ is invertible too (recall we are working in a matrix algebra).  Using the $C^*$-identity, the inequality
\begin{equation}\label{psi 1 close}
\|\psi(s)^*\psi(s)-1\|<2\delta<1/2
\end{equation}
also implies that $\|\psi(s)\|<\sqrt{3/2}$.  Using the functional calculus, we also see that $\|1-(\psi(s)^*\psi(s))^{-1/2}\|<\frac{3}{2}\epsilon/5$.  We may therefore define $\theta(s):=\psi(s)(\psi(s)^*\psi(s))^{-1/2}$ and see that  
$$
\|\theta(s)-\psi(s)\|\leq \|\psi(s)\|\|1-(\psi(s)^*\psi(s))^{-1/2}\|<\epsilon/4.
$$
It follows from this that the map defined by $\theta(s)=\psi(s)(\psi(s)^*\psi(s))^{-1/2}$ for $s\in S$ satisfies 
\begin{align*}
\|&\theta(st)-\theta(s)\theta(t)\| \\ & \leq \|\psi(st)-\theta(st)\|+\|(\theta(s)-\psi(s))\theta(t)\|+\|\psi(s)(\psi(t)-\theta(t))\|+\|\psi(st)-\psi(s)\psi(t)\|\\&<\epsilon
\end{align*}
for all $s,t\in S$ such that $st\in S$ also, and so $\theta$ is a cycle for $K_\epsilon(S)$.  It follows from this and Lemma \ref{uni close} that the map
\begin{equation}\label{c* to gp}
KK_{\epsilon/10}^{\lin}(S\cup S^{-1},\C)\to K_\epsilon(S),\quad [\psi]\mapsto [\theta].
\end{equation}
is well-defined.  
 
Having made these definitions commutativity of the triangles in the statement follows directly.  The result on inverse limits follows from this and Lemma \ref{triangle isos}.
\end{proof} 

\begin{theorem}\label{group base}
Let $\Gamma$ be a countable group.  Assume that $C^*\Gamma$ is quasidiagonal and has the LLP.  Then with notation as in Definition \ref{inv lim approx} there exists a canonical isomorphism
$$
QR(\Gamma)\cong KL(C^*\Gamma,\C).
$$
\end{theorem}

\begin{proof}
This follows directly from Theorem \ref{main no ucp} and Proposition \ref{gp lin} on taking inverse limits.
\end{proof}

\begin{remark}
We do not know if the isomorphism of Theorem \ref{group base} preserves multiplication in general; we do not need this for our main results, so did not pursue it seriously.  It is a consequence of Theorem \ref{final the} below that it is at least a ring isomorphism under some additional assumptions.
\end{remark}

We now record some consequences of Theorem \ref{group base} and Corollary \ref{stab lin cor} for the group $QR(\Gamma)$ under additional assumptions on $C^*\Gamma$.  The proof uses very similar ideas to the proof of Theorem \ref{group base} in concert with Corollary \ref{stab lin cor} as well as the idea of Remark \ref{other reps rem}, and is left to the reader.

\begin{corollary}\label{rep form}
Let $\Gamma$ be a countable group.  Assume that $C^*\Gamma$ has the LLP, and assume also that there is a collection $\mathcal{R}$ of finite-dimensional representations of $C^*\Gamma$ that is closed under isomorphisms and finite direct sums, and separates points.  Let $I_\Gamma$ be the directed set of pairs $(S,\epsilon)$ with $S$ a finite subset of $\Gamma$ and $\epsilon>0$ as in Definition \ref{gp order}.  The following hold.
\begin{enumerate}[(i)]
\item \label{rep diff} Any element $x\in QR(\Gamma)$ can be represented as a collection $(x_{S,\epsilon})_{(S,\epsilon)\in I_\Gamma}$ where each $x_{S,\epsilon}\in K_\epsilon(S)$ is of the form $[\phi]-[\pi]$ where $\phi$ is an $(S,\epsilon)$-representation, and $\pi$ is a finite-dimensional unitary representation of $\Gamma$.  
\item \label{zero rep} If $x=([\phi_{S,\epsilon}]-[\psi_{S,\epsilon}])$ represents the zero element in $QR(\Gamma)$ then for each $(S,\epsilon)$ there exists $(T_0,\delta_0)\geq (S_\epsilon)$ such that for any $(T,\delta)\geq (T_0,\delta_0)$ there exists a finite-dimensional unitary representation $\pi:\Gamma\to U(V_\pi)$ and a unitary isomorphism $u:V_{\phi_{T,\delta}}\oplus V_\pi\to V_{\psi_{T,\delta}}\oplus V_\pi$ such that  
$$
\|u(\phi_{T,\delta}(s)\oplus  \pi(s))u^*-\psi_{T,\delta}(s)\oplus \pi(s)\|<\epsilon
$$
for all $s\in S$. \qed
\end{enumerate}
\end{corollary}

We next recall (a slight variant of) Kasparov's coassembly map from \cite[Section 8]{Kasparov:1975ht} (see also \cite[Section 6]{Kasparov:1988dw} for a more fully developed version).  Let $\Gamma$ be a countable discrete group, and let $B\Gamma$ be a CW complex model for its classifying space with (contractible) universal cover $\pi:E\Gamma\to B\Gamma$.  Let $\Gamma$ act on $E\Gamma$ by deck transformations, and on $C^*\Gamma$ by left multiplication\footnote{This should be thought of as an action by automorphisms of $C^*\Gamma$ when the latter is considered as a right Hilbert module over itself; it is not an action by $C^*$-algebra automorphisms.}.  Let $M_\Gamma$ denote the space $E\Gamma\times_\Gamma C^*\Gamma$ defined by quotienting $E\Gamma\times C^*\Gamma$ by the diagonal action; note that $M_\Gamma$ is canonically a locally trivial bundle over $B\Gamma$ with fibre $C^*\Gamma$.  For a finite subcomplex $Y$ of $B\Gamma$ with inclusion map $i_Y:Y\to B\Gamma$, let $M_{\Gamma,Y}:=i_Y^*(M_\Gamma)$ denote the pullback, a locally trivial bundle over $Y$ with fibers $C^*\Gamma$.  A general version of the Serre-Swan theorem (with essentially the same proof as the usual version - see for example \cite[Theorem 6.18]{Karoubi:1978ai}) shows that the space $E_{\Gamma,Y}$ of sections $M_{\Gamma,Y}$ is a finitely generated, projective module over $C(Y)\otimes C^*\Gamma$, and thus defines an element 
\begin{equation}\label{egy}
[E_{\Gamma,Y}]\in K_0(C(Y)\otimes C^*\Gamma)=KK(\C,C(Y)\otimes C^*\Gamma).
\end{equation}
Let us be a little more concrete about $[E_{\Gamma,Y}]$, and define it in terms of a projection in an appropriate matrix algebra over $C(Y)\otimes C^*\Gamma$.  Analogously to our work in Section \ref{qr bun sec}, fix the following data:
\begin{enumerate}[(i)]
\item a locally finite cover $\mathcal{U}$ of $B\Gamma$ and for each $U\in \mathcal{U}$ a $\Gamma$-equivariant homeomorphism $\psi_U:\pi^{-1}(U)\to U\times \Gamma$ such that $\pi(\psi_U^{-1}(x,g))=x$ for all $x\in U$ and $g\in \Gamma$ (here $\Gamma$ equivariance is with respect to the action on $\pi^{-1}(U)$ by deck transformations, and the product action on $U\times \Gamma$ coming from the trivial action on $U$, and the left translation action of $\Gamma$ on itself);
\item a partition of unity $(\chi_U)_{U\in \mathcal{U}}$ on $B\Gamma$ subordinate to $\mathcal{U}$.
\end{enumerate}
For each $U,V\in \mathcal{U}$ with $U\cap V\neq \varnothing$, we then have that the composition 
$$
\psi_V\circ \psi_U^{-1}:(U\cap V)\times \Gamma\to (U\cap V)\times \Gamma
$$
is given by $(x,h)\mapsto (x,hg_{VU})$ for some element $g_{VU}\in \Gamma$.  Now let $U_1,...,U_n$ be an enumeration of the elements of $\mathcal{U}$ that intersect non-trivially with $Y$.  For each $i,j\in \{1,...,n\}$, let $g_{ij}$ be shorthand for $g_{U_iU_j}$, and define 
\begin{equation}\label{pgy}
p_{\Gamma,Y}:=\sum_{i,j=1}^n e_{ij}\otimes g_{ij}\otimes \chi_i^{1/2}\chi_j^{1/2}\in M_n(\C)\otimes C^*\Gamma\otimes C(Y).
\end{equation}
This projection then represents the class $[E_{\Gamma,Y}]$ in $K_0(C(Y)\otimes C^*\Gamma)$.

We may now define the coassembly map as follows.  For a finite subcomplex $Y$ of $B\Gamma$, define 
$$
\mu_Y^*:K^0(C^*\Gamma)\to K^0(Y)
$$
to be the Kasparov product with the class $[E_{\Gamma,Y}]\in KK(\C,C(Y)\otimes C^*\Gamma)$ defined in line \eqref{egy} above.  Using continuity of the Kasparov product (see \cite[Theorem 3.5]{Dadarlat:2005aa}) and the fact that the canonical topology on $K^0(Y)$ is Hausdorff (actually, discrete), $\mu_Y^*$ descends to a map 
$$
\mu_Y^*:KL^0(C^*\Gamma)\to K^0(Y).
$$
These maps are compatible as $Y$ varies, and so (with notation as in Definition \ref{lim kt}) define a map into the inverse limit 
\begin{equation}\label{co ass}
\mu^*:KL(C^*\Gamma,\C)\to LK^0(B\Gamma),
\end{equation}
which we call the \emph{coassembly map}\footnote{This is not quite the same as Kasparov's coassembly map, which is defined in \cite[Section 6]{Kasparov:1988dw} to be a map $\alpha:K^*(C^*\Gamma)\to RK^*(B\Gamma)$ from the $K$-homology of $C^*\Gamma$ to the representable $K$-theory of $B\Gamma$; our coassembly map should be thought of as a ``Hausdorffified version'' of this, and is more appropriate for our purposes.}.  

We now want to study situations under which Kasparov's coassembly map is particularly well-behaved.  The next result follows through a combination of highly non-trivial results; for the sake of brevity, we just give the main points without going into details on background.

\begin{proposition}\label{ca iso}
Let $\Gamma$ be a countable, a-T-menable, torsion-free group.  Then the coassembly map of line \eqref{co ass} is a topological isomorphism.
\end{proposition}

\begin{proof}
We consider a sequence of maps
\begin{align}
KL^*(C^*\Gamma) & \to \text{Hom}_{\Lambda}(\underline{K}_*(C^*\Gamma),\underline{K}_*(\C))  \label{umct1} \\
& \to \text{Hom}_{\Lambda}(\underline{RK}_*(B\Gamma),\underline{K}_*(\C)) \label{ass} \\
& \to \text{Hom}_{\Lambda}(\lim_{\to} \underline{K}_*(Y),\underline{K}_*(\C)) \label{lim} \\
& \to \lim_{\leftarrow}\text{Hom}_{\Lambda}(\underline{K}_*(Y),\underline{K}_*(\C)) \label{lim out} \\
& \to \lim_{\leftarrow} K^*(Y) \label{umct2} \\
& \to K^*_l(B\Gamma) \label{lim k}
\end{align}
that we now explain.  First, all groups with an underline are total $K$-theory or $K$-homology as in the work of Dadarlat-Loring on the universal multicoefficient theorem \cite{Dadarlat:2005aa}, and $\Lambda$ is the ring (or category) of Bockstein operations that acts on such groups.  The group $\underline{RK}_*(B\Gamma)$ is by definition the limit of the total $K$-homology groups $\underline{K}_*(Y)$ as $Y$ varies over all compact subsets of $B\Gamma$.  

The map in line \eqref{umct1} is the map appearing in the universal multicoefficient theorem of Dadarlat-Loring \cite{Dadarlat:2005aa} and is a topological isomorphism by \cite[Theorem 4.1]{Dadarlat:2005aa} and the fact that the UCT holds for $C^*\Gamma$ when $\Gamma$ is a-T-menable (as observed by Tu \cite[Proposition 10.7]{Tu:1999bq} based on work of Higson-Kasparov \cite{Higson:2001eb}).  The map in line \eqref{ass} is induced by taking the Kasparov product with the elements $[E_{\Gamma,Y}]$ of line \eqref{egy} and taking a direct limit; it is an isomorphism by the main theorem of \cite{Higson:2001eb} and the fact that $\Gamma$ is torsion-free (which lets us identify $\underline{RK}(B\Gamma):=\lim_Y \underline{K}_0(Y)$ with the left hand side of the Baum-Connes assembly map, suitably totalized). The map in line \eqref{lim} is induced by the isomorphism ${\displaystyle \underline{RK}_*(B\Gamma)=\lim_{\to} \underline{K}_*(Y)}$, which is true by definition; it is a topological isomorphism.  The map of line \eqref{lim out} is pure algebra, and is a topological isomorphism by the definitions involved.  The map in line \eqref{umct2} is induced by the UCT isomorphism and Spanier-Whitehead duality, and is a topological isomorphism when the right hand side is equipped with the inverse limit topology (here we use that $K^*(Y)$ is finitely generated, and has the discrete topology).  Finally, the map in line \eqref{lim k} is a topological isomorphism by definition.

The composition of all these maps agrees with $\mu^*$: this follows from the definition of $\mu^*$, as this leads to a duality with $\mu_*$ as recorded for example in \cite[Lemma 6.2]{Kasparov:1988dw}. 
\end{proof}

\begin{remark}
From the proof of Proposition \ref{ca iso}, we see that its conclusions hold whenever the Baum-Connes-Kasparov assembly map $\mu_*:\lim_Y K_*(Y)\to K_*(C^*\Gamma)$ is an isomorphism (compare also \cite[Lemma 6.1]{Willett:2024aa} for the relationship between this and the assembly map in total $K$-theory), and $C^*\Gamma$ satisfies the UCT.  These assumptions are known to hold a little more generally than for a-T-menable groups: for example, for $\Gamma=\Z^2\rtimes F_2$, where $F_2$ acts via any homomorphism $F_2\to GL(2,\Z)$ and other similar group extensions (see for example \cite{Chabert:2001hl}).  However, we expect they fail for all infinite property (T) groups, and this is known to be the case for many examples. 
\end{remark}

\begin{theorem}\label{final the}
Assume that $\Gamma$ satisfies the conditions of Theorem \ref{group base} and is in addition torsion-free and a-T-menable.  Then the coassembly map
$$
\alpha:QR(\Gamma)\to LK^0B\Gamma)
$$
from Proposition \ref{coass} is an isomorphism of topological rings.  This isomorphism moreover takes the class of an actual representation (see Remark \ref{act rep rem}) to the class of the associated flat bundle (see Remark \ref{bun rem}). 
\end{theorem}

\begin{proof}
We consider the diagram below
$$
\xymatrix{ QR(\Gamma) \ar[d] \ar[r]^-\alpha & LK^0B\Gamma) \\
KL(C^*\Gamma,\C) \ar[ur]_-{\mu^*} & }
$$
where the vertical arrow is the isomorphism from Theorem \ref{group base}, the horizontal map is the approximate Atiyah-Segal map of Definition \ref{aas}, and the diagonal arrow is the coassembly map of line \eqref{co ass}.  The vertical and diagonal arrows are isomorphisms by Theorem \ref{group base} and Proposition \ref{ca iso} respectively.  Thus it suffices to show that the diagram commutes.  This follows from the formula in line \eqref{pgy}, Remark \ref{lin pair 2} above, and the Definition of $\alpha$, so we are done.
\end{proof}

\section{Torsion examples from crystallographic groups}\label{tor sec}

In this section, we use the result of Theorem \ref{final the} to study some examples in the theory of `weak stability' as studied by Dadarlat \cite{Dadarlat:384aa} and the author \cite{Willett:2024aa}.  We focus on one specific example of non-trivial torsion phenomena, although also briefly explain some generalizations; this also has the side effect of showing how Theorem \ref{final the} can be used to give some concrete computations.

\begin{theorem}\label{tor the}
Let $\Gamma$ be the fundamental group of the product $T\times K$, where $T$ is the $2$-torus, and $K$ is the Klein bottle.  Then there exists a finite subset $S_0$ of $\Gamma$ and $\epsilon_0>0$ such that for any $(S,\epsilon)\geq (S_0,\epsilon_0)$ in the sense of Definition \ref{gp order} there exists an $(S,\epsilon)$-representation $\phi$ of $\Gamma$ such that:
\begin{enumerate}[(i)]
\item \label{far part} for any finite-dimensional unitary representations $\pi$ and $\sigma$ of $\Gamma$,
$$
\|\phi(s)\oplus \pi(s)-\sigma(s)\|\geq \epsilon_0
$$
for all $s\in S_0$;
\item \label{close part} there exists finite-dimensional representations $\pi$ and $\sigma$ of $\Gamma$ such that 
$$
\|\phi(s)\oplus \phi(s)\oplus \pi(s)-\sigma(s)\|< \epsilon
$$
for all $s\in S$.
\end{enumerate}
\end{theorem}

\begin{remark}
Theorem \ref{tor the} gives examples of non-trivial $2$-torsion in some sense.  There is nothing particularly special about $2$-torsion in this context: examples with non-trivial $p$-torsion phenomena can be built in a similar way, for example using the manifolds studied in \cite[Section 3]{Charlap:1965aa} in place of the Klein bottle (plus the Atiyah-Hirzebruch spectral sequence to relate cohomology and $K$-theory).  There is also nothing particularly special about crystallographic groups: one could fir example give non-amenable examples by considering the fundamental group of $S\times T$ where $S$ is a closed non-orientable surface of higher genus, and $T$ is the $2$-torus.  However, we will just keep to the case above as it seems the simplest non-trivial example: note in particular that such torsion examples cannot exist if $B\Gamma$ is at most three-dimensional.
\end{remark}

The key facts underlying Theorem \ref{tor the} are topological.  First, we have a well-known lemma.  For the statement, recall that if $X$ is a compact Hausdorff space equipped with a fixed base point $x_0$, then the reduced $K$-theory $\widetilde{K}^0(X)$ is defined to be the kernel of the map $K^0(X)\to K^0(\{x_0\})$ induced by the inclusion $x_0\in X$.  If moreover $X$ is path connected, then any two choices of basepoint give the same subgroup; we will therefore not mention the choice of basepoint in that case.

\begin{lemma}
Let $X$ be a finite connected CW complex.  Let $E$ and $F$ be flat vector bundles over $X$ of the same dimension.  Then the corresponding class
$$
[E]-[F]\in \widetilde{K}^0(X)
$$
in the reduced $K$-theory of $X$ is torsion.
\end{lemma}

\begin{proof}
The point of the proof is to use Baum-Douglas theory to reduce to the case of manifolds; we can then apply the Atiyah-Singer index theorem and  Chern-Weil theory.  To sketch the details: as $\widetilde{K}^0(X)$ is a finitely generated abelian group, the UCT isomorphism implies that it suffices to show that the pairing of $[E]-[F]$ with any $K$-homology class is zero.  Now, using the Baum-Douglas model of $K$-homology \cite{Baum:1980pt} (see also \cite{Baum:2007ek} for a detailed exposition), any $K$-homology class can be represented by a triple $(M,S,f)$ for an appropriate spin$^\text{c}$ manifold $M$, vector bundle $S$, and continuous map $f:M\to X$.  Moreover by naturality, the pairing between $[M,S,f]$ and $[E]-[F]$ is the same as the pairing between $[M,S,\text{id}_M]$ and $f^*[E]-f^*[F]$, where now we are using the pairing between the $K$-homology and $K$-theory of $M$.  In terms of analytic $K$-homology, the class $[M,S,\text{id}_M]$ is the same as the class $[D_S]$ of a choice of spin$^\text{c}$ Dirac operator on $M$ twisted by $S$ (see for example \cite[Section 6]{Baum:2007ek}).  The pairing of this class with $f^*[E]-f^*[F]$ is then given in terms of index theory by 
$$
\text{Index}(D_{S\otimes f^*(E)})-\text{Index}(D_{S\otimes f^*(F)})
$$
(compare for example \cite[Lemma 11.4.1]{Higson:2000bs}.  These indices can be computed via the Atiyah-Singer index theorem (compare for example the computations in \cite[Theorems 3.10 and D.15]{Jr.:1989uq}): we have that 
$$
\text{Index}(D_{S\otimes f^*(E)})-\text{Index}(D_{S\otimes f^*(F)})=\langle [M],\alpha\cdot \text{ch}(S)\cdot(\text{ch}(f^*E)-\text{ch}(f^*F))\rangle
$$
where: the pairing on the right denotes is that between homology and cohomology (with real coefficients); $[M]$ denotes the homology fundamental class of $M$ (for the orientation induced by the spin$^\text{c}$ structure); $\alpha$ is a cohomology class canonically associated to the spin$^\text{c}$ structure; and $\text{ch}:K^0(M)\to \bigoplus_{k=0}^\infty H^{2k}(M;\R)$ is the (real-valued) Chern character.  

Now, the pullback bundles $f^*E$ and $f^*F$ on $M$ are still flat (and of the same dimension).  We may assume they are smooth, and equipped with a Riemannian connection with zero curvature (see for example \cite[discussion on pages 294-5]{Milnor:1974zf}).  Hence using Chern-Weil theory (see for example \cite[Appendix C]{Milnor:1974zf}), the Chern character of the difference $[f^*E]-[f^*F]$ is zero, completing the argument.
\end{proof}

\begin{definition}\label{flat sg}
Let $X$ be a finite connected CW complex.  We write $\widetilde{K}^0_{\text{fl}}(X)$ for the (finite) subgroup of $K^0(X)$ consisting of formal differences of flat bundles over $X$ of the same dimension. 
\end{definition}

The following result is the main topological input we need: the key point is to produce an aspherical space $X$ such that the torsion subgroup of $K^0(X)$ is not generated by formal differences of flat bundles.  The existence of examples like this is presumably known to some experts, but does not seem to be well-known\footnote{Slightly surprisingly, I tried quite hard to produce one with generative AI, but was not successful.  In the end, this example is just something I made up myself; this was however, around two months ago, and the models have likely improved enough now that AI would have been useful.}.

\begin{proposition}\label{tor vs flat}
Let $M=T\times K$ be the product of the torus and Klein bottle.  Then the reduced $K$-theory of $M$ satisfies  
$$
\widetilde{K}^0(M)\cong \Z\oplus (\Z/2)\oplus (\Z/2).
$$
Moreover, the subgroup $\widetilde{K}^0_{fl}(M)$ of Definition \ref{flat sg} is isomorphic to $\Z/2$. 
\end{proposition}

\begin{proof}
One computes (use for example a CW complex decomposition and the $K$-theory long exact sequence) that the $K$-theory groups of $T$ and $K$ are given respectively by $K^0(T)\cong \Z\oplus \Z$ and $K^0(K)\cong \Z\oplus (\Z/2)$.  The computation of $K^0(M)$, whence also $\widetilde{K}^0(M)$, then follows readily from the K\"{u}nneth formula (see \cite{Atiyah:1962aa} for a topological viewpoint, and \cite{Schochet:1982aa} for a $C^*$-algebraic one).  

We now compute the subgroup generated by flat bundles: equivalently, we compute the image of the augmentation ideal $I(\Gamma)$ in the (finite-dimensional, unitary) representation ring $R(\Gamma)$ under the Atiyah-Segal homomorphism $\alpha:I(\Gamma)\to \widetilde{K}^0(M)$.  To show that this is contained in a single copy of $\Z/2$, it suffices to show that $\alpha(R(\Gamma))$ is contained in a subgroup of the form $\Z[1]\oplus \Z/2$, where $\Z[1]$ is the subgroup of $K^0(M)$ generated by the trivial bundle.  For this, it suffices to consider the image of irreducible representations under $\alpha$, so let $\pi:\Gamma\to U(\C^n)$ be an irreducible representation.  Using Schur's lemma, the restriction of $\pi$ to the central subgroup $\pi_1(T)\cong \Z^2$ of $\pi_1(M)$ must take values in the scalars.  As the space of (one-dimensional) characters of $\Z^2$ is path connected, this implies that $\pi$ is path connected to an irreducible representation whose restriction to $\pi_1(T)$ is trivial; to be precise, we mean that there is a continuous path between $\pi$ and a representation whose restriction to $\pi_1(T)$ is trivial, where the path is understood as in the space of unitary representations $\Gamma\to U(\C^n)$ when the latter is equipped with the topology of pointwise convergence.  Moreover, the map $\alpha:R(\Gamma)\to K^0(M)$ depends on $\pi$ only through its path component in the space of unitary representations (of a given fixed dimension) in this sense.

It follows that any vector bundle in the image of $\alpha:R(\Gamma)\to K^0(M)$ is pulled back from a (flat) vector bundle on $K^0(K)$.  The result follows readily from this (plus the well-known fact that the canonical copy of $\Z/2$ in $K^0(K)$ is indeed generated by a formal difference $[E]-[1]$ where $[E]$ is a flat line bundle).
\end{proof}

\begin{proof}[Proof of Theorem \ref{tor the}]
Let $\Gamma$ be as given.  As $T\times K$ has contractible universal cover (specifically, it is a copy of $\R^4$), it is a model for $B\Gamma$.  Using Proposition \ref{tor vs flat} there is an order two class in $y\in \widetilde{K}^0(M)$ which is not in the subgroup $\widetilde{K}^0_{\text{fl}}(M)$ of Definition \ref{flat sg}.  The group $\Gamma$ is amenable and torsion-free, so satisfies the hypotheses of Theorem \ref{final the}.  Hence there exists an element $x\in QR(\Gamma)$ such that $\alpha(x)=y$.  As $\Gamma$ is amenable and residually finite, we may moreover apply Corollary \ref{rep form} part \eqref{rep diff} to assume that $x$ is represented in the form $([\phi_{S,\epsilon}]-[\pi_{S,\epsilon}])_{(S,\epsilon)\in I_\Gamma} $ where each $\phi_{S,\epsilon}$ is an $(S,\epsilon)$-representation, and each $\pi_{S,\epsilon}$ is an honest representation.

By definition of the homomorphism $\alpha$, and as $B\Gamma$ is compact, there must be $(S_0,\epsilon_0)$ such that the natural map $K_{\epsilon}(S)\to K^0(B\Gamma)$ takes $[\phi_{S,\epsilon}]-[\pi_{S,\epsilon}]$ to $y$ for all $(S,\epsilon)\geq (S_0,\epsilon_0)$.  

Fix now $(S,\epsilon)\geq (S_0,\epsilon_0)$ as in the statement.  As $\alpha$ is an isomorphism and $y$ has order two, we must have $x+x=0$.  Using Corollary \ref{rep form}, part \eqref{zero rep} there exists $(T_0,\delta_0)\geq (S,\epsilon)$ such that for any $(T,\delta)\geq (T_0,\delta_0)$ there exists a representation $\theta$ and a unitary isomorphism $u:V_{\phi_{T,\delta}}\oplus V_\theta\to V_{\pi_{T,\delta}}\oplus V_\theta$ such that 
\begin{equation}\label{tdelt close}
\|u(\phi_{T,\delta}(s)\oplus \phi_{T,\delta}(s)\oplus \theta(s))u^*-\pi(s)\oplus \theta(s)\|<\epsilon
\end{equation}
for all $s\in S$.  

We claim that for any $(T,\delta)\geq (T_0,\delta_0)$, the map $\phi:=\phi_{T,\delta}$ has the desired properties.

Indeed, assume for contradiction that part \eqref{far part} does not hold.  Then we have that $[\phi]+ [\pi]=[\sigma]$ in $K_{\epsilon_0}(S_0)$.  Hence with $\pi':=\pi_{T,\delta}$ we have that $\alpha([\phi]-[\pi'])=\alpha([\sigma]-[\pi]-[\pi'])$, where we abuse notation by also writing $\alpha$ for the map $K_{\epsilon_0}(S_0)\to K^0(B\Gamma)$.  However, $[\sigma]$, $[\pi]$ and $[\pi']$ are actual representations, and $\pi$ takes actual representations to flat bundles.  This contradicts the choice of $y$, and that $\alpha([\phi]-[\pi'])=y$.

Now let us look at part \eqref{close part}, and let $\theta$ and $u$ be as in the discussion above, so in particular the relationship in line \eqref{tdelt close} holds.  This implies the desired conclusion: take $\sigma(s):=u^*(\pi(s)\oplus \theta(s))u$.  
\end{proof}

\begin{remark}\label{exp phi}
Using that our map $\alpha:QR(\Gamma)\to LK^0(B\Gamma)$ is a \emph{ring} isomorphism, one can explicitly write down a quasi-representation with the properties in Theorem \ref{tor the}.  We sketch the idea, although the details need a little more work that we leave to the interested reader.

Let $y\in \widetilde{K}^0(M)$ be the class that we use in the proof of Theorem \ref{tor the}.  Let $\pi_T:M\to T$ and $\pi_K:M\to K$ be the coordinate projections.  Then the class $y$ is the equal to the product $\pi_K^*([L]-[1])\pi_T^*[B]$ where: $1$ is the trivial complex line bundle over $K$, and $L$ is the (flat) complex line bundle over $K$ such that $[L]-[1]$ ($1$ denotes the trivial line bundle) generates the torsion subgroup $\Z/2$ of $K^0(K)$ (i.e.\ the first Chern class of $L$ is the generator of $H^2(K)\cong \Z/2$); and $B$ is the Bott line bundle over $T$, i.e.\ the line bundle whose first Chern class is the generator of $H^2(T)$ determined by some fixed orientation (because $\pi_K^*([L]-[1])$ has order two, the sign chosen actually makes no difference to which class we get).

For any $n\geq 1$, let now $u,v\in U_n(\C)$ be the usual Voiculescu matrices \cite{Voiculescu:1983km}, i.e.\ $u$ is the left shift and $v$ is the diagonal matrix with entries the $n^\text{th}$ roots of unity in counter-clockwise order.  Let $\psi_n$ be the quasi-representation of $\Gamma$ defined by composing the canonical quotient map $\Gamma\to \pi_1(T)\cong \Z^2$ with the map sending $(m,k)\in \Z^2$ to $u^mv^k$.   Let now $(S,\epsilon)\in I_\Gamma$ be large enough so that the map $K_\epsilon(S)\to K^0(M)$ of Lemma \ref{coass tech lem} is defined (for some fixed choice of covering coordinate data as in Definition \ref{cov dat}).  Then for any suitably large $n$, $\psi_n$ defines a class $[\psi_n]$ in $K_\epsilon(S)$.  Let moreover $1_n:\Gamma\to U_n(\C)$ denote the trivial representation.  Then well-known computations going back to Loring's thesis \cite{Loring:1985ud} show that the map 
$$
K_\epsilon(S)\to K^0(M)
$$
takes $[\psi_n]-[1_n]$ to $[B]-[1]$ for all suitably large $n$ (up to replacing $B$ by its dual bundle, which for us is not important).  Let also $\overline{\psi_n}$ denote the contragredient quasi-representation, which has the property that $[\overline{\psi_n}]-[1_n]$ maps to $[B^*]-[1]$, where $B^*$ is the dual line bundle.

On the other hand write $\pi_1(K)\cong \Z\rtimes \Z$ where $\Z$ acts on itself by the non-trivial automorphism and call the standard generators $a$ (of the acted-on copy of $\Z$) and $b$ (of the acting copy of $\Z$).  Then one can compute that the (flat) line bundle $L$ corresponds to the representation $\sigma:\pi_1(K)\to U_1(\C)$ that takes $a$ to $-1$ and $b$ to $1$.  Abusing notation, also write $\sigma:\Gamma\to U_1(\C)$ for the pullback of this representation to $\Gamma$.

From all of this discussion, it follows with a bit more work from the fact that $\alpha:QR(\Gamma)\to K^0(M)$ is a ring isomorphism that the quasi-representation 
$$
(\psi_n\otimes \sigma) \oplus \overline{\psi_n}
$$
has the properties required of the quasi-representation $\phi$ in the statement of Theorem \ref{tor the} for any $n$ suitably large.
\end{remark}


\appendix

\section{Some generalities on inverse limits}\label{inv lim app}

In this section we collect some generalities on inverse systems and inverse limits.  All of this is standard, but maybe not so well-known to operator algebraists so we thought recording some facts was worthwhile.  A classic reference is \cite{Jensen:1972wu}, while a more modern reference (with a much more sophisticated take on the ideas) is \cite[Chapters 2, 3, and 6]{Kashiwara:2006aa}; both of these go far beyond what we need.  The latter in particular includes material on the more general setting of pro-categories: while this is no doubt a more powerful machine (and we suspect it might be worthwhile developing the ideas of the paper in that context for dealing with some more subtle questions), the extra abstraction did not seem worth the effort for our relatively simple applications.

\begin{definition}
An \emph{inverse system}, or \emph{projective system}, of abelian groups (or rings), consists of a collection $(G_i)_{i\in I}$ of abelian groups (rings) indexed by a directed set $I$, and so that for each $i\geq j$ there is a homomorphism $\phi_{ji}:G_i\to G_j$ such that whenever $i\geq j\geq k$ we have $\phi_{kj}\circ \phi_{ji}=\phi_{ki}$, and such that $\phi_{ii}=\text{id}$ for all $i$.
\end{definition}

The maps $\phi_{ji}:G_i\to G_j$ are called the \emph{structure maps} of the inverse system.  Throughout this appendix, we adopt the convention that all unlabeled arrows in diagrams are structure maps for the appropriate inverse system.

\begin{definition}\label{inv lim def}
Let $(G_i)_{i\in I}$ be an inverse system of topological\footnote{In our applications, the topology will always be the discrete topology.} abelian groups.  Its \emph{inverse limit}, or \emph{projective limit}, is defined to be 
$$
\lim_{\leftarrow}G_i:=\Bigg\{(g_i)\in \prod_{i\in I} G_i \mid \phi_{ij}(g_j)=g_i \text{ for all }j\geq i\Bigg\},
$$
equipped with the group operations and the topology that it inherits from the group operations and product topology on $\prod_{i\in I} G_i$.  If moreover each $G_i$ is a ring, we equip the inverse limit also with the induced ring structure.  We also use the notation ``${\displaystyle \lim_{\leftarrow,I}}G_i$'' if we need to be explicit about the directed set involved.
\end{definition}

\begin{definition}
Let $I$ be a directed set.  A subset $J$ of $I$ is \emph{cofinal} if for $i\in I$ there exists $j\in J$ with $j\geq i$. 
\end{definition}

The following lemma consists of a direct check that we leave to the reader.

\begin{lemma}\label{cof lem}
Let $(G_i)_{i\in I}$ be an inverse system of abelian groups, and let $J\subseteq I$ be a cofinal subset.  Then the natural restriction map $\prod_IG_i\to \prod_JG_i$ induces an isomorphism
$$
\lim_{\leftarrow,I}G_i\cong \lim_{\leftarrow,J} G_i.\eqno\qed
$$
\end{lemma}

The following result is the key fact we need.

\begin{lemma}\label{triangle isos}
Let $I$ and $J$ be directed sets that contain cofinal sequences.  Let $(G_i)_{i\in I}$ and $(H_j)_{j\in J}$ be inverse systems of abelian groups indexed by $I$ and $J$.  Assume that there exist non-decreasing maps $\phi:I\to J$ and $\psi:J\to I$ with cofinal image such that for each $i\in I$ there is a homomorphism $\phi_i:G_i\to H_{\phi(i)}$ and for each $j$ a homomorphism $\psi_j:H_j\to G_{\psi(j)}$ with the following properties: for all $j\in J$ and all $i\in I$ such that $i\geq \psi(j)$ and $\phi(i)\geq j$, the diagram 
$$
\xymatrix{ G_i \ar[rrr] \ar[dr]^-{\phi_i} & & & G_{\psi(j)} \\
& H_{\phi(i)} \ar[r] & H_j \ar[ur]^-{\psi_{j}} & }
$$
commutes, and for all $i\in I$ and all $j\in J$ such that $j\geq \phi(i)$ and $\psi(j)\geq i$, the diagram
$$
\xymatrix{ & G_{\psi(j)} \ar[r] & G_i \ar[dr]^-{\phi_i} & \\
H_j \ar[ur]^-{\psi_j} \ar[rrr] & & & H_{\phi(i)} }.
$$
commutes.

Then the maps $\phi_i$ and $\psi_j$ induce homeomorphic isomorphisms
$$
\lim_{\leftarrow}G_i\cong \lim_{\leftarrow}H_j.
$$
\end{lemma}

\begin{proof}
Using the assumptions we may construct cofinal sequences $(i_n)_{n=1}^\infty$ of $I$ and $(j_n)_{n=1}^\infty$ of $J$ such that if we write $G_n:=G_{i_n}$ and $H_n:=H_{j_n}$ then there are maps $\phi_n:G_n\to H_n$ and $\psi_n:H_n\to G_{n-1}$ (built as compositions of the maps $\phi_i$ and $\psi_j$ in the statement, and the structure maps) such that the diagrams below commute
\begin{equation}\label{tris}
\xymatrix{ G_{n} \ar[dr]_-{\phi_{n}} \ar[rr] & & G_{n-1} \\ & H_{n} \ar[ur]_-{\psi_{n}} & }\quad \xymatrix{ & G_{n-1} \ar[dr]^-{\phi_{n-1}} & \\ H_n \ar[ur]^-{\psi_n} \ar[rr] & & H_{n-1} }. 
\end{equation}
We now consider the sequence of maps:
$$
\prod_I G_i\to \prod_{n=1}^\infty G_n{\rightarrow \atop \leftarrow} \prod_{n=1}^\infty H_n \leftarrow \prod_J H_j.
$$
The first and third map are restriction maps as in Lemma \ref{cof lem}, and induce isomorphisms on inverse limits by the result of that lemma.  The central arrows are induced by $(g_n)_{n=1}^\infty\mapsto (\phi_n(g_n))_{n=1}^\infty$ and $(h_n)_{n=1}^\infty\mapsto (\psi_{n+1}(h_{n+1}))_{n=1}^\infty$.  Using the condition that the triangles in line \eqref{tris} commute, it is not too difficult to see that these central maps induce mutually inverse isomorphisms on inverse limits, so we are done.
\end{proof}

\bibliography{Generalbib}

\end{document}